\documentclass[10pt,reqno]{amsart}
\usepackage{syntonly}
\usepackage{url,hyperref}
\usepackage{epsfig,amssymb,amsmath,version}
\usepackage{amssymb,version,graphicx,subfigure,fancybox,mathrsfs,pifont,booktabs,wrapfig}
\usepackage{color}
\usepackage{amsmath,mathabx}
\usepackage{epstopdf}
\usepackage{float}
\usepackage{amsfonts}
\usepackage{bm}
\usepackage{mathrsfs}
\allowdisplaybreaks[0]
\usepackage[titletoc]{appendix}

\usepackage{booktabs}
\usepackage{multirow}
\usepackage{colortbl}
\definecolor{tabcolor}{rgb}{.105,.410,.113}

\allowdisplaybreaks[4]

\catcode`\@=11 \theoremstyle{plain}
\@addtoreset{equation}{section}   
\renewcommand{\theequation}{\arabic{section}.\arabic{equation}}

\@addtoreset{figure}{section}
\renewcommand\thefigure{\thesection.\@arabic\c@figure}

\@addtoreset{table}{section}
\renewcommand\thetable{\thesection.\@arabic\c@table}
\newtheorem{thm}{\bf Theorem}
\renewcommand{\thethm}{\arabic{section}.\arabic{thm}}

\newenvironment{theorem}{\begin{thm}} {\end{thm}}
\newtheorem{cor}{\bf Corollary}

\newenvironment{corollary}{\begin{cor}} {\end{cor}}
\newtheorem{lem}{\bf Lemma}
\renewcommand{\thelem}{\arabic{section}.\arabic{lem}}
\newenvironment{lemma}{\begin{lem}}{\end{lem}}
\theoremstyle{remark}
\theoremstyle{definition} 
\newtheorem{definition}{Definition}[section] 
\newtheorem{remark}{Remark}[section]
\newtheorem{example}{Example}

\newtheorem{pro}{\bf Proposition} 
\renewcommand{\thepro}{\arabic{section}.\arabic{pro}} 
\newenvironment{proposition}{\begin{pro}} {\end{pro}}

\newcommand{\refe}[1]{{\rm (\ref{#1})}}

\newcommand \dsum {\displaystyle\sum}
\newcommand \dint {\displaystyle\int}

\def\CP{{\mathcal P}}

\def \bx{\bs x}
\def \by{\bs y}
\def \bxi{\bs \xi}

\def\i{\mathrm{i}}
\def\e{\mathrm{e}}
\def\d{\mathrm{d}}
\def\C{{\widehat{C}}}
\def\CC{ {\mathbb{C}} }

\def\hp{\hat{\partial}}
\def\H{\mathcal{B}}
\def\RR{\mathbb{R}}
\def\ZZ{\mathbb{Z}}
\def\sph{\mathbb{S}}

\def\CP{{\mathcal P}}
\def\CH{{\mathcal H}}

\def\s{\sigma}
\def\NN{\mathbb{N}}

\def\BB^d{\mathbb{B}^d}

\newcommand{\bs}[1]{\boldsymbol{#1}}
\begin{document}
\graphicspath{{./figs/}}

\title[Generalised Hermite functions]
{Generalised Hermite  Spectral Methods for PDEs involving integral fractional Laplacian and Schr\"{o}dinger operators}
\author[C. Sheng,\, S. Ma,\, H. Li,\, L. Wang  \& \ L. Jia]{Changtao Sheng${}^{1}$,  \; Suna Ma${}^{2},$  \; Huiyuan Li${}^{3}$,  \; Li-Lian Wang${}^{1}$ \;and\; Lueling Jia${}^{4}$}

\thanks{${}^{1}$Division of Mathematical Sciences, School of Physical and Mathematical Sciences, Nanyang Technological University,
		637371, Singapore. The research of the  authors is partially supported by Singapore MOE AcRF Tier 2 Grants:  MOE2018-T2-1-059 and MOE2017-T2-2-144. Emails: ctsheng@ntu.edu.sg (C. Sheng) and  lilian@ntu.edu.sg (L. Wang).\\
		\indent ${}^{2}$School of Mathematical Sciences, Peking University, Beijing 100871, China.  Email: masuna@csrc.ac.cn (S. Ma).\\
		\indent  ${}^{3}$State Key Laboratory of Computer Science/Laboratory of Parallel Computing, Institute of Software,
Chinese Academy of Sciences, Beijing 100190, China. The work of this author is partially supported by the National Natural Science Foundation of China (No. 11871455 and 11971016).  Email: huiyuan@iscas.ac.cn (H. Li).\\
\indent  ${}^{4}$Beijing Computational Science Research Center, Beijing, 100193, P.R. China.  The research of this author is supported in part by the National Natural Science Foundation of China  (NSFC 11871092 and NSAF U1930402). Email: lljia@csrc.ac.cn (L. Jia).
         }


\begin{abstract} In this paper, we introduce two new families of generalised Hermite polynomials/functions (GHPs/GHFs) in arbitrary dimensions, and develop  efficient and accurate generalised Hermite spectral algorithms for PDEs with integral fractional Laplacian (IFL)
and/or  Schr\"{o}dinger operators in $\mathbb R^d.$  As a generalisation of the G. Szeg\"{o}'s  family in 1D (1939),  the first family of  
GHPs (resp. GHFs) are orthogonal with respect to $|\bx|^{2\mu} \e^{-|\bx|^2}$ (resp. $|\bx |^{2\mu}$) in $\mathbb R^d$. 
We further define adjoint generalised Hermite functions (A-GHFs) which have an interwoven connection with the corresponding GHFs
through  the Fourier transform,  and which are orthogonal with respect to the inner product $[u,v]_{H^s(\mathbb R^d)}=((-\Delta)^{s/ 2}u, (-\Delta)^{s/2} v )_{\mathbb R^d}$ associated with the IFL of order $s>0$.  
Thus,  the spectral-Galerkin method using A-GHFs as basis functions  leads to a diagonal stiffness matrix for  the IFL (which is known to be notoriously difficult and expensive to discretise). The new basis also finds efficient and accurate in solving  PDEs with the 
fractional  Schr\"{o}dinger operator: $(-\Delta)^s +|\bs x|^{2\mu}$  with $s\in (0,1]$ and $\mu>-1/2.$ 
Following the same spirit, we construct the second family of  GHFs, dubbed as M\"untz-type generalised Hermite functions (M-GHFs), which are orthogonal with respect to an inner product associated with the underlying Schr\"{o}dinger operator, and  are tailored to the singularity of the solution at the origin.
We demonstrate that the M\"untz-type GHF spectral method leads to sparse matrices and spectrally accurate to some  Schr\"{o}dinger eigenvalue problems. 
\end{abstract}

\maketitle

\vspace*{-15pt}

\section{Introduction}

In the seminal monograph  \cite[P. 371]{GS1939} (1939),  Szeg\"o first introduced  a generalisation of the Hermite polynomials (denoted by $H^{(\mu)}_n(x), \, \mu>-1/2, \,  x\in \mathbb R:=(-\infty,\infty)$ and dubbed as generalised Hermite polynomials (GHPs)), through an explicit second-order differential equation in an exercise problem. 
 The GHPs defined therein  are orthogonal with respect to  the weight function $|x|^{2\mu}\e^{-x^2}$.   Chihara perhaps was among the first who systematically  studied  the properties of  the GHPs, and  the associated generalised Hermite functions (GHFs):  $\widehat H^{(\mu)}_n(x):=\e^{-x^2/2} H^{(\mu)}_n(x)$ (orthogonal with respect to  the weight function  $|x|^{2\mu}$),   in his PhD thesis  \cite[entitled as ``Generalised Hermite Polynomials"]{chihara1955} (1955).  Later, some  standard properties were collected in his book   \cite{chihara2011} (1978). 
Whereas the usual Hermite polynomials/functions  are well-studied especially in  spectral approximations,
 there have been very limited works
on this generalised family   (see, e.g.,  \cite{SCF1964,rose1994,Rosler1998,twoclasses} for  the properties or further generalisations). Indeed, to the best of our knowledge,  the generalised Hermite spectral methods in both theory and applications are still under-explored, and worthy of deep investigation.

 The main purpose  of this paper is to introduce two new families of GHPs/GHFs in arbitrary spatial dimensions, and explore their applications in solutions of PDEs involving the integral fractional Laplacian and/or  Schr\"odinger operators.  
 
Firstly, we construct  the $d$-dimensional GHPs $\{H_{k,\ell}^{\mu,n}(\bx)\}$ (cf.\! \eqref{MHP}) and GHFs $\{\widehat H_{k,\ell}^{\mu,n}(\bx)\}$ (cf.\! \eqref{LagFun01}),  which are  orthogonal with respect to the weight functions 
  $|\bs x|^{2\mu}\e^{-|\bs x|^2}$ and $|\bs x|^{2\mu}$ in $\mathbb R^d$ with  $\mu>-\frac12, $ respectively.  In one dimension,  they reduce to Szeg\"o's GHPs/GHFs (up to a constant multiple).  More importantly,  we  introduce  for the first time a family of  adjoint
  generalised Hermite functions (A-GHFs)  $\{\widecheck{H}^{\mu,n}_{k,\ell}(\bx)\}$ (cf.\! \eqref{adher}) with some appealing properties. For example,   this adjoint pair is closely interwoven through the Fourier transform   
  \begin{equation}\label{uHs0001}
  \mathscr{F}[\widehat{H}_{k,\ell}^{\mu,n}](\bxi)=\i^{n+2k}\widecheck{H}_{k,\ell}^{\mu,n}(\bxi),\quad  \mathscr{F}[\widecheck{H}_{k,\ell}^{\mu,n}](\bxi)=(-\i)^{n+2k}\widehat{H}_{k,\ell}^{\mu,n}(\bxi).
  \end{equation}
  More notably, the A-GHFs are orthogonal with respect to the inner product that induces the 
  so-called Gagliardo semi-norm of the fractional Sobolev space $H^s(\mathbb R^d)$ for $s\in (0,1],$ that is, 
\begin{equation}\label{uHs000}
[\widecheck{H}_{k, \ell}^{s, n}, \widecheck{H}_{j, \iota}^{s, m}]_{H^s(\mathbb R^d)}=\big((-\Delta)^{\frac{s}{2}}\widecheck{H}_{k, \ell}^{s, n},(-\Delta)^{\frac{s}{2}}\widecheck{H}_{j, \iota}^{s, m}\big)_{\mathbb R^d}=\delta_{jk} \delta_{mn} \delta_{\ell\iota},
\end{equation}
where $(-\Delta)^s$ is the integral fractional Laplacian operator (cf.\! \eqref{eq:LapDF}-\eqref{fracLap-defn}). As an immediate consequence, the use of  A-GHFs as basis functions in the spectral-Galerkin approximation of the integral fractional Laplacian leads to a diagonal  stiffness matrix.  Indeed, it has been a nightmare for computing this matrix in  a usual tensorial Hermite 
spectral method when $d= 3$ (cf.\! \cite{mao2017}).  On the other hand, this new basis offers an efficient  algorithm for solving PDEs with the fractional Schr\"{o}dinger operator: $(-\Delta)^s+V(\bs x)$ with $V(\bs x)=|\bs x|^{2\mu}$ or more general 
$V(\bs x)=|\bs x|^{2\mu}W(\bs x)$ with $s\in (0,1]$ and $\mu>-1/2$ (where $W$ is smooth).  In light of the orthogonality  \eqref{uHs000},   the stiffness matrix under the Galerkin framework using the basis $\{\widecheck{H}_{k, \ell}^{s, n}\}$    becomes  diagonal, 
while the singular potential $|\bs x|^{2\mu}$ can be treated as the (orthogonal) weight function  by using the connection relation 
between $\{\widecheck{H}_{k, \ell}^{s, n}\}$  and $\{\widehat{H}_{j, l}^{\mu, n}\}$ (cf.\! \eqref{coefrd1} and \eqref{adher}). We remark that there is a growing interest in the numerics of the fractional Schr\"{o}dinger problems (see, e.g., 
\cite{Bao2019,Bao2020} and the references therein). 
 
It is noteworthy that the  3D GHPs with $\mu=0$ and an appropriate scaling reduce to the  Burnett polynomials  \cite{burnett1936} (1936), which are mutually orthogonal with respect to the Maxwellian  
 $\mathcal{M}(\bx)={(2 \pi)^{-3/2}} \e^{-{|\bx|^{2}}/{2}},$ and a useful  basis  in solving kinetic equations (cf.\! \cite{Cai2018,Hu2020} and the references therein).  Remarkably, we can show  that the GHFs with $\mu=0$ are eigenfunctions of the Schr\"{o}dinger operator with the square potential (cf.\! \eqref{ghrdiff2}): 
 \begin{equation}\label{ghrdiff200} 
 \big(\!-\Delta+|\bx|^{2}\big) \widehat{H}_{k, \ell}^{0,n}(\bx)=(4 k+2 n+d) \widehat{H}_{k, \ell}^{0,n}(\bx). 
\end{equation}
In fact, such a notion in 2D has been  explored  in \cite{BLS2009}  for computing the ground states and dynamics of the Bose-Einstein condensation. 

\smallskip 
It is of  fundamental and  practical interest to search for the explicit eigen-functions  for 
the Schr\"{o}dinger operator with a more general potential or some variance, which serves as  the second purpose of this paper.  The main finding in Theorem \ref{thmorth}  is  that for $\theta>\max(1-d/2,0),$ there exists a family of  M\"untz-type GHFs $\{\widehat {\mathcal H}^{\theta,n}_{k,\ell}\}$ (cf.\! \eqref{ghfrdx}) satisfying 
\begin{equation}\label{orthsch00}
\big(\!-\Delta+\theta^2 |\bx|^{4\theta-2}\big)\widehat {\mathcal H}^{\theta,n}_{k,\ell}(\bx)=  2\theta^2 \big((n+d/2-1)/\theta+2k+1\big)\, |\bx|^{2\theta-2}\widehat {\mathcal H}^{\theta,n}_{k,\ell}(\bx).
\end{equation}
In particular, for $\theta=1/2,$ we find 
 \begin{equation}\label{orthsch0011}
\Big(\!\!-\Delta-\frac{n+k+(d-1)/2} { |\bx|} \Big)\widehat {\mathcal H}^{\frac12,n}_{k,\ell}(\bx)=  -\frac 1 4  \widehat {\mathcal H}^{\frac12,n}_{k,\ell}(\bx).
\end{equation}
 With a proper scaling, this gives  the eigen-pairs of the Schr\"odinger operator with Coulomb potential: 
$-\frac12\Delta- \frac{|Z|}{|\bx|}$, where $Z$ is a nonzero constant (cf.\! Corollary \ref{corScheig}).
By construction, this new family of functions in the radial direction turns out to be  some special M\"{u}ntz functions, so it is dubbed as  M\"untz-type for distinction.  
We remark that a M\"{u}ntz polynomial  $\sum_{k=0}^{n} a_{k} r^{\lambda_{k}}$ is generated by 
 a M\"{u}ntz sequence: $\lambda_{0}<\lambda_{1}<\lambda_{2}<\cdots<\lambda_n$ (cf.\!  \cite{Muntz} (1914)), and
 the set of M\"{u}ntz polynomials with  $\lambda_{0}=0,$ and real coefficients $\{a_k\}$ are dense in the space of continuous functions if and only if $\sum_{k=0}^{\infty} \lambda_{k}^{-1}=\infty$ (cf. \cite{Szasz}). Such a tool finds very effective in approximating singular solutions (see, e.g., \cite{ShenWang2011,Hou2017}).  Indeed,  we shall demonstrate in  Section \ref{sect6Schro} that the  M\"untz-type GHF spectral-Galerkin approach is the method of choice of the  Schr\"odinger eigenvalue problems with the fractional power potential in terms of both the efficiency and accuracy.   We shall see that spectral accuracy can be achieved in fitting the singular eigenfunctions.

  In Table \ref{severalHermite},  we provide a roadmap of two types of generalisations and some of their 
  properties that are essential for developing efficient spectral algorithms for PDEs with integral fractional Laplacian  in  Section \ref{Section3}  and the Schr\"{o}dinger eigenvalue problems in 
  Section \ref{sect6Schro}. 
 \begin{table}[!th] 
  \renewcommand*{\arraystretch}{1.5}
 \caption{\small Two families of GHPs/GHFs and their essential properties} \label{severalHermite}
 \vspace*{-6pt} {\small 
\begin{tabular}{|c|c|c|}\hline
   & Type & Property 
\\ \hline  
 \multirow{6}{5.5em}{Generalised Hermite polynomials \& functions}&$d$-D GHP: $H_{k,\ell}^{\mu,n}(\bx)$ in \eqref{MHP}  & Orthogonal w.r.t. $|\bx|^{2\mu}\e^{-|\bs x|^2};$
 \\ 
  &1D GHP:\, $H_{n}^{(\mu)}(x)$ in  \cite{GS1939}  & Burnett polynomials  \cite{burnett1936}, if $\mu=0$ 
  \\ \cline{2-3}
 & $d$-D GHF: $\widehat{H}^{\mu,n}_{k,\ell}(\bx)$ in \eqref{LagFun01}  & Orthogonal w.r.t. $|\bx|^{2\mu};$
 \\ 
 &1D GHF: $\widehat{H}^{(\mu)}_{n}(x)$ in \eqref{GHermfunc}  &  Eigenfunctions of $-\Delta+|\bs x|^2$, if $\mu=0$ 
  \\ \cline{2-3}
 &$d$-D A-GHF: $\widecheck{H}^{\mu,n}_{k,\ell}(\bx)$ in \eqref{adher} & Orthogonal w.r.t.  $((-\Delta)^{\frac{\mu}2}\cdot,(-\Delta)^{\frac{\mu}2}\cdot)_{\RR^d};$
 \\
  &1D A-GHF: $\widecheck{H}^{(\mu)}_{n}(x)$ in \eqref{adher1d1} &  Diagonal stiffness matrix for $(-\Delta)^\mu,$ if $\mu>0$ 
   \\  \hline
 \multirow{4}{5.5em}{M\"untz-type generalised Hermite functions}&$d$-D M-GHF: $\widehat{{\mathcal H}}^{\theta,n}_{k,\ell}(\bx)$ in \eqref{ghfrdx} & Orthogonal w.r.t.  $(  \nabla\, \cdot\,,  \nabla\,\cdot )_{\RR^d}
+\theta^2  (|\bx|^{4\theta-2}\cdot, \cdot )_{ \RR^d}$
\\ \cline{2-3}
 &$\widehat{{\mathcal H}}_{k,\ell}^{\frac 1 2,n}(\bx)$ in Subsection \ref{coulomb}  & Eigenfunctions (with a scaling) of
$ -\frac 12 \Delta  +\frac{|Z|}{|\bx |}$  
\\ \cline{2-3}
  & \multirow{2}{13em}{  $\widehat{{\mathcal H}}_{k,\ell}^{\frac1{\mu+1},n}(\bx)$ in Subsection \ref{coulomb2}}  &
  Optimal basis for the  Schr\"{o}dinger operator:  \\
    &  &    $-\frac 12 \Delta  +|Z| |\bx |^{\frac{2\nu-2\mu}{\mu+1}}$ for $\mu,\nu$ in \eqref{munucond} 
     \\   \hline
\end{tabular}
}
 \end{table}
 \smallskip

\section{Generalized Hermite polynomials/functions in multiple dimensions }\label{Section3}
\setcounter{equation}{0} \setcounter{lem}{0} \setcounter{thm}{0} \setcounter{cor}{0}\setcounter{remark}{0} 
In this section, we first make necessary preparations by introducing some notation and properties of the spherical harmonic functions.   We then  define the multi-dimensional GHPs and GHFs, and construct the adjoint GHFs. We present various appealing properties of these new families of basis functions, and elaborate on their differences and connections with the most relevant Hermite-related polynomials/functions in literature.   

\subsection{Preliminaries}\label{pre}


Let $\mathbb R=(-\infty, \infty),$ $\mathbb N=\{1,2,\cdots\},$ and ${\mathbb N}_0:=\{0\} \cup {\mathbb N}.$
For $d\in \mathbb N,$ we denote by $\mathbb{R}^d$ the  $d$-dimensional Euclidean space equipped with 
the inner product and norm  $\langle\bx,\bs y\rangle:= \sum^d_{i=1} x_iy_i$, and  $r=|\bx| := \sqrt{ \langle\bx, \bx\rangle}$, respectively,
for any $\bx,\bs y\in \mathbb{R}^d.$    Denote the unit vector along any nonzero vector $\bs x$ by $\hat \bx =\bx/|\bx|.$ 

We next introduce the $d$-dimensional spherical harmonics, upon which  we define the $d$-dimensional generalised Hermite polynomials/functions. Here, we follow the setting in the book \cite{Dai2013}. 
Let $\CP_n^d$ be the space of all  real $d$-dimensional homogeneous polynomials of degree
$n$ as follows
\begin{equation}\label{Pnnn}
\mathcal{P}_{n}^{d}={\rm span}\big\{\bs{x}^{\bs{k}}=x_{1}^{k_{1}} x_{2}^{k_{2}} \ldots x_{d}^{k_{d}}:  k_{1}+k_{2}+\cdots+k_{d}=n\big\}.
\end{equation}
As an important subspace of $\mathcal{P}_{n}^{d},$  the  space of all real harmonic polynomials of degree $n$ is defined as 
\begin{equation}\label{Hnnn}
\mathcal{H}_n^d:=\big\{P\in \CP_n^d : \Delta P(\bx)=0 \big\}.
\end{equation}
It is known that  the dimensionality 
\begin{equation}\label{cpCp}
    \dim (\CP_n^d)  = \binom{n+d-1}{n},  \quad  \dim (\mathcal{H}_n^d) = \binom{n+d-1}{n} - \binom{n+d-3}{n-2}:=a_n^d,
\end{equation}
where it is understood that for $n=0,1,$ the value of the second binomial coefficient is zero  
 (cf. \cite[(1.1.5)]{Dai2013}). In fact,  for $d=1,$ all  harmonic polynomials are spanned by  $\{1, x\}.$

Recall that the $d$-dimensional spherical coordinates read
\begin{equation}\label{d_sphere}
\begin{split}
&x_{1}=r\cos\theta_{1};\; x_{2}=r\sin\theta_{1}\cos\theta_{2};\;\cdots\cdots; \;x_{d-1}=r\sin\theta_{1}\cdots\sin\theta_{d-2}\cos\theta_{d-1}; 
\\ &x_{d}=r\sin\theta_{1}\cdots\sin\theta_{d-2}\sin\theta_{d-1}, \;\;\; \theta_1, \cdots, \theta_{d-2}\in [0,\pi],  \;\;  \theta_{d-1}\in[0,2\pi],
\end{split}
\end{equation}
with  the spherical volume element 
\begin{equation}\label{coord1} 
\begin{split} \d \bx=r^{d-1} \sin ^{d-2}\left(\theta_{1}\right) \sin ^{d-3}\left(\theta_{2}\right) \cdots \sin \left(\theta_{d-2}\right) \d r \,\d \theta_{1}\, \d \theta_{2} \cdots \d \theta_{d-1}:=r^{d-1} \d r\, \d\sigma (\hat\bx). 
\end{split}
\end{equation}
In spherical  coordinates, the $d$-dimensional Laplacian  takes the form
\begin{equation} \label{eq:Delta}
   \Delta = \partial^2_r+ \frac{d-1}{r} \partial_r + \frac{1}{r^2} \Delta_{\mathbb{S}^{d-1}},
\end{equation}
where $\Delta_{\mathbb{S}^{d-1}}$ is the Laplace-Beltrami operator on the   unit sphere $\mathbb{S}^{d-1}:=\{ \bx \in\RR^d : |\bx|=1\}.$ Define the inner product of $L^2(\mathbb S^{d-1})$ as
 $$
        \langle f, g \rangle_{\,\mathbb{S}^{d-1}}: = \int_{\mathbb{S}^{d-1}} f(\hat \bx ) g(\hat \bx)\, \d\s(\hat \bx).
$$

The $d$-dimensional spherical harmonics are the restrictions of harmonic polynomials in $\mathcal{H}_n^d$ to  $\mathbb{S}^{d-1}$,   denoted by  $\mathcal{H}_n^d\big|_{\mathbb{S}^{d-1}}$. It is important to remark  the correspondence between a harmonic polynomial and the related  spherical harmonic   function (cf. \cite[Ch.\! 1]{Dai2013}):  for any $Y(\bx)\in  \mathcal{H}_n^d,$ 
  \begin{equation}\label{reponsd}
  Y(\bx)=|\bx|^nY( \bx/|\bx|) =r^nY({\hat \bx}), 
  \end{equation}
with   $Y(\hat \bx) \in \mathcal{H}_n^d\big|_{\mathbb{S}^{d-1}}$.  It is noteworthy that $Y(\bx) $ is a homogeneous polynomial in $\mathbb R^d$, while 
  $Y(\hat \bx)$ is a non-polynomial function on the unit sphere.  
  For $n\in \NN_0$, let $\{Y_\ell^n: 1 \le \ell \le a_n^d\}$ be the  real (orthogonal)  spherical harmonic basis of $\CH_n^d|_{\mathbb{S}^{d-1}}$, and note that the spherical harmonics of different degree are mutually  orthogonal   (cf. \cite[Thm. 1.1.2]{Dai2013}), i.e.,  $\mathcal{H}_n^d|_{\mathbb{S}^{d-1}}\perp  \mathcal{H}_m^d|_{\mathbb{S}^{d-1}}$  for $m\not=n.$
Thus, we have 
\begin{align}
\label{Yorth}
 \langle Y_\ell^n, Y_\iota^m\rangle_{\,\mathbb{S}^{d-1}}= \int_{\mathbb{S}^{d-1}} Y_\ell^n(\hat \bx ) Y_\iota^m(\hat \bx)\, \d\s(\hat \bx)=\delta_{nm} \delta_{\ell\iota},
 \quad (\ell,n), (\iota, m) \in \Upsilon_\infty^d,
 \end{align}
  where we introduce two-related the index sets 
  \begin{equation}\label{UpsD}
\begin{split}
& \Upsilon_\infty^d=\big\{(\ell,n) : 1\le \ell\le a_n^d,\;\; 0\le n<\infty, \;\; \ell, n\in \mathbb N_0\big\}, \\
& \Upsilon_N^d=\big\{(\ell,n) :   1\le \ell \le a_n^d,\;\; 0\le n\le N, \;\; \ell,  n\in \mathbb N_0 \big\}.
\end{split}
\end{equation}
Remarkably, the spherical harmonic basis functions  are  eigenfunctions of the Laplace-Beltrami problem:
\begin{equation} \label{eq:LaplaceBeltrami}
        \Delta_{\mathbb S^{d-1}}  Y_\ell^n (\hat \bx) = - n(n+d-2) Y_\ell^n(\hat \bx). 
\end{equation}
  
The second building block of the GHPs/GHFs is the generalized Laguerre polynomials, denoted by $L_k^{(\alpha)}(z)$  for $z\in (0, \infty)$ and $\alpha>-1$. They are orthogonal with respect to the weight function $z^{\alpha}\e^{-z}$  (cf. 
 Szeg\"{o} \cite{GS1939}):
\begin{align}
\int_0^{\infty}L_k^{(\alpha)}(z)L_j^{(\alpha)}(z)\,z^{\alpha}\,\e^{-z}\,\d z=\frac{\Gamma(k+\alpha+1)}{k!}\delta_{kj}, \quad k,j\in\NN_0.\label{Lagpoly}
\end{align}
We  refer to \cite{GS1939} and \cite[Ch.\! 7]{ShenTao2011} for the properties of the generalised Laguerre polynomials.

\subsection{Generalized Hermite polynomials/functions in $\RR^d$} \label{sec2.2}
We define the $d$-dimensional GHPs and GHFs as follows. 
\begin{definition}\label{ddGHPF}  {\em For  $\mu>-\frac 1 2,$ $k\in \mathbb N_0$ and $ (\ell,n)\in \Upsilon_\infty^d,$
we define the $d$-dimensional generalised Hermite  polynomials  as 
\begin{equation}\label{MHP}  
\begin{split}
H_{k,\ell}^{\mu,n}(\bx) & :=H_{k,\ell}^{\mu,n}(\bx;d)=L_k^{(n+\frac{d-2}{2}+\mu)}(|\bx|^2)Y_{\ell}^n(\bx)\\
 & =r^nL_k^{(n+\frac{d-2}{2}+\mu)}(r^2)Y_{\ell}^n(\hat{\bx}),\quad \bx =r \, \hat \bx, 
\end{split}
\end{equation}
and   the $d$-dimensional generalised Hermite  functions  as 
\begin{equation}
\label{LagFun01}
 \widehat{H}_{k,\ell}^{\mu,n}(\bx)=\sqrt{1/\gamma^{\mu, d}_{k,n}}\, \e^{-\frac{|\bx|^2}{2}}H_{k,\ell}^{\mu,n}(\bx), 
 \quad {\rm where}\quad \gamma^{\mu, d}_{k,n}:=\frac{\Gamma(k+n+\frac{d}{2}+\mu)}{2\, k!}.
\end{equation}}
\end{definition}

\begin{remark}\label{absisA}{\em As we shall see later,   the one-dimensional GHPs  {\rm(}up to a constant multiple{\rm)}  coincide with the one-dimensional generalisation first introduced in Szeg\"{o} \cite[P. 371]{GS1939} {\rm(1939)}, from which we name the above new families. 
Indeed, they include several special types of multivariate Hermite polynomials with many applications in both theory and numerics. For example, 
the three-dimensional GHPs with $\mu=0$ and a proper scaling lead to the Burnett polynomials \cite{burnett1936} {\rm(1936)} which
have  rich applications in  kinetic theory  {\rm(}see \cite{Cai2018} and the references therein{\rm)}.  
 The notion of  constructing  special Laguerre-Fourier  basis functions {\rm(}relevant to the two-dimensional GHPs with  $\mu=0${\rm)}    for  computing the ground states and dynamics of Bose-Einstein condensation \cite{Pitaevskii 2003book}  was found 
 effective  in e.g.,  \cite{BLS2009}.  Very recently,  the PhD dissertation   \cite{yurova2020} discussed  the extension of the tensorial {\rm(}usual{\rm)} Hermite polynomials to the generalised anisotropic  Hermite functions of the form
\begin{equation}\label{orthsAH}
H_{\ell}^{G, E, t}(\bs x)=\frac{t^{|\ell| / 2}}{\sqrt{2|\ell| \ell !}} H_{\ell}\left(G^{T} {\bs x}\right) \exp \left(-{\bs x}^{T} E^{T} E {\bs x}\right),
\end{equation}
where $E, G \in \mathbb{R}^{d \times d}$ are arbitrary invertible matrices, $t>0$ is a parameter 
and $H_{\ell}(\bs x)=H_{\ell_1}(x_1)\cdots H_{\ell_d}(x_d)$ are tensor
product of 1D Hermite polynomials.  Interesting applications in quantum dynamics  \cite{Lubich2008book} were discussed therein. 
\qed }
\end{remark}


Before we consider the applications of the GHPs and GHFs, we first present some of their appealing properties. 
By construction, they enjoy the following important  orthogonality. 
\begin{thm}\label{thm2.1}
For $\mu>-\frac 1 2, $  $k,j\in \mathbb N_0$ and $ (\ell,n), (\iota, m) \in \Upsilon_\infty^d,$ 
 the GHPs are mutually orthogonal  with respect to the weight function $|\bx|^{2\mu}\e^{-|\bx|^2}$, namely, 
\begin{equation}
\label{OMH}
\begin{split}
 \int_{\RR^d} H_{k,\ell}^{\mu,n}(\bx)H_{j,\iota}^{\mu,m}(\bx)\,|\bx|^{2\mu}\,\e^{-|\bx|^2}\,\d \bx=\gamma^{\mu, d}_{k,n}\, \delta_{mn}\delta_{kj}\delta_{\ell\iota}, 
\end{split}
\end{equation}
and the GHFs are orthonormal, viz., 
\begin{equation} \label{HerfunOrthA0}
 \int_{\RR^d} \widehat{H}_{k,\ell}^{\mu,n}(\bx) \widehat{H}_{j,\iota}^{\mu,m}(\bx)\, |\bx |^{2\mu}\,\d \bx
 =\delta_{mn}\delta_{k j}\delta_{\ell \iota}.
\end{equation}
\end{thm}
\begin{proof} The orthogonality \eqref{HerfunOrthA0}  is  a direct consequence of  \eqref{LagFun01} and
\eqref{OMH}, so we only need to show  \eqref{OMH}. 
In view of the definition \eqref{MHP}, we use   the spherical coordinates transformation \eqref{d_sphere}-\eqref{coord1}, and 
find from the orthogonality \eqref{Yorth} and  \eqref{Lagpoly} that 
\begin{eqnarray}
&& \int_{\RR^d} H_{k,\ell}^{\mu,n}(\bx)H_{j,\iota}^{\mu,m}(\bx)\, |\bx|^{2\mu} \, \e^{-|\bx|^2} \, \d \bx \nonumber
\\&&=\int_{0}^{\infty}  L_k^{(n+\frac{d-2}{2}+\mu)}(r^2) \, L_j^{(m+\frac{d-2}{2}+\mu)}(r^2)  \, r^{2\mu+2n+d-1} \e^{-r^2}\d r \int_{\sph^{d-1}} Y^n_{\ell}(\hat \bx) Y^m_{\iota}(\hat \bx) \,\d \sigma(\hat \bx) \nonumber
\\&&=\delta_{mn}\delta_{\ell\iota}\int_{0}^{\infty}  L_k^{(n+\frac{d-2}{2}+\mu)}(r^2)  \, L_j^{(m+\frac{d-2}{2}+\mu)}(r^2)  \,  r^{2\mu+2n+d-1} \e^{-r^2}\,\d r \nonumber
\\&&= \frac12\delta_{mn}\delta_{\ell\iota}  \int_{0}^{\infty} L_k^{(n+\frac{d-2}{2}+\mu)}(\rho)\, L_j^{(n+\frac{d-2}{2}+\mu)}(\rho)  \,  \rho^{n+\frac{d-2}{2}+\mu} \e^{-\rho}\,\d \rho\nonumber
\\&&=\frac{\Gamma(k+n+\frac{d}{2}+\mu)}{2 k!} \delta_{mn}\delta_{kj}\delta_{\ell\iota} =\gamma^{\mu, d}_{k,n}\, \delta_{mn}\delta_{kj}\delta_{\ell\iota},  \nonumber 
\end{eqnarray}
which yields  \eqref{OMH} and  ends the proof. 
\end{proof}


%


The $d$-dimensional GHPs/GHFs satisfy the recurrence relations. 
\begin{proposition}\label{propA} {\em For $\mu>-\frac 1 2 $ and fixed  $(\ell,n)\in \Upsilon_\infty^d,$ we have the following recurrence relations in $k:$ 
\begin{equation}\label{eq:Lemeq2} 
 (k+1)H_{k+1,\ell}^{\mu,n}(\bx)  = \big(2k+n+\frac{d}{2}+\mu-|\bx|^2\big) H_{k,\ell}^{\mu,n}(\bx)-\big(k+n+\frac{d}{2}-1+\mu\big)H_{k-1,\ell}^{\mu,n}(\bx),  
\end{equation}
and for the GHFs, 
\begin{equation}\begin{split}\label{threerda}
a_k\, \widehat{H}_{k+1, \ell}^{\mu,n}(\bx)=(b_{k}-|\bx|^2) \widehat{H}_{k, \ell}^{\mu,n}(\bx)-c_{k}\widehat{H}_{k-1, \ell}^{\mu,n}(\bx),
\end{split}\end{equation}
where
\begin{equation*}\begin{split}
&a_k=\sqrt{(k+1)(k+n+d/2+\mu)},\;\; b_{k}=2 k+n+d/2+\mu,\;\;c_{k}=\sqrt{k(k-1+n+d/2+\mu)}.
\end{split}\end{equation*}}
\end{proposition}
\begin{proof}  Recall  the  three-term recurrence relation of the Laguerre polynomials (cf. \cite{GS1939}):
\begin{equation}
\label{eq:Lagthreeterm}
\begin{split}
& (k+1) L_{k+1}^{(\alpha)}(z)=(2k+\alpha+1-z)L_{k}^{(\alpha)}(z)-(k+\alpha)L_{k-1}^{(\alpha)}(z).
\end{split}
\end{equation}
Then the relation \eqref{eq:Lemeq2}  is a direct consequence of \eqref{MHP}   and \eqref{eq:Lagthreeterm}. 

From \eqref{LagFun01}, we have 
\begin{equation}
\label{LagFun0100}
H_{k,\ell}^{\mu,n}(\bx)= \sqrt{\gamma^{\mu, d}_{k,n}}\, \e^{\frac{|\bx|^2}{2}}  \widehat{H}_{k,\ell}^{\mu,n}(\bx),
\end{equation}
so substituting it into \eqref{eq:Lemeq2} and working out the constants, we obtain \eqref{threerda}.
\end{proof}


The  GHFs with  different parameters  are connected through the following identity, which finds very useful in the algorithm development.
\begin{proposition}\label{ghfexpandrd}
{\em For $\mu,\nu >-\frac 1 2$ and $(\ell,n)\in \Upsilon_\infty^d,$  there holds 
\begin{equation}\label{coefrd1}
 \widehat{H}_{k,\ell}^{\mu,n}(\bx)=\sum_{j=0}^k  {}^{\mu}_{\nu}\CC^{k}_{j}\,\widehat{H}_{j,\ell}^{\nu,n}(\bx),\quad \bx \in \mathbb R^d, \;\; k\in \mathbb N_0, 
 \end{equation}
 where the connection coefficients are given by 
 \begin{equation}\label{coefrd2} 
 {}^{\mu}_{\nu}\CC^k_j=\frac{\Gamma(k-j+\mu-\nu)}{\Gamma(\mu-\nu)\,(k-j)!}\, \sqrt{\frac{k!\,\Gamma(j+n+\frac{d}{2}+\nu)}{j!\,\Gamma(k+n+\frac{d}{2}+\mu)}}\,.
 \end{equation} 
}
 \end{proposition}
\begin{proof}
 Recall the property  of the generalized Laguerre polynomials (cf. \cite[(7.4)]{Askey1975}): 
 \begin{align}\label{huandi}
   L_k^{(\mu+\beta+1)}(z)=\sum_{j=0}^k \frac{\Gamma(k-j+\beta+1)}{\Gamma(\beta+1)(k-j)!}L_j^{(\mu)}(z),
 \end{align}
 so  we can derive the identity from  Definition \ref{ddGHPF} and direct calculation.
\end{proof}
\begin{remark}\label{munu=0} {\em As $\Gamma(0)=\infty,$ we can find that in the limiting sense: $ {}^{\mu}_{\mu}\CC^{k}_{j}=\delta_{kj}.$ \qed  }
\end{remark}


Remarkably, for $\mu=0,$ the GHFs are the eigenfunctions of the Schr\"{o}dinger operator:  $-\Delta+|\bx|^{2}$ with a square potential. This property plays an important part in the error analysis to be conducted in the forthcoming section. 
\begin{theorem}
\label{rddiffrecur2}
For $k\in \mathbb{N}_0,(\ell,n)\in \Upsilon_\infty^d$, the GHFs with $\mu=0$ satisfy 
\begin{equation}\label{ghrdiff2} 
 \big(\!-\Delta+|\bx|^{2}\big) \widehat{H}_{k, \ell}^{0,n}(\bx)=(4 k+2 n+d)\, \widehat{H}_{k, \ell}^{0,n}(\bx). 
\end{equation}
\end{theorem}
\begin{proof}
According to  \cite[Lemma 2.1]{2018Ma} with $\alpha=n+d/2-1$ and $\beta=\alpha+1-d/2$,  we have 
\begin{equation}\label{d1}
\begin{split}
 \Big[\partial_{r}^2+\frac{d-1}{r} \partial_{r}-\frac{n(n+d-2)}{r^2}-r^2+4 k+2n+d\Big] \big[r^n L_{k}^{(n+d/2-1)}(r^2)\,\e^{-\frac{r^2}{2}} \big]=0.
\end{split}
\end{equation}
In view of $Y(\bx)=r^nY(\hat\bx)$, \refe{eq:Delta}, \refe{eq:LaplaceBeltrami}, \refe{MHP}, \refe{LagFun01} and \refe{d1}, we obtain
\begin{equation}\label{d2}
\begin{split}
-\Delta \widehat{H}^{0,n}_{k,\ell}(\bx)&=-\sqrt{1/\gamma^{0,d}_{k,n}}\Big[\partial_r^2 + \frac{d-1}{r} \partial_r- \frac{n(n+d-2)}{r^2}\Big]\, \big[r^n\,L_{k}^{(n+d/2-1)}(r^2)\,\e^{-\frac{r^2}{2}}\big]\,Y(\hat\bx)
\\&=\sqrt{1/\gamma^{0,d}_{k,n}}\big[-r^2+4 k+2n+d\big]\, \big[r^n L_{k}^{(n+d/2-1)}(r^2)\,\e^{-\frac{r^2}{2}}\big]\,Y(\hat\bx)
\\&=(-r^2+4 k+2n+d)\widehat{H}^{0,n}_{k,\ell}(\bx),
\end{split}
\end{equation} 
which leads to \refe{ghrdiff2}. 
\end{proof}

\subsection{Adjoint generalized Hermite functions in $\RR^d$} \label{sec2.3} Our efficient spectral algorithms are essentially built upon the A-GHFs. 
\begin{definition}\label{adjointGHFs}{\em 
For $\mu>-\frac 1 2$ and $(\ell,n)\in \Upsilon_\infty^d,$  the $d$-dimensional adjoint  GHFs   are defined by
 \begin{equation}\label{adher}
 \widecheck{H}_{k,\ell}^{\mu,n}(\bx)=\sum_{j=0}^k (-1)^{k-j}\; {}^{\mu}_0\CC^k_{j}\, \widehat{H}_{j,\ell}^{0,n}(\bx),
  \quad \bx \in \mathbb R^d, \;\; k\in \mathbb N_0,
 \end{equation}
 where the coefficients $\{{}^{\mu}_0\CC^k_j\}$ are given by   \refe{coefrd2}. }
\end{definition}
\begin{remark}{\em In light of the connection relation in Proposition {\rm \ref{ghfexpandrd}},  it is evident that $ \widecheck{H}_{k,\ell}^{\mu,n}(\bx)$ can be expressed as a linear combination of the counterparts $\big\{\widehat{H}_{j,\ell}^{\mu,n}(\bx)\big\}_{j=0}^k.$ \qed }
\end{remark}

It is seen from \eqref{coefrd1} (with $\nu=0$) that  the GHF can be represented as 
 \begin{equation}\label{adherGHFs}
\widehat{H}_{k,\ell}^{\mu,n}(\bx)=\sum_{j=0}^k  {}^{\mu}_0\CC^k_{j}\, \widehat{H}_{j,\ell}^{0,n}(\bx),
 \end{equation}
which only differs from its adjoint by  the signs  of the coefficients.  Notably,  such a subtlety results in an intimate relation between this adjoint pair through the Fourier transform: 
\begin{align}\label{FourierDef}
 \hat u(\bs \xi):= \mathscr{F}[u]( \bxi)=\frac{1}{(2\pi)^{\frac{d}{2}}}\int_{\mathbb{R}^d} u(\bx)\,\e^{-\i\langle \bxi,  \bx \rangle}\,\d{\bx}, \;\;\mathscr{F}^{-1}[\widehat{u}]( \bx)=\frac{1}{(2\pi)^{\frac{d}{2}}}\int_{\mathbb{R}^d} \widehat u(\bxi)\,\e^{\i\langle \bxi,  \bx \rangle}\,\d{\bxi}.
\end{align}
Moreover,  the use of A-GHFs as basis functions in a spectral-Galerkin framework can diagonalise the nonlocal integral fractional Laplacian $ (-\Delta)^s$ for $s>0$. Recall that   
for $s>0,$ the fractional Laplacian of $u\in  \mathscr{S}(\mathbb{R}^d)$ (functions of Schwarz class)
 can be naturally defined via the Fourier transform:
\begin{equation}\label{eq:LapDF}
\begin{split}
 (-\Delta)^s u(\bx)={\mathscr F}^{-1}\big[|\bxi|^{2s}  {\mathscr F}[u](\bxi)\big](\bx),\quad \bx\in {\mathbb R}^d.
\end{split}
\end{equation}
For $0<s<1,$ the fractional Laplacian can be equivalently defined by the 
point-wise formula (cf. \cite{Nezza2012BSM}):
\begin{equation}\label{fracLap-defn} 
(-\Delta)^s u(\bx)=C_{d,s}\, {\rm p.v.}\! \int_{\mathbb R^d} \frac{u(\bx)-u(\by)}{|\bx-\by|^{d+2s}}\, {\rm d}\by,\quad C_{d,s}:=
\frac{2^{2s}s\Gamma(s+d/2)}{\pi^{d/2}\Gamma(1-s)},
\end{equation}
where ``p.v." stands for the principle value.

\begin{theorem}\label{newGHfsA}
For $\mu>-\frac 1 2, (\ell,n)\in \Upsilon_\infty^d$ and $k\in \mathbb N_0,$  we have 
\begin{equation}\label{FourierFLH}
 \mathscr{F}[\widecheck{H}_{k,\ell}^{\mu,n}](\bxi)=(-\i)^{n+2k}\widehat{H}_{k,\ell}^{\mu,n}(\bxi), \quad 
 \mathscr{F}^{-1}[\widehat{H}_{k,\ell}^{\mu,n}](\bx)=\i^{n+2k}\widecheck{H}_{k,\ell}^{\mu,n}(\bx), 
\end{equation} 
and for $s>0,$
\begin{equation}\label{FourierFLH2} 
 \mathscr{F}[(-\Delta)^{s}\widecheck{H}_{k,\ell}^{\mu,n}](\bxi)=(-\i)^{n+2k}|\bxi|^{2s}\widehat{H}_{k,\ell}^{\mu,n}(\bxi).
\end{equation} 
Moreover, the adjoint GHFs are orthonormal in the sense that for $s>0,$
\begin{equation} \label{innadh1}
\big((-\Delta)^{\frac{s}{2}}\widecheck{H}_{k, \ell}^{s, n},(-\Delta)^{\frac{s}{2}}\widecheck{H}_{j, \iota}^{s, m}\big)_{\mathbb R^d}=\delta_{jk} \delta_{mn} \delta_{\ell\iota}.
\end{equation}
\end{theorem}
\begin{proof} 
We first show that  the GHFs with   $\mu=0$  
are eigenfunctions of the Fourier transform, namely,
\begin{equation}\label{FourierH}
 \mathscr{F}[\widehat{H}_{k,\ell}^{0,n}](\bxi)
 = (-\i)^{n+2k}\widehat{H}_{k,\ell}^{0,n}(\bxi).
\end{equation}
According to \cite[Lemma 9.10.2]{AAR1999},  we have that  for  $\omega>0$,
\begin{equation}\label{adlemma21}
\int_{\sph^{d-1}} Y_{\ell}^{n}(\hat{\bx}) \,\e^{-\i \omega\,  \langle \hat{\bxi},  \hat{\bx} \rangle}  \d \sigma(\hat{\bx})=\frac{(-\i)^n (2\pi)^{\frac{d}{2}}}{\omega^{\frac{d-2}2}} J_{n+\frac{d-2}{2}}(\omega)\, Y_{\ell}^{n}(\hat{\bxi}), \;\;\; \hat{\bxi}\in \sph^{d-1}, 
\end{equation}
where $J_{\nu}(z)$ is the  Bessel functions of the first kind of order $\nu.$ 
Then using  Definition  \ref{ddGHPF} with $\mu=0,$ and  \eqref{adlemma21}  with $\omega=\rho r$ and $\rho=|\bxi|$,   leads to 
\begin{equation}\begin{split}\label{sh1}
 & \mathscr{F}[\widehat{H}_{k,\ell}^{0,n}](\bxi)=\frac{1}{(2\pi)^{\frac{d}{2}}}\int_{\mathbb{R}^{d}} \widehat{H}_{k,\ell}^{0,n}(\bx)\e^{-\i \langle  \bxi, \bx\rangle} \mathrm{d} \bx
\\&\quad = \frac 1 {\sqrt{\gamma^{0, d}_{k,n}}}\frac{1}{(2\pi)^{\frac{d}{2}}}\int_{0}^{\infty}r^n L_k^{(n+\frac{d-2}{2})}(r^2)\e^{-\frac{r^2}2}  \bigg\{\int_{\mathrm{S}^{d-1}} Y_{\ell}^{n}  (\hat{\bx}) \e^{-\i \rho r \langle\hat{\bxi}, \hat{\bx}\rangle}     \d \sigma(\hat{\bx})\bigg\}  r^{d-1}\,  \d r
 \\&\quad =\frac 1 {\sqrt{\gamma^{0, d}_{k,n}}} \frac{(-\i)^n}{\rho^{\frac{d-2}2}} \bigg\{\int_{0}^{\infty}r^{n+\frac d 2} L_k^{(n+\frac{d-2}{2})}(r^2)\, \e^{-\frac{r^2}2}   J_{n+\frac{d-2}{2}}(\rho r)\,  \d r\bigg\} Y_{\ell}^{n}(\hat{\bxi}),\;\; \rho>0. 
\end{split}\end{equation}
Recall the integral identity of the generalised Laguerre polynomials (cf.  \cite[P. 820]{Gradshteyn2015Book}):  for $ \alpha>-1,$ 
\begin{equation}\label{tlag1}
 \int_0^{\infty}r^{\alpha+1}L_k^{(\alpha)}(r^2)\,\e^{-\frac{r^2}{2}}\,J_{\alpha}(\rho r) \,\d r= (-1)^k \rho^{\alpha}  L_k^{(\alpha)}(\rho^2)\,\e^{-\frac{\rho^2}{2}},\;\;\; \rho>0. 
\end{equation}
Thus,  taking $\alpha=n+\frac{d-2}{2}$ in \eqref{tlag1}, we can work out the integral in  \refe{sh1} and then obtain from \refe{LagFun01} with $\mu=0$  that 
\begin{align}\label{asdadd}
 \mathscr{F}[\widehat{H}_{k,\ell}^{0,n}](\bxi)
 =  \frac { (-1)^k} {\sqrt{\gamma^{0, d}_{k,n}}} \frac{ (-\i)^n}{\rho^{\frac{d-2}{2}}}
 \rho^{n+\frac{d-2}{2}}  L_k^{(n+\frac{d-2}{2})}(\rho^2)\e^{-\frac{\rho^2}{2}}\, Y_{\ell}^n(\hat{\bxi}) =(-\i)^{n+2k}\widehat{H}_{k,\ell}^{0,n}(\bxi).
\end{align}
This yields \eqref{FourierH}. 


From  Definition \ref{adjointGHFs} and the property  \eqref{asdadd}, we obtain
\begin{equation}\begin{split}\label{provefour1}
 \mathscr{F}[\widecheck{H}_{k,\ell}^{\mu,n}](\bxi)  
 &=\sum_{j=0}^k (-1)^{k-j}\, {}^{\mu}_0\CC^k_{j} \, \mathscr{F}[\widehat{H}_{j,\ell}^{0,n}](\bxi) 
 = \sum_{j=0}^k (-1)^{k-j}\, (-\i)^{n+2j}\,  {}^{\mu}_0\CC^k_{j}\,  \widehat{H}_{j,\ell}^{0,n}(\bxi)
 \\&=(-\i)^{n+2k}\sum_{j=0}^k  {}^{\mu}_0\CC^k_{j} \, \widehat{H}_{j,\ell}^{0,n}(\bxi)=(-\i)^{n+2k} \widehat{H}_{k,\ell}^{\mu,n}(\bxi),
\end{split}\end{equation}
where  in the last step, we used  \eqref{adherGHFs}. This gives the first identity in  \eqref{FourierFLH}, and the second is its immediate consequence.  The property \refe{FourierFLH2} follows  directly from  \eqref{FourierFLH} and the definition of fractional Laplacian \eqref{eq:LapDF}. 

Finally,  using the Parseval's identity and \refe{FourierFLH2},    we  derive from the orthogonality  \eqref{HerfunOrthA0}  that
 \begin{align*}
&\big((-\Delta)^{\frac{s}{2}}\widecheck{H}_{k, \ell}^{s, n}, (-\Delta)^{\frac{s}{2}}\widecheck{H}_{j, \iota}^{s, m}\big)_{\mathbb R^d}=\big( \mathscr{F}[(-\Delta)^{\frac s2}  \widecheck{H}_{k,\ell}^{s,n}], \mathscr{F}[ (-\Delta)^{\frac s2}   \widecheck{H}_{j,\iota}^{s,m}] \big)_{\mathbb R^d}
 \\&\quad = (-\i)^{n-m+2k-2j}( |\bxi|^{2s} \widehat{H}_{k,\ell}^{s,n}, \widehat{H}_{j,\iota}^{s,m}  )_{\mathbb R^d}=\delta_{mn}\delta_{kj}\delta_{\ell\iota}.
 \end{align*}
This yields  \eqref{innadh1} and  ends the proof.
\end{proof}

Note the GHFs are real-valued, so we infer from \eqref{FourierDef} readily that  
$$
 \mathscr{F}[\widehat{H}_{k,\ell}^{\mu,n}(\bs x)](\bxi)= \big\{\mathscr{F}^{-1}[\widehat{H}_{k,\ell}^{\mu,n}(\bs x)](\bs \xi)\big\}^*.
$$
Thus, we find from \eqref{FourierFLH} immediately the following ``reversed'' form of \eqref{FourierFLH}. 
\begin{cor}\label{newGHfsAB}
For $\mu>-\frac 1 2, (\ell,n)\in \Upsilon_\infty^d$ and $k\in \mathbb N_0,$  we have 
\begin{equation}\label{FourierFLH00}
 \mathscr{F}[\widehat{H}_{k,\ell}^{\mu,n}](\bxi)=\i^{n+2k}\widecheck{H}_{k,\ell}^{\mu,n}(\bxi), \quad 
 \mathscr{F}^{-1}[\widecheck{H}_{k,\ell}^{\mu,n}](\bx)=(-\i)^{n+2k}\widehat{H}_{k,\ell}^{\mu,n}(\bx). 
\end{equation} 
\end{cor}

\begin{remark} {\em  The fractional Sobolev orthogonality  \eqref{innadh1} has profound implications even for the integral-order Laplacian $(-\Delta)^m$ with $m\in {\mathbb N}.$ For example,  we find from \eqref{adher} with $s=1$  that the A-GHFs read 
\begin{equation*}
 \widecheck{H}_{k,\ell}^{1,n}(\bx)=\sqrt{\frac{k!}{\Gamma(k+n+\frac{d}{2}+1)}} \sum_{j=0}^k\sqrt{\frac{\Gamma(j+n+\frac{d}{2})}{j!}}\widehat{H}_{j,\ell}^{0,n}(\bx),
\end{equation*} 
which are orthogonal with respect to $(\nabla\cdot,\nabla \cdot)_{\RR^d}.$
However, this attractive property is not valid for the usual Hermite-based methods based on tensorial  Hermite functions $\prod^d_{j=1}\widehat{H}_{n_j}(x_j)$. Thus, it is advantageous to use the A-GHFs for 
usual Laplacian and bi-harmonic Laplacian {\rm(}using the A-GHFs with $s=2${\rm)} in $\mathbb R^d$. \qed 
}
\end{remark}

\begin{remark} {\em In contrast to  \eqref{FourierH},   the eigen-functions of  the finite Fourier transform are the ball prolate spheroidal wave functions  introduced in \cite{2018ZhangLi},  which find   useful in approximating bandlimited functions. \qed }
\end{remark}


\subsection{Differences and connections with some existing generalisations}\label{1DGHPF}
There have  been some existing generalisations of the usual Hermite polynomials/functions in different senses, so we  feel compelled to outline the differences and connections between the GHPs/GHFs herein with the most relevant ones in literature.  

\subsubsection{GHPs/GHFs in Szeg\"{o} \cite{GS1939} } \label{subsect2.2}

  Note from  \eqref{cpCp} that for $d=1$, $a^1_0=a_1^1=1$ and $a_n^1=0$ for $n\ge 2,$ so there exist 
 only two orthonormal harmonic polynomials: $Y_{1}^{0}(x)=\frac{1}{\sqrt{2}} \text { and } Y_{1}^{1}(x)=\frac{x}{\sqrt{2}}$.
Thus,  the GHPs in Definition \ref{ddGHPF}  reduce to  
\begin{equation}\begin{split}\label{1drelat0}
&H_{2k}^{(\mu)}(x):=(-1)^k\,2^{2k+\frac12}\,k! \,H_{k,1}^{\mu,0}(x)=(-1)^k \,2^{2k}\,k!\,  L_k^{(\mu-\frac 1 2)}(x^2),\\
&H_{2k+1}^{(\mu)}(x):= (-1)^k \,2^{2k+\frac3 2}\,k! \,H_{k,1}^{\mu,1}(x)= (-1)^k \,2^{2k+1}\,k!\, x \,  L_k^{(\mu+\frac 1 2)}(x^2),
 \end{split}
 \end{equation}
 which are mutually orthogonal with respect to the weight function $|x|^{2\mu}\,\e^{-x^2}$ on $\mathbb R.$ 
 In fact, this family of GHPs was first  introduced by Szeg\"{o} in \cite[P. 371]{GS1939} as an exercised problem 
 and promoted by   Chihara in the PhD dissertation   \cite[entitled as ``Generalised Hermite Polynomials"]{chihara1955}(1955),  and  his book   \cite{chihara2011}(1978).  According to   Szeg\"{o} \cite[Prob.\! 25]{GS1939},  the GHPs with $\mu>-\frac 12$ satisfy the differential equation: 
 \begin{equation}\label{problem25}
\begin{split}
&x y^{\prime \prime}+2(\mu-x^2) y^{\prime}+(2 n x-\theta_n x^{-1}) y=0, \quad 
\theta_n=
\begin{cases}0, & {n \text { even}}\,,  \\ {2\mu,} & n \text { odd}\,; 
\end{cases}  \;\;\;\; y=H_{n}^{(\mu)}(x),
\end{split}
\end{equation} 
Some other properties  of $\{H_{n}^{(\mu)}\}$ can be founded in  Chihara \cite{chihara1955,chihara2011}. We also refer to some limited works on
 the analytic studies  or further generalisations \cite{Rosler1998,twoclasses}. With the normalisation in \eqref{1drelat0},
 the orthonormal  GHFs take the form
\begin{equation}\label{GHermfunc}
\widehat{H}_n^{(\mu)}(x):=\sqrt{1/\gamma^{(\mu)}_n}\;\e^{-\frac{x^2}{2}}H_n^{(\mu)}(x),\quad 
\gamma_n^{(\mu)}:=2^{2n}\,\Big[\frac n 2 \Big]!\, \Gamma\Big(\Big[\frac {n+1} 2 \Big]+\mu+\frac 1 2\Big),
\end{equation}
 In particular, for $\mu=0,$ they reduce to the usual Hermite polynomials/functions. For distinction,   we  denote them  by 
 $H_n(x)$ and $\widehat H_n(x),$ respectively. 
 
It is known that   $\{\widehat H_n\}$ are the eigenfunctions of the Fourier transform. However, this property cannot carry over to    
the GHFs with $\mu\not=0.$  In \cite[(2.34)]{twoclasses},  the Fourier transform of $ \widehat{H}_{n}^{(\mu)}(x)$ was expressed 
in terms of the Kummer hypergeometric function ${}_1F_1(\cdot).$  In contrast, 
the general result in Corollary \ref{newGHfsAB}  implies a more informative representation as follows
\begin{equation}\label{FourierFLH1dAsd}
\frac 1 {\sqrt{2\pi}}\int_{\mathbb R} \widehat H_n^{(\mu)}(x) {\rm e}^{-{\rm i} \xi x }\d x= 
\i^n\widecheck{H}_{n}^{(\mu)}(\xi),\quad \xi\in \mathbb R, 
\end{equation}
where the adjoint GHFs are given by 
\begin{equation}\begin{split}\label{adher1d1}
 \widecheck{H}^{(\mu)}_n(x)=\!\!\sum_{j=0\atop j+n\,{\rm even}}^{n}\!(-1)^{\frac{n-j}{2}}\, {}^{\mu}_0\C^n_j\,  \widehat{H}_j(x), \quad\;  x\in \mathbb R, 
 \end{split}\end{equation}
 and for even $j+n,$ the coefficients are 
 \begin{equation}\label{thenj}\begin{split}
{}^{\mu}_{0}\C^n_j=\frac{(-1)^{\frac{n-j}{2}}\sqrt{\Gamma([\frac{n}{2}]+1)\Gamma([\frac{j+1}{2}]+\frac12)}\,\, \Gamma\big(\frac{n-j}{2}+\mu\big)}    {\Gamma(\mu)\sqrt{\Gamma([\frac{n+1}{2}]+\mu+\frac12)\Gamma([\frac{j}{2}]+1)}\,\, \Gamma\big(\frac{n-j}{2}+1\big) }.
\end{split}\end{equation}
Note that the formulation of the adjoint GHFs in \eqref{adher1d1}-\eqref{thenj} needs some simple calculation from  \eqref{adher} and \eqref{1drelat0}.  

Indeed,  the study of one-dimensional GHPs/GHFs is of  much independent interest in developing methods using multi-dimensional tensorial basis functions, or possible sophisticated generalisation \eqref{orthsAH} discussed in Yurova \cite{yurova2020} with applications in quantum dynamics and plasma physics.
In what follows, we present some approximation results which can be extended to the tensorial case straightforwardly and which are new to the best of our knowledge. 

Define the weight functions $\chi^{(\mu)}=|x|^{2\mu} \e^{-x^2}$ and $\omega^{(\mu)}=|x|^{2\mu}.$  Consider  the $L^2_{\chi^{(\mu)}}$-orthogonal
projection $\Pi^{(\mu)}_{N}: L^2_{\chi^{(\mu)}}(\RR)\rightarrow \mathbb{P}_{N}={\rm span}\{1,x,\cdots, x^N\}$ defined by
 \begin{equation}\label{ll2orth}
\big(u-\Pi^{(\mu)}_{N} u, v\big)_{\chi^{(\mu)}}=0, \quad \forall v \in \mathbb{P}_{N}.
\end{equation}
For any $u\in L^2_{\omega^{(\mu)}}(\RR)$, we have $u\e^{\frac{x^2}{2}}\in L^2_{\chi^{(\mu)}}(\RR)$, and define
 \begin{equation}\label{proj1d}
\widehat{\Pi}^{(\mu)}_{N}u:=\e^{-\frac{x^2}{2}} \Pi^{(\mu)}_{N}(u\, \e^{\frac{x^2}{2}}) \in {\mathcal V}_{N}:=\big\{\phi=\e^{-\frac{x^2} 2} \psi : \psi\in {\mathbb P}_{N}\big\},
\end{equation}
which turns out to  be the $L^2_{\omega^{(\mu)}}$-orthogonal projection, as 
\begin{equation}
(u-\widehat{\Pi}^{(\mu)}_{N} u, v)_{\omega^{(\mu)}}=\big(u \e^{\frac{x^2}{2}}-\Pi^{(\mu)}_{N}(u\e^{\frac{x^2}{2}}), v\e^{\frac{x^2}{2}}\big)_{\chi^{(\mu)}}=0, \quad \forall v\in {\mathcal V}_{N}.
\end{equation}


Similar to the introduction of the Dirac's ladder operators in usual Hermite approximation (i.e., $\mu=0$, see \cite{Lubich2008book,ShenTao2011,yurova2020}), we define the new derivative operator
\begin{equation}
 \label{pseudoder} 
D_x u=\partial_x u_{\rm e} + \partial_x \Big(\frac{u_{\rm o}}{2x}\Big),\;\;\; u_{\rm e}(x)=\frac{u(x)+u(-x)}{2},\;\; 
u_{\rm o}(x)=\frac{u(x)-u(-x)}{2}.
 \end{equation}
Note that if $u$ is an odd (resp.  even) function, then $D_x u=\partial_x (u/(2x))$ (resp.  $D_x u=\partial_x u).$ Clearly, $D_x u$ is an odd function.  Then the modified higher order derivative of  general $u$  takes  the form  
\begin{equation}\label{higherOrder}
D_x^2 u=D_x \{ D_x u\}=\partial_x\Big \{\frac {D_x u} {2x}\Big\},\quad  
D_x^3 u=\partial_x\Big\{\frac 1 {2x} \partial_x\Big \{\frac {D_x u} {2x}\Big\}\Big\},
\end{equation}
and likewise, we can define $D_x^l u$ for $l\ge 4.$ Accordingly, to characterise the space of functions to be approximated, 
we introduce the vector space ${\mathcal B}_{\mu}^{m}(\RR),$ $m\in {\mathbb N}$,
equipped with the norm and semi-norm
\begin{equation*}\begin{split}
&\|u\|_{{\mathcal B}_{\mu}^{m}(\RR)}=\Big(\|u\|_{\chi^{(\mu)}}^{2}+\sum_{l=1}^{m}(\|D_x^{l} u_{\rm e}\|_{\chi^{(\mu+l-1)}}^{2}+\|D_x^l u_{\rm o}\|_{\chi^{(\mu+l)}}^{2})\Big)^{\frac12}, 
\\&|u|_{{\mathcal B}_{\mu}^{m}(\RR)}=\Big(\|D_x^mu_{\rm e}\|^2_{\chi^{(\mu+m-1)}}+\|D_x^mu_{\rm o}\|^2_{\chi^{(\mu+m)}}\Big)^{\frac12},\quad m\geq 1.
\end{split}\end{equation*}
For $m=0$, we define  ${\mathcal B}_{\mu}^{0}(\RR)=L^2_{\chi^{(\mu)}}(\RR).$
The main approximation results are stated below, whose proof will be given  in Appendix \ref{appb}.
 \begin{thm}
 \label{lemghp} 
 For any $u \in {\mathcal B}^m_\mu(\RR)$ with $\mu>-\frac12, \mu\not=0$ and integer $0\le m\le  [\frac{N+1} 2],$ we have 
 \begin{equation}\label{resultghp}
\big\|\Pi^{(\mu)}_{N}u-u\big\|_{\chi^{(\mu)}}\leq    \Big(\Big[\frac {N+1} 2\Big]-m+1\Big)_m^{-\frac 12}\, |u|_{{\mathcal B}_{\mu}^{m}(\RR)},
\end{equation}
On the other hand, if $u e^{\frac {x^2} 2} \in {{\mathcal B}}_{\mu}^{m}(\RR),$ with $\mu>-\frac12, \mu\not=0$ and with integer  $0\leq m\leq [\frac{N+1}2],$ then
\begin{equation}\label{errest0}
\big\| \widehat{\Pi}^{(\mu)}_{N} u-u\big\|_{\omega^{(\mu)}} \leq  \Big(\Big[\frac {N+1} 2\Big]-m+1\Big)_m^{-\frac 12}\, \big|u e^{\frac {x^2} 2} \big|_{{\mathcal B}_{\mu}^{m}(\RR)}. 
\end{equation}
Here, $(\alpha)_m=\alpha(\alpha+1)\cdots(\alpha+m-1)$ the rising factorial in the Pochhammer symbol.
 \end{thm}
 \begin{remark} {\em The above approximation result is  extendable to  the $d$-dimensional tensorial Hermite polynomials:
 $H_{\ell}^{(\mu)}(\bs x)=H_{\ell_1}^{(\mu_1)}(x_1)\cdots H_{\ell_d}^{\mu_d}(x_d),$ so is the tensorial Hermite functions. It is likely to explore   the generalisation recently considered in \cite{yurova2020}.} \qed  
 \end{remark}

\subsubsection{2D GHFs versus  generalised Hermite bases  for  Bose-Einstein condensates in \cite{BLS2009}} \label{subsect2.5}
For $d=2$, the dimensionality of the space $\mathcal{H}_n^2$ in \eqref{cpCp} is  $a^2_n=2-\delta_{n0},$ with the orthogonal basis   given by the real and imaginary parts of $(x_1+\i x_2)^n$. In polar coordinates, we have
\begin{equation}\label{2Dbasis}
Y_{1}^{0}(\bs{x})=\frac{1}{\sqrt{2 \pi}}, \quad Y_{1}^{n}(\bs{x})=\frac{r^{n}}{\sqrt{\pi}} \cos (n \theta), \quad Y_{2}^{n}(\bs{x})=\frac{r^{n}}{\sqrt{\pi}} \sin (n \theta), \quad n \geq 1.
\end{equation}
Then by \eqref{LagFun01}, the GHFs  can be expressed as
\begin{equation}\label{1drelat01}
\begin{split}
&\widehat H_{k,1}^{\mu,0}(\bx)= \frac 1 {\sqrt{2\pi\gamma^{\mu, 2}_{k,0}}}\,\e^{-\frac{r^2}{2}} L_k^{(\mu)}(r^2), \quad
\widehat H_{k,1}^{\mu,n}(\bs x)= \frac {1} {\sqrt{\pi\gamma^{\mu, 2}_{k,n}}}\,r^n\, \e^{-\frac{r^2}{2}} \,L_k^{(n+\mu)}(r^2)\cos (n\theta),\;
\\& \widehat H_{k,2}^{\mu,n}(\bs x)= \frac {1} {\sqrt{\pi\gamma^{\mu, 2}_{k,n}}}\,r^n\, \e^{-\frac{r^2}{2}}\,L_k^{(n+\mu)}(r^2)\sin (n\theta), \quad n\ge 1,\;\; k\ge 0.
 \end{split}
 \end{equation}
 
Note that similar constructions for the 2D GHFs with $\mu=0$ have been explored in  the computation of  the ground states and dynamics of Bose-Einstein condensation (cf.\! \cite{BLS2009}),  
governed by the Gross-Pitaevskii equation with an angular momentum rotation term:
\begin{equation}\label{GPErd}
\begin{aligned}
&{\rm i} \partial_{t} \psi(\bs x, t)=\Big(\!-\frac12\Delta+\frac{\gamma^{2}}{2}  |\bs x|^{2}+\Omega L_z+\beta|\psi(\bs x,t)|^2\Big) \psi(\bs x, t),\;\;\bs x\in \RR^2,\; t>0,\\&
\psi(\bs x,0)=\psi_0(\bs x),\;\;\bs x\in \RR^2; \quad \psi(\bs x,t)\to 0\;\; {\rm as}\;\;  |\bs x|\to\infty, \;\;\; t\ge0, 
\end{aligned}
\end{equation}
where  the constants $\gamma,\beta>0$, $\Omega$ is the dimensionless angular momentum rotation speed and $L_z=-{\rm i}(x\partial_y-y\partial_x)=-{\rm i}\partial_\theta$ in polar coordinates. 
The efficient  spectral algorithm therein  was built upon the  constructive basis  $\{r^n\, \e^{-\frac{\gamma r^2}{2}}\,L_k^{(n)}(\gamma r^2){\rm e}^{{\rm i}m\theta}\}$ that could diagonalise the Schr\"{o}dinger  operator:  $-\Delta+\gamma|\bs x|^2.$
Similar idea was extended to \eqref{GPErd} in $\mathbb R^3$ in cylindrical coordinates by using the tensor product of 
the 2D basis and the usual Hermite function in the $z$-direction in \cite{BLS2009}.

As shown in Theorem \ref{rddiffrecur2}, the GHFs with $\mu=0$ are eigenfunctions of the operator:  $-\Delta+|\bs x|^2,$ so with a proper scaling, the spectral algorithm leads to a diagonal matrix for the operator: $-\Delta+\gamma^2 |\bs x|^2.$
As we shall show in the late part,  our GHFs with $\mu\not=0$  offer a new  and efficient tool for the solutions of PDEs involving a  more general Schr\"{o}dinger operator: $(-\Delta)^s + |\bs x|^{2\mu}$ with $s\in (0,1]$ and $\mu>-1/2.$ 

 \subsubsection{3D GHPs versus Burnett polynomials \cite{burnett1936}} \label{subsect2.6}
 For $d=3$, the dimensionality of $\mathcal{H}_n^3$ in \eqref{cpCp}  is $a^3_n=2n+1$.  The   orthonormal basis in the spherical coordinates $\bx=(r\sin\theta\cos\phi,r\sin\theta\sin\phi,$ $r\cos\theta)^t$ takes the form 
\begin{equation}\label{1drelat020}
\begin{split}
&Y_{1}^{n}(\bs{x})=\dfrac{1}{\sqrt{8 \pi}} P_{n}^{(0,0)}(\cos \theta);  \;\;
 Y_{2 l}^{n}(\bs{x})=\dfrac{r^{n}}{2^{l+1} \sqrt{\pi}}(\sin \theta)^{l} P_{n-l}^{(l, l)}(\cos \theta) \cos (l \phi), 
\\& Y_{2 l+1}^{n}(\bs{x})=\dfrac{r^{n}}{2^{l+1} \sqrt{\pi}}(\sin \theta)^{l} P_{n-l}^{(l, l)}(\cos \theta) \sin(l \phi), \;\;\; 1\le l\le n, 
\end{split}
 \end{equation}
 where $\{P_k^{(l,l)}\}$ are the Gegenbauer polynomials.  Then the 3D GHPs/GHFs in Definition \ref{ddGHPF} read more explicit. 
 In fact, for $\mu=0,$ the GHPs with a scaling turn out to be the Burnett polynomials, which were first proposed  by Burnett \cite{burnett1936} as follows   
\begin{equation}\label{d3Burnett}
B_{k,\ell}^n(\bx)=c_{k}^n \; r^n L_{k}^{(n+\frac12)}\Big(\frac{r^2}2\Big)Y_{\ell}^{n}(\hat{\bx}), \quad  k\in \mathbb{N}_0 , \;  (\ell, n)\in \Upsilon_\infty^3,
\end{equation}
where $c_{k}^n$ is the normalisation  constant so that they are orthogonal in the sense
\begin{equation}\label{edcaseA}
\int_{\mathbb{R}^3} B_{k,\ell}^n(\bx)\, B_{j,\iota}^m(\bx) \, \e^{- \frac {|\bx|^2} 2}\, \d \bx=\delta_{kj}\delta_{mn}\delta_{\ell\iota}.
\end{equation}
As a result,   the Burnett polynomials are mutually orthogonal with respect to the Maxwellian  
 $\mathcal{M}(\bx)=\frac{1}{(2 \pi)^{3/2}} \e^{-\frac{|\bx|^{2}}{2}}.$  It is evident that by \eqref{MHP} and \eqref{OMH} (with $d=3$ and $\mu=0$), 
 \begin{equation}\label{BurrentRelation}
H_{k,\ell}^{0,n}(\bx) = \tilde c_k^n \,B_{k,\ell}^n(\sqrt 2\bx),\quad \bs x\in \mathbb R^3.
 \end{equation}
We remark that the Burnett polynomials are frequently used  as basis functions in solving kinetic equations (cf.\! \cite{Cai2018,Hu2020} and the references therein). 
 
 \section{GHF approximation of the IFL and the Schr\"{o}dinger equation}\label{secghg}\setcounter{equation}{0}  \setcounter{thm}{0} \setcounter{lem}{0} \setcounter{remark}{0}

In this section, we implement and analyse the GHF-spectral-Galerkin method for PDEs involving integral fractional Laplacian. 


\subsection{GHF-spectral-Galerkin method for a fractional model problem}  As an illustrative example, we   consider 
\begin{equation}\label{soureceprob}
(-\Delta)^s u(\bx)+\gamma u(\bx)=f(\bx) \;\;\;  {\rm in}\;\;  \mathbb{R}^d; \quad 
 u(\bx)\to 0 \;\;\;  {\rm as}\;\;  |\bx|\to \infty,
\end{equation}
where $s\in (0,1), \gamma>0$, $f\in H^{-s}(\mathbb{R}^d),$   and the fractional Laplacian operator is defined in \eqref{eq:LapDF}-\eqref{fracLap-defn}.   Here, the fractional Sobolev space $H^{s}(\mathbb{R}^d)$ with real $s$ is defined as in \cite{Nezza2012BSM}.

A weak formulation of  \eqref{soureceprob} is to find $u\in H^s(\RR^d)$ such that
\begin{equation}\label{uvsh}
\begin{split}
{\mathcal A}_s (u,v) =\big((-\Delta)^{\frac{s}{2}}u,(-\Delta)^{\frac{s}{2}}v\big)_{\mathbb R^d}+ \gamma(u,v)_{\mathbb{R}^d}=(f,v)_{\mathbb{R}^d},\quad \forall v\in  H^s(\mathbb R^d).
\end{split}
\end{equation}
From \eqref{eq:LapDF}, we find readily the continuity  and coercivity of the bilinear form ${\mathcal A}_s(\cdot,\cdot)$. 
Then we conclude from the standard Lax-Milgram lemma  that  the problem \eqref{uvsh} admits a unique solution satisfying
$\|u\|_{H^{s}(\mathbb{R}^d)}\le c\|f\|_{H^{-s}(\mathbb{R}^d)}.$

We choose  the finite dimensional approximation space spanned by the $d$-dimensional GHFs in Definition \ref{ddGHPF} or equivalently by the A-GHFs in Definition \ref{adjointGHFs}.  However, in view of 
\eqref{innadh1},  it is advantageous to use the latter as the basis functions, so we define 
\begin{equation}\label{spacefinite0}
{\mathcal V}_{\!N}^d := \text{span}\big\{\widecheck{H}_{k,\ell}^{s,n}(\bx): \; 0\le n\le N, \; 1\le \ell\le a_n^d,~
0\le 2 k\le N-n,\; k, \ell, n\in {\mathbb N}_0\big\}. 
\end{equation}
Then, the spectral-Galerkin approximation to \eqref{uvsh} is to find $u_N \in {\mathcal V}_{\!N}^d$ such that
\begin{align}\label{specf}
 \mathcal{A}_s(u_{N},v_{N})= (f,v_{N})_{\mathbb R^d},  \quad \forall v_N\in {\mathcal V}_{\!N}^d.
\end{align}
As with the continuous problem  \eqref{uvsh}, it has a unique solution  $u_{N}\in {\mathcal V}_{\!N}^d.$

In the real implementation,  we  write 
\begin{align}\label{uexpansion}
u_{N}(\bx)= \sum_{n=0}^N \sum_{\ell=1}^{a_n^d}\sum_{k=0}^{[\frac{N-n}{2}]} \tilde  u_{k,\ell}^{n}\widecheck{H}_{k,\ell}^{s,n}( \bx),
\end{align}
and  arrange the unknown coefficients in the order 
\begin{equation}\label{ufform}
\begin{split}
&\bs{u}=\big(\tilde{\bs u}_1^0,\tilde{\bs u}_2^0,\cdots,\tilde{\bs u}_{a_0^d}^0,\tilde{\bs u}_1^1,\tilde{\bs u}_2^1,\cdots,\tilde{\bs u}_{a_1^d}^1, \cdots,\tilde{\bs u}_1^N,\tilde{\bs u}_2^N,\cdots,\tilde{\bs u}_{a_N^d}^N\big)^{t},\\
&  \tilde{\bs u}^n_{\ell} = \big(\tilde{u}^n_{0,\ell}, \tilde{u}^n_{1,\ell},\cdots, \tilde{u}^n_{[\frac{N-n}{2}],\ell}\big)^t,
\end{split}
\end{equation}
and likewise for $\bs f,$ but with the components $\tilde  f^n_{k,\ell} = (f, \widecheck{H}^{s,n}_{k,\ell})_{\mathbb R^d}.$  
The orthogonality \eqref{innadh1} implies that the stiffness matrix is an identity matrix. Moreover, in view of  the orthogonality of the spherical harmonic basis (cf.  \eqref{Yorth}), the corresponding mass matrix is block diagonal as follows 
\begin{equation}\label{Mmatrix}
\begin{split}
&\bs M= \mathrm{diag}\big\{\bs M^0_{1}, \bs M^0_{2},  \cdots,{\bs M}_{a_0^d}^0, {\bs M}_1^1,{\bs M}_2^1,\cdots,{\bs M}_{a_1^d}^1, \cdots, {\bs M}_1^N, {\bs M}_2^N,\cdots,{\bs M}_{a_N^d}^N \big\}, 
\end{split}
\end{equation}
where the entries of each diagonal block  can be computed by 
\begin{equation}\label{Mnell}
\begin{split}
(\bs M^n_{\ell})_{kj} &= \big(\widecheck{H}_{k, \ell}^{s, n},  \widecheck{H}_{j, \ell}^{s, n}\big)_{\mathbb R^d}=\sum_{p=0}^k(-1)^{k-p}\;{}^s_0\CC^k_p \sum_{q=0}^j (-1)^{j-q}\;{}^s_0\CC^j_q\, \big(\widehat{H}_{p,\ell}^{0,n},\widehat{H}_{q,\ell}^{0,n}\big)_{\mathbb R^d} 
\\ &= (-1)^{k+j}\!\!  \sum_{p=0}^{\min \{j, k\}}\!\!{}^{s}_0\CC^k_p\; {}^s_0\CC^j_p\,.
\end{split}
 \end{equation}
Thus the linear system of \eqref{specf} can be written as
\begin{align}\label{eq:AlgS}
 (\bs I+\gamma\bs M)\bs{u} = \bs{f}.
\end{align}
\begin{remark}{\em
With the new basis at our disposal, the above method  has remarkable advantages over 
the existing Hermite approaches  {\rm(}cf. \cite{mao2017,2018tang}{\rm).} 
Although the usual one-dimensional  Hermite functions are  eigenfunctions of the Fourier transform, 
we observe from  \eqref{eq:LapDF} that the  factor $|\bxi|^{2s}$ is non-separable and singular, 
so the use of tensorial Hermite functions leads to a dense stiffness matrix whose entries are difficult 
to evaluate due to the involved singularity for $d\ge 2$.   \qed
}\end{remark}

 \subsubsection{Error analysis}
 Applying the first Strang lemma \cite{SF1973}  for the standard Galerkin framework (i.e.,  \eqref{uvsh} and \eqref{specf}), we obtain immediately that 
\begin{equation}\label{stranL}
\left\|u-u_{N}\right\|_{H^{s}(\mathbb R^d)} \le c   \inf _{v_{N}\in {\mathcal V}_{\!N}^d}\left\|u-v_{N}\right\|_{H^{s}(\mathbb R^d)}.
\end{equation}
To obtain optimal error estimates, we have to resort to some intermediate approximation results related to certain orthogonal projection.  
To this end, we  consider the  $L^2$-orthogonal projection $\pi_{\!N}^d:L^2(\RR^d)\to {\mathcal V}_{\!N}^d$ such that
\begin{equation}\label{L2projA}
(\pi_{\!N}^d u-u, v)_{\RR^d}=0, \quad \forall v \in {\mathcal V}_{\!N}^d.
\end{equation}
 From  Definition \ref{adjointGHFs} and with a change of basis functions, we find readily that 
\begin{equation}\label{spacefinite}
{\mathcal V}_{\!N}^d := \text{span}\big\{\widehat{H}_{k,\ell}^{0,n}(\bx): \; 0\le n\le N, \; 1\le \ell\le a_n^d,~
0\le 2 k\le N-n,\; k, \ell, n\in {\mathbb N}_0\big\}. 
\end{equation}
Thus,  we can equivalently write 
\begin{equation}\label{definePin}
\pi_{\!N}^du(\bx)= \sum_{n=0}^N \sum_{\ell=1}^{a_n^d}\sum_{k=0}^{[\frac{N-n}{2}]} \widehat u_{k,\ell}^{n}\widehat{H}_{k,\ell}^{0,n}( \bx).
\end{equation}
 


Based on \refe{ghrdiff2}, we introduce the function space  $\H^{r}(\RR^d)$ equipped with the norm
\begin{equation}\label{defHrspace}
\|u\|^2_{\H^r(\RR^d)}=\begin{cases}
\|(-\Delta+|\bx|^2)^mu\|^2_{L^2(\RR^d)},   &r=2m,\\[10pt]
\displaystyle \frac12\Big(\|(\bx+\nabla)(-\Delta+|\bx|^2)^mu\|^2_{L^2(\RR^d)}+\|(\bx-\nabla)(-\Delta+|\bx|^2)^mu\|^2_{L^2(\RR^d)}\Big), & r=2m+1, 
\end{cases}
\end{equation}
where integer $r\geq0$, and $\bx+\nabla$ and $\bx-\nabla$ are the lowering and raising operators, respectively.

The main approximation result is stated below. 
\begin{theorem}\label{approxthm}
Let $s\in(0,1)$. For any $u\in \H^r(\RR^d)$ with integer $r\ge 1$, we have 
\begin{equation} \begin{split}
&\|\pi_{\!N}^d u-u\|_{H^{s}(\RR^d)} \leq (2N+d+2)^{(s-r)/2}\|u\|_{\H^r(\RR^d)}.
\end{split} \end{equation} 
\end{theorem}
\begin{proof}
(i). We first estimate the $L^2$-error.  
For $r=2m+1$, a direct calculation gives 
\begin{equation}\begin{split}\label{etemp1}
\|u\|_{\H^r(\RR^d)}^2&=\frac12\Big(\|(\bx+\nabla)(-\Delta+|\bx|^2)^mu\|^2_{L^2(\RR^d)}+\|(\bx-\nabla)(-\Delta+|\bx|^2)^mu\|^2_{L^2(\RR^d)}\Big)\\&=\big((-\Delta+|\bx|^2)^{m+1}u,(-\Delta+|\bx|^2)^mu\big)_{\RR^d}.
\end{split}\end{equation}
Thanks to the orthogonality \eqref{HerfunOrthA0}, \eqref{ghrdiff2}-\eqref{defHrspace} and \eqref{etemp1},  we have that for any $r\geq 0,$
\begin{equation}\label{deru}
\|u\|^2_{\H^r(\RR^d)}=\sum_{n=0}^\infty \sum_{\ell=1}^{a_n^d}\sum_{k=0}^\infty h^{n,r}_{k,d}|\hat{u}^{n}_{k,\ell}|^{2}, 
\quad h^{n,r}_{k,d}=(4 k+2 n+d)^{r}.
\end{equation}
 Then, we  derive from  \eqref{definePin}  and \eqref{deru}  that
\begin{equation}\label{err1}
\begin{split}
\|\pi_{\!N}^d u-u\|^{2}_{L^2(\RR^d)}&=\sum_{n=0}^\infty  \sum_{\ell=1}^{a_n^d} 
\sum_{k=\lceil\frac{N+1-n}{2}\rceil}^\infty  h^{n,0}_{k, d}  |\hat{u}^{n}_{k,\ell}|^2
 \\& \leq \max _{2k+n\ge N+1}  \left\{\frac{h^{n,0}_{k, d}}{h^{n,r}_{k,d}}\right\}  \sum_{n=0}^\infty \sum_{\ell=1}^{a_n^d}\sum_{k=\lceil\frac{N+1-n}{2}\rceil}^\infty h^{n,r}_{k,d}|\hat{u}^{n}_{k,\ell}|^2
 \\&\leq (2N+2+d)^{-r}\|u\|^2_{\H^r(\RR^d)}.
\end{split}
\end{equation}

If $r=2m, $ we find  from \eqref{defHrspace} that \eqref{etemp1} simply becomes 
\begin{equation}\begin{split}\label{etemp2}
\|u\|_{\H^r(\RR^d)}^2&=\big((-\Delta+|\bx|^2)^{m}u,(-\Delta+|\bx|^2)^mu\big)_{\RR^d},
\end{split}\end{equation}
so we can follow the same lines as above to derive the $L^2$-estimate.

\smallskip 
(ii). We next estimate the $H^1$-error. 
Using the triangle inequality and \refe{etemp1}, we obtain that
\begin{equation}\begin{split}\label{deruerr} 
\|\nabla(\pi_{\!N}^d u&-u)\|^2_{L^2(\RR^d)}\leq  \frac12\Big(\|(\bx+\nabla)(\pi_{\!N}^d u-u)\|^2_{L^2(\RR^d)}+\|(\bx-\nabla)(\pi_{\!N}^d u-u)\|^2_{L^2(\RR^d)}\Big)  
\\&=\left((-\Delta+|\bx|^2)(\pi_{\!N}^d u-u),(\pi_{\!N}^d u-u)\right)_{\RR^d}\leq \sum_{n=0}^\infty \sum_{\ell=1}^{a_n^d}\sum_{k=\lceil\frac{N+1-n}{2}\rceil}^\infty h^{n,1}_{k,d}|\hat{u}^{n}_{k,\ell}|^{2}
\\&\leq \max _{2k+n\geq N+1}\left\{\frac{h^{n,1}_{k,d}}{h^{n,r}_{k,d}}\right\}\sum_{n=0}^\infty \sum_{\ell=1}^{a_n^d}\sum_{k=\lceil\frac{N+1-n}{2}\rceil}^\infty h^{n,r}_{k,d}|\hat{u}^{n}_{k,\ell}|^2 \\&\leq (2N+2+d)^{1-r}\|u\|^2_{\H^r(\RR^d)}. 
 \end{split}\end{equation} 
 
 Finally, the desired results can be obtained by the $L^2$- and $H^1$-bounds derived above and the following space interpolation inequality  (cf. \cite[Ch. 1]{Agranovich2015Book})
 \begin{equation}\label{orthNs22}
 \|u \|_{H^s(\mathbb R^d)}\le\|u \|_{L^2(\mathbb R^d)}^{1-s}\,\|u \|_{H^1(\mathbb R^d)}^{s},\quad s\in (0,1).
\end{equation}
This ends the proof.
\end{proof}

Taking $v_{N}=\pi_{N}^du$ in \refe{stranL} and using Theorem \ref{approxthm}, we immediately obtain the following
error estimate.
\begin{theorem}
Let $u$ and $u_{N}$ be the solutions to \eqref{uvsh} and \eqref{specf}, respectively.
If $u\in \H^r(\RR^d)$ with integer  $r\ge 1$, then we have
\begin{equation}\label{errbound} \begin{split}
&\|u-u_{N}\|_{H^{s}(\RR^d)} \le c (2N+d+2)^{(s-r)/2}\|u\|_{\H^r(\RR^d)},\quad s\in (0,1),
\end{split} \end{equation} 
where $c$ is a positive constant independent of $N$ and $u.$
\end{theorem}

\subsubsection{Numerical results}
We conclude this section with some numerical results.  For the convenience of implementation, we fix the degree of the numerical solution in both radial and angular direction in \eqref{uexpansion}, so the numerical solution takes the form
\begin{equation}\label{unk}
u_{\!N,K}(\bx)= \sum_{n=0}^{N} \sum_{\ell=1}^{a_n^d}\sum_{k=0}^{K} \hat u_{k,\ell}^{n}\,\widecheck{H}_{k,\ell}^{s,n}(\bx).
\end{equation}
Here, we focus on $d=2,3.$

 \begin{figure}[!th]
\subfigure[$d=2$ and $u_e(\bx)=e^{-|\bx|^2}$]{
\begin{minipage}[t]{0.42\textwidth}
\centering
\rotatebox[origin=cc]{-0}{\includegraphics[width=0.85\textwidth]{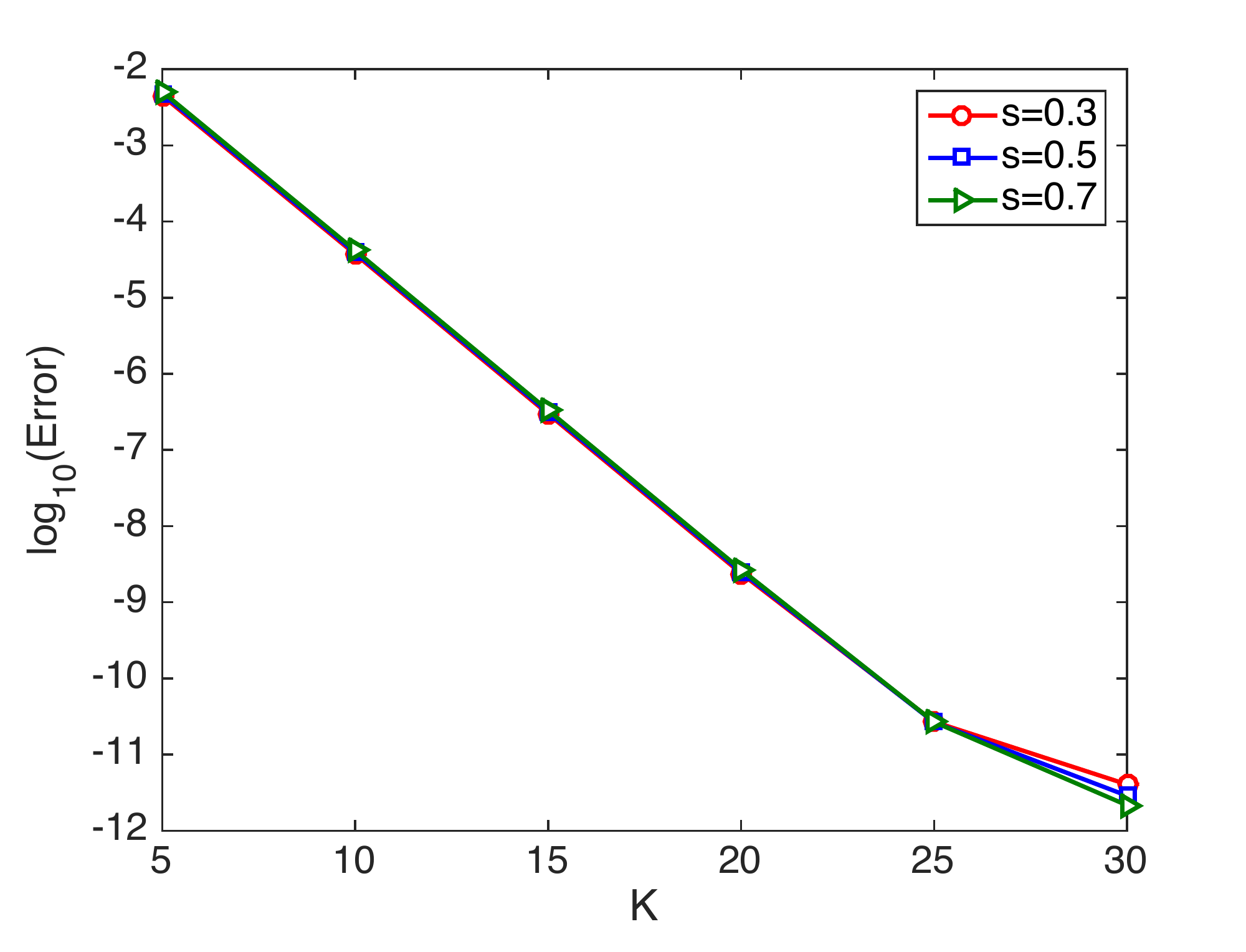}}
\end{minipage}}
\subfigure[$d=2$ and $u_a(\bx)=(1+|\bx|^2)^{-2}$]{
\begin{minipage}[t]{0.42\textwidth}
\centering
\rotatebox[origin=cc]{-0}{\includegraphics[width=0.85\textwidth]{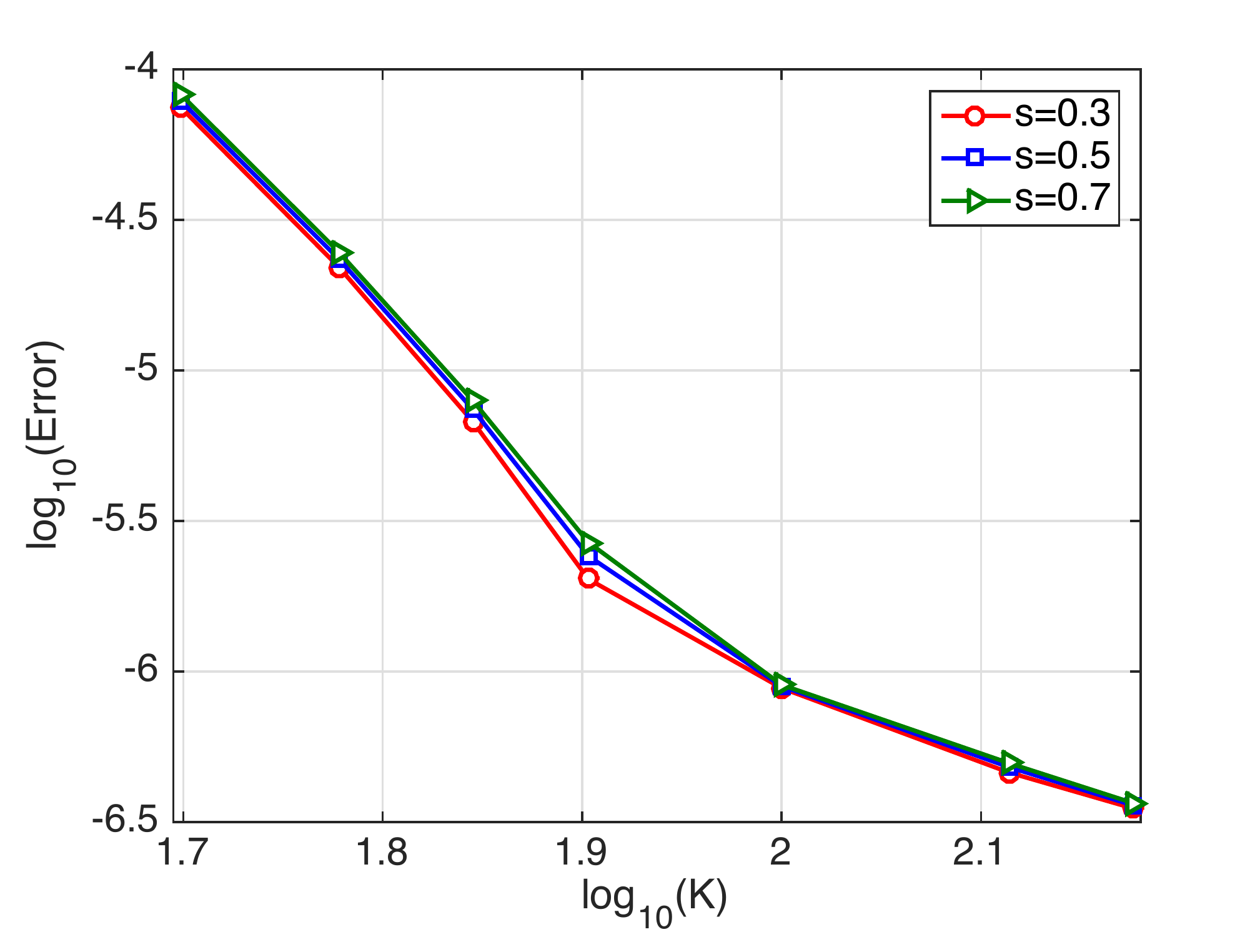}}
\end{minipage}}\vskip -5pt

\subfigure[$d=3$ and $u_e(\bx)=e^{-|\bx|^2}$]{
\begin{minipage}[t]{0.42\textwidth}
\centering
\rotatebox[origin=cc]{-0}{\includegraphics[width=0.85\textwidth]{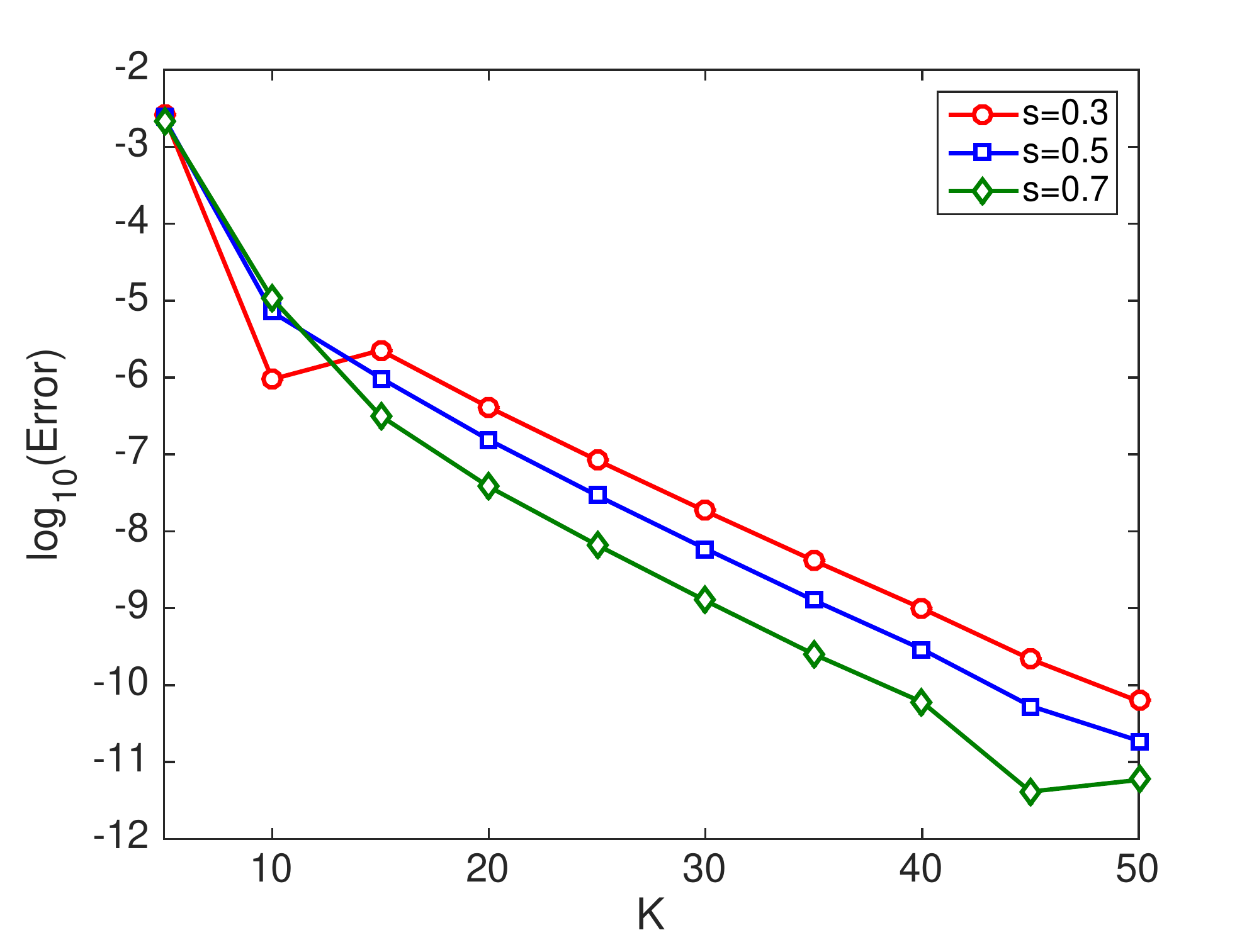}}
\end{minipage}}
\subfigure[$d=3$ and $u_a(\bx)=(1+|\bx|^2)^{-2}$]{
\begin{minipage}[t]{0.42\textwidth}
\centering
\rotatebox[origin=cc]{-0}{\includegraphics[width=0.85\textwidth]{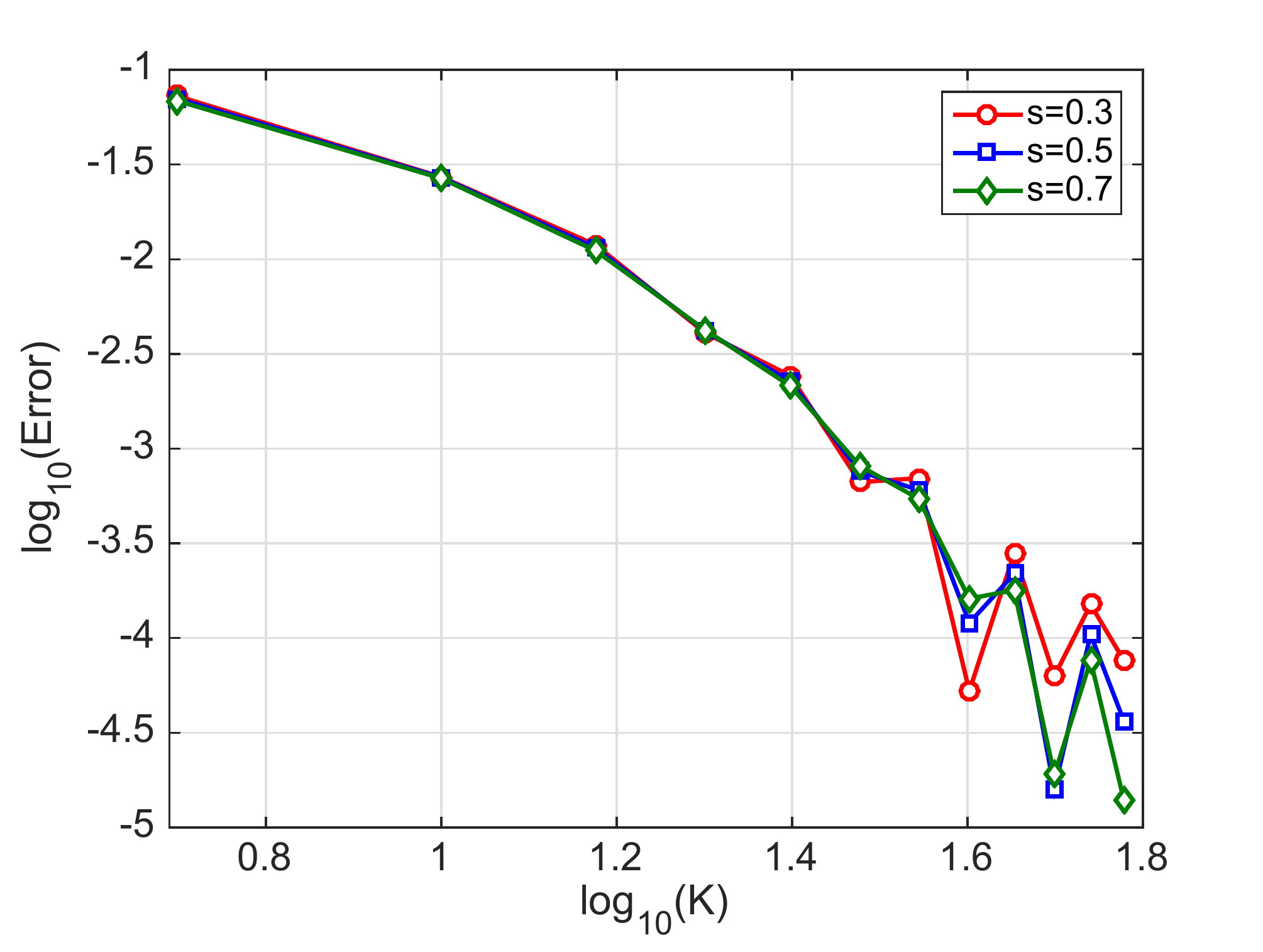}}
\end{minipage}}
\caption
{\small  The maximum errors of the GHF-spectral-Galerkin scheme with $\gamma=1$ for Example \ref{some_exact}  with exact solutions in \eqref{test1}. Here  $s=0.3,~0.5,~0.7$.}\label{Example1}
\end{figure}

  \begin{figure}[!th]
\subfigure[$d=2$ with given source term $f_e(\bx)$]{
\begin{minipage}[t]{0.42\textwidth}
\centering
\rotatebox[origin=cc]{-0}{\includegraphics[width=0.85\textwidth]{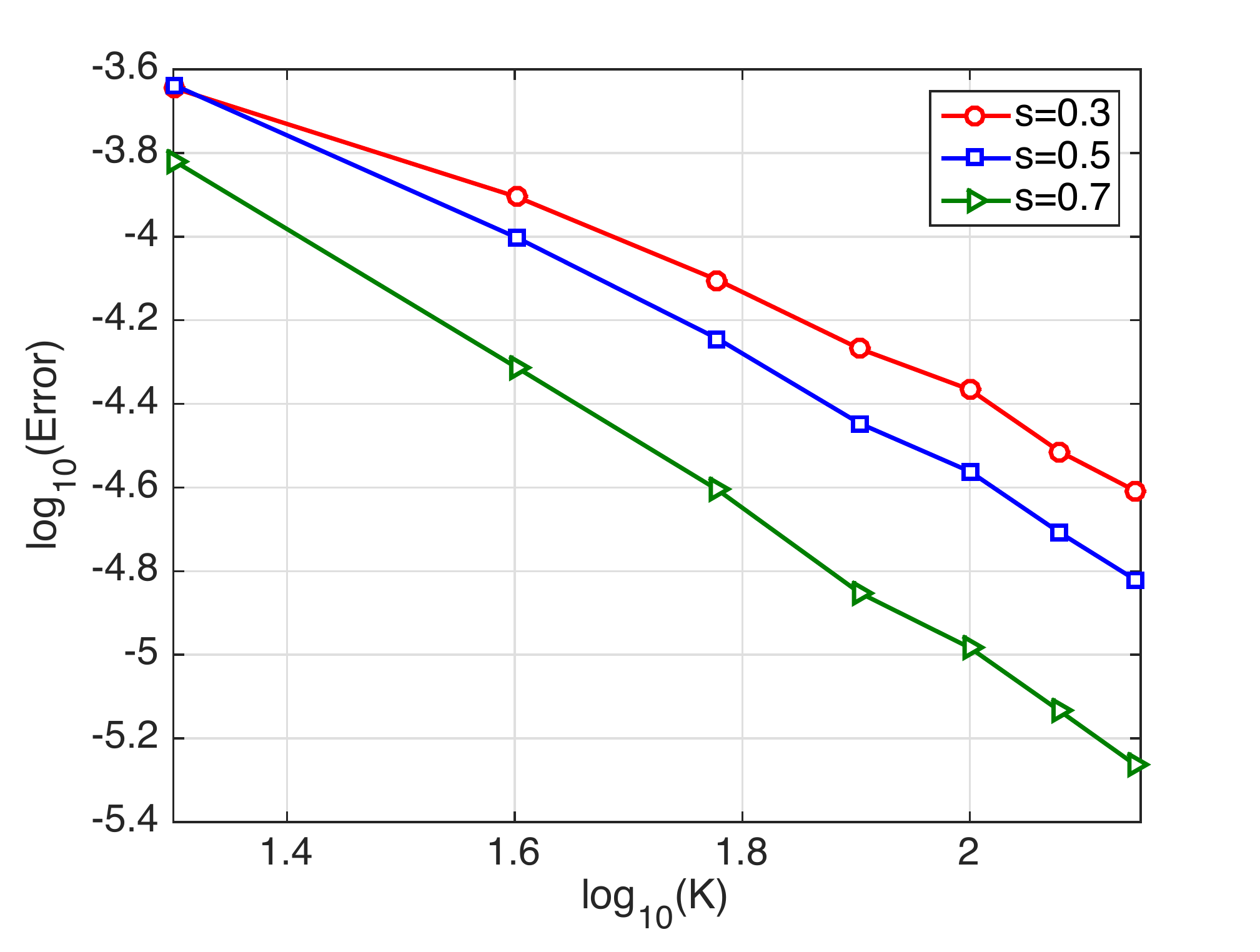}}
\end{minipage}}
\subfigure[$d=2$ with given source term $f_a(\bx)$]{
\begin{minipage}[t]{0.42\textwidth}
\centering
\rotatebox[origin=cc]{-0}{\includegraphics[width=0.85\textwidth]{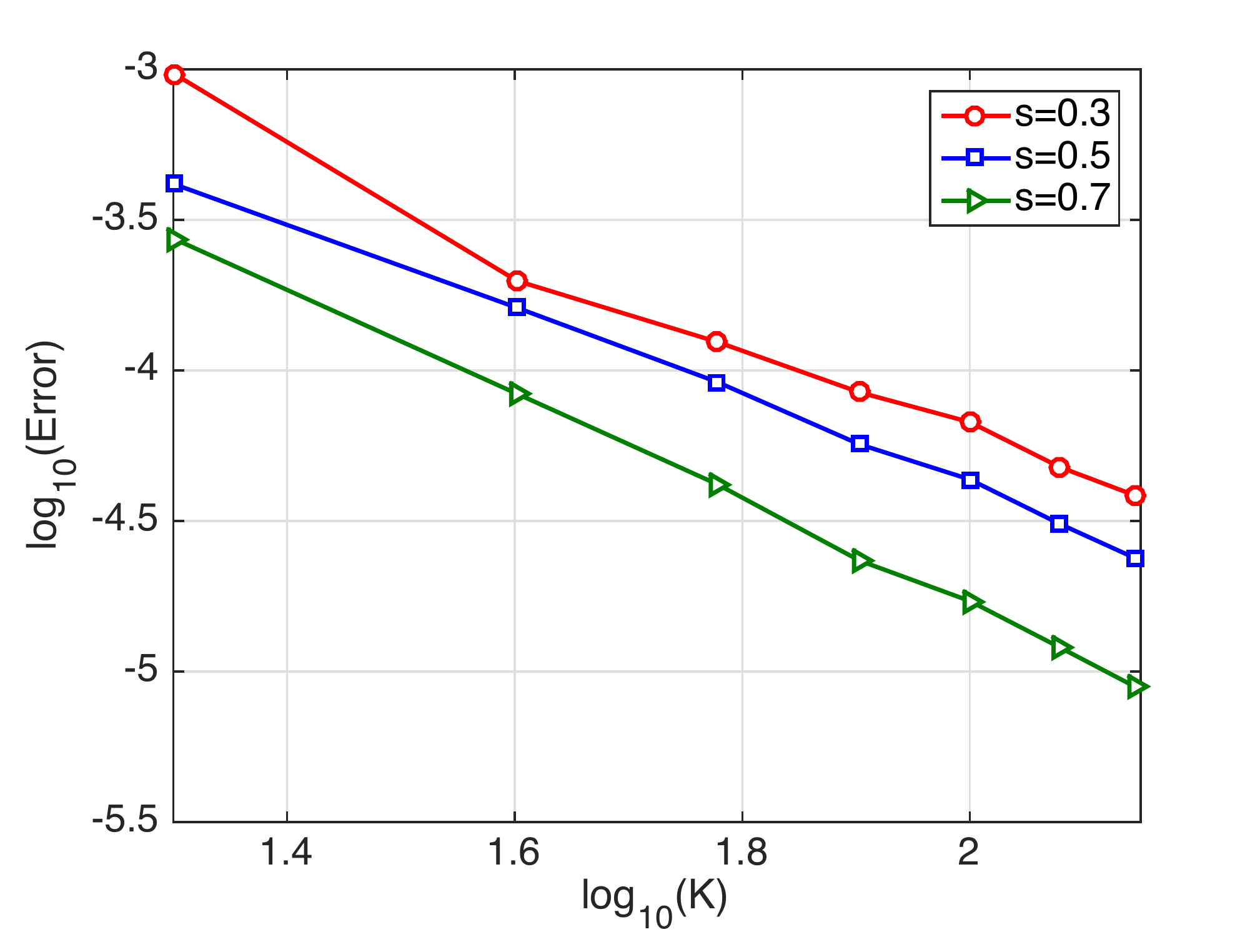}}
\end{minipage}}\vskip -5pt

\subfigure[$d=3$ with given source term $f_e(\bx)$]{
\begin{minipage}[t]{0.42\textwidth}
\centering
\rotatebox[origin=cc]{-0}{\includegraphics[width=0.85\textwidth]{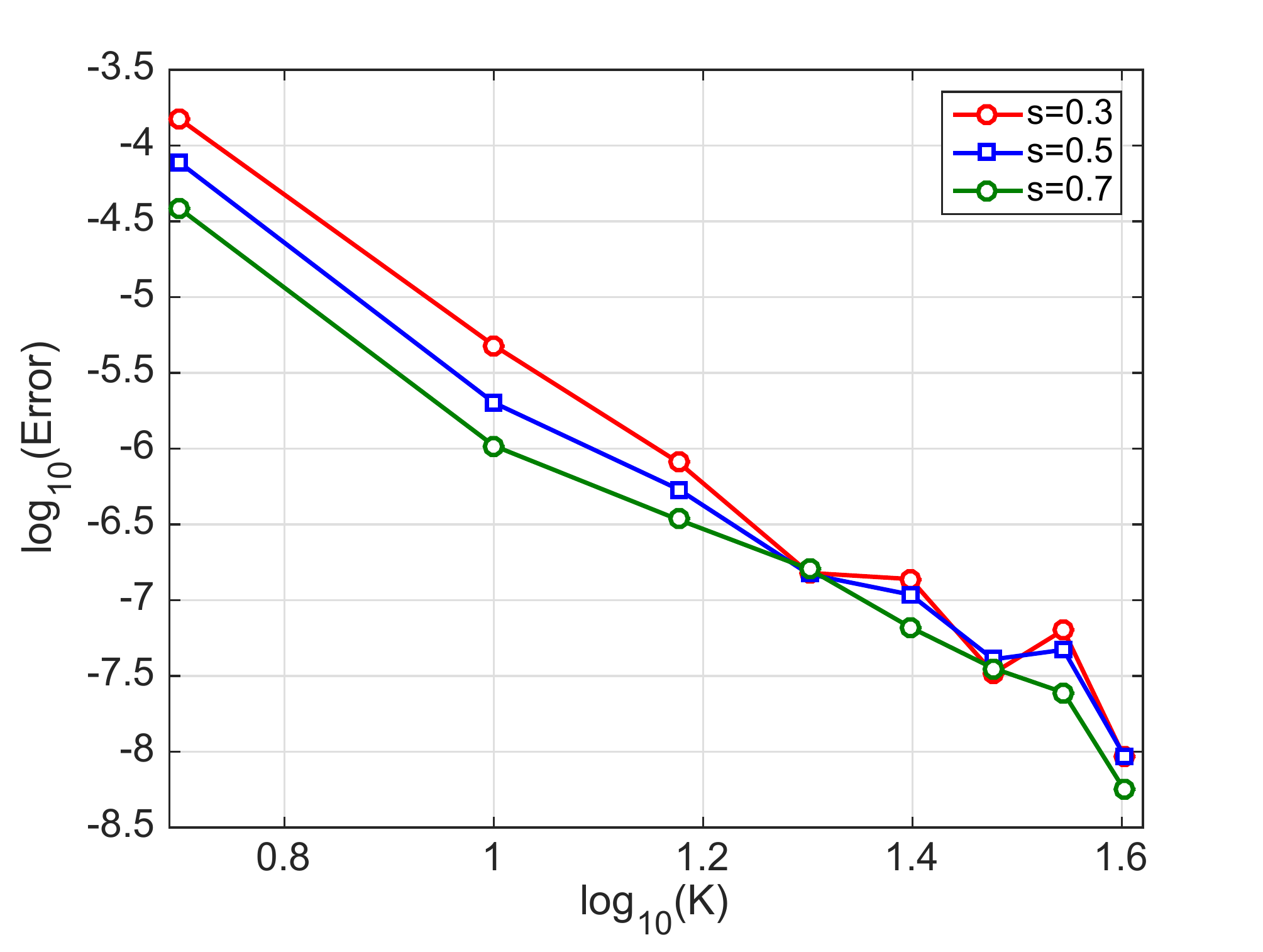}}
\end{minipage}}
\subfigure[$d=3$ with given source term $f_a(\bx)$]{
\begin{minipage}[t]{0.42\textwidth}
\centering
\rotatebox[origin=cc]{-0}{\includegraphics[width=0.85\textwidth]{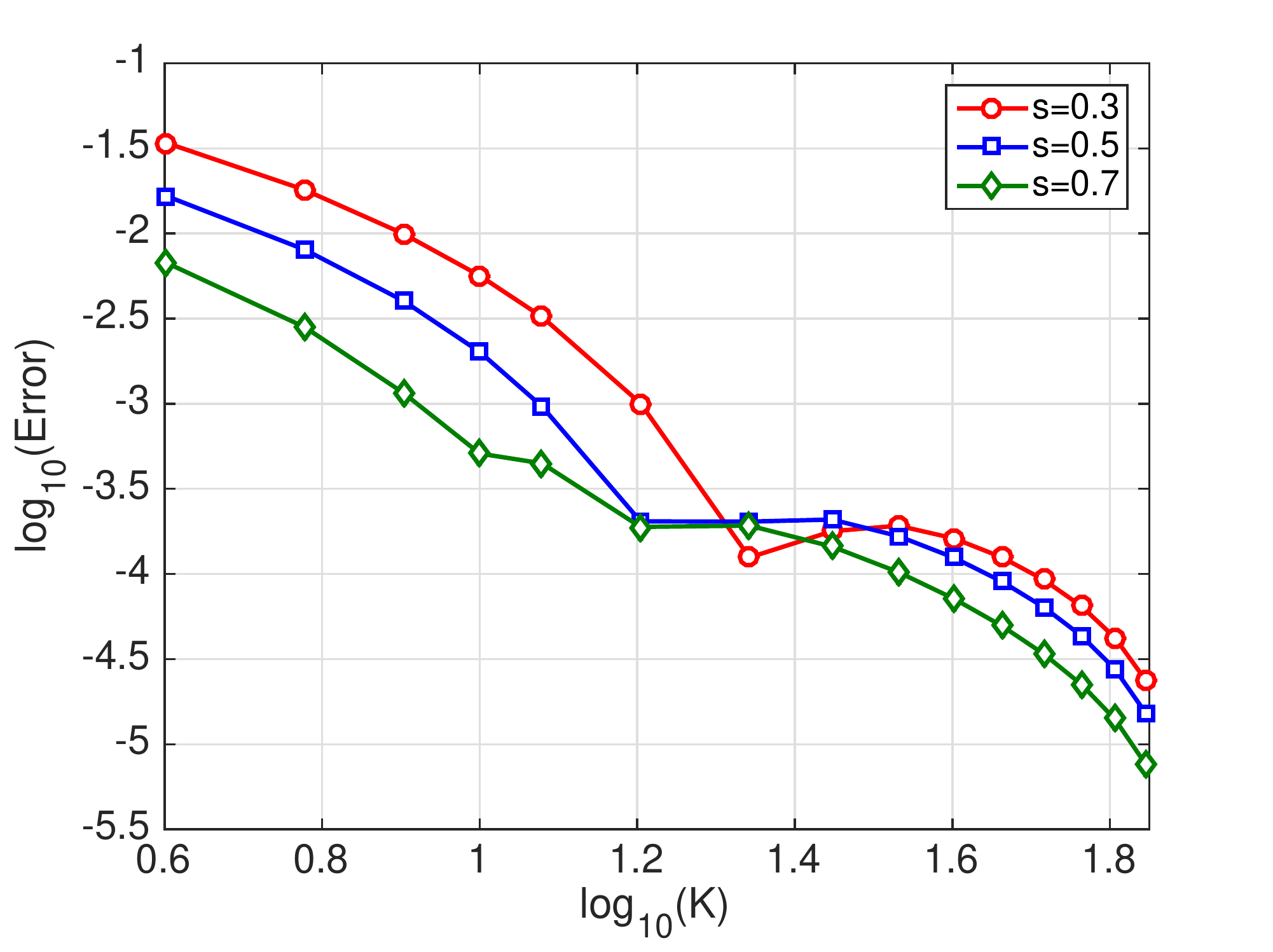}}
\end{minipage}}\vskip -5pt
\caption
{\small  The maximum errors of the GHF-spectral-Galerkin scheme with $\gamma=1$ for Example \ref{some_source}  with given source functions in \eqref{test2}. Here  $s=0.3,~0.5,~0.7$ and $r=2$.}\label{fig1dsource}
\end{figure}

\begin{example}\label{some_exact} {\bf (Problem \eqref{soureceprob} with exact solution).} We first consider \eqref{soureceprob} with the following exact solutions:
 \begin{equation}\begin{split}\label{test1}
 &u_e(\bx)=\e^{-|\bx|^2},\quad u_a(\bx)=(1+|\bx|^2)^{-r},\quad r>0,\;\; \bx\in {\mathbb R^d}.
  \end{split}\end{equation}
According to \cite[Prop.\! 4.2 \& 4.3]{sheng2019}, the source terms $f_e(\bx)$ and $f_a(\bx)$ are respectively given by
\begin{eqnarray}\label{fasy}
&&f_e(\bx)= \gamma \e^{-|\bx|^2}+\frac{ 2^{2s} \Gamma(s+d/2)}{\Gamma(d/2)}\, {}_{1} F_{1}\Big(s+\frac{d}{2};\frac{d}{2};-|\bx|^2\Big) ,\nonumber
\\&&f_a(\bx)= \gamma (1+|\bx|^2)^{-r}+\frac{2^{2s}\Gamma(s+r)\Gamma(s+d/2)}{ \Gamma(r)\Gamma(d/2)} {}_2F_1\Big(s+r,s+\frac{d}{2};\frac{d}{2};-|\bx|^2\Big).\nonumber
\end{eqnarray}

For $d=2,3$, we take $s=0.3,~0.5,~0.7$ and the degree in angular direction is fixed $N\equiv 10$ (see \eqref{unk}). In Figure \ref{Example1} (c)-(f), we plot the maximum errors, in semi-log scale and log-log scale,  for $u_e$ and $u_a$ with $d=2,3$ against various $K$, respectively.
As expected, we observe the exponential and algebraic  convergence for $u_e$ and $u_a, $ respectively.
\end{example}

\begin{example}\label{some_source} {\bf (Problem \eqref{soureceprob} with a source term).} We next consider \eqref{soureceprob} with the following source functions:
 \begin{equation}\begin{split}\label{test2}
 &f_e(\bx)=\sin(|\bx|)\e^{-|\bx|^2},\quad f_a(\bx)=\cos(|\bx|)(1+|\bx|^2)^{-r},\quad r>0,\;\;\bs x\in {\mathbb R^d}.
  \end{split}\end{equation}

\end{example}

The exact solutions are unknown, and we use the numerical solution with $K=80$, $N=20$  as the reference solution.
 For $d=2,3$, we plot the maximum errors, in log-log scale, for \eqref{soureceprob} against various $K$ in Figure \ref{fig1dsource} (c)-(f), which we take $s=0.3,~0.5,~0.7$ and fix $N\equiv 10$.  As shown in \cite{sheng2019},
  the solution of \eqref{soureceprob}  decays algebraically, even for exponentially decaying  source terms. Indeed, we observe an algebraic order of convergence.

  \subsection{GHF-spectral-Galerkin method for fractional Schr\"{o}dinger equations}\label{fractionalfNLS}
   
 As a  second example,  we consider the fractional Schr\"{o}dinger equation:
\begin{equation}\label{fracGPE}
\begin{aligned}
&{\rm i} \partial_{t} \psi(\bs x, t)=\Big[\frac12(-\Delta)^s+\frac{\gamma^{2}}{2}  |\bs x|^{2\mu}\Big] \psi(\bs x, t),\;\;\; \bs x\in \RR^2,\;\; t>0,\\&
\psi(\bs x,0)=\psi_0(\bs x),\;\;\bs x\in \RR^2; \quad \psi(\bs x,t)\to 0\;\; {\rm as}\;\;  |\bs x|\to\infty, \;\;\; t\ge0, 
\end{aligned}
\end{equation}
 where $s\in (0,1]$, $\mu>-1/2$,  the constant $\gamma>0$, and the function
 $\psi_0$  is  given. 
Here,  we focus on the linear equation. Indeed,  using a suitable time-splitting scheme, 
one only needs to solve a linear Schr\"{o}dinger equation at each time step for some typical nonlinear cases 
(see, e.g., \cite{BLS2009}). 
 We remark that  the fractional Schr\"odinger equation   \eqref{fracGPE} is the model of interest  in the study of 
   fractional quantum mechanics, see \cite{laskin2000fractional,zhang2015}, where in \cite{laskin2000fractional}, 
   this fractional Hamiltonian appeared  more reasonable   to study the problem of quarkonium.

To solve \eqref{fracGPE} efficiently, we adopt the A-GHFs spectral method in space and the 
Crank-Nicolson scheme in time discretization.   Let $\Delta t$ be the time-stepping size, and $\psi^k(\bs x)\approx \psi(\bs x,k\Delta t).$ Then we look for  $\psi^{n+1}\in H^s(\RR^2)$ such that 
\begin{equation}\label{psiv}
\begin{split}
{\rm i}\Big(\frac{\psi^{n+1}-\psi^{n}}{\Delta t},v\Big)_{\mathbb R^2}=\frac12\big((-\Delta)^{\frac{s}{2}}\psi^{n+\frac12},(-\Delta)^{\frac{s}{2}}v\big)_{\mathbb R^2}+ \frac{\gamma^{2}}{2} (|\bs x|^{2\mu}\psi^{n+\frac12},v)_{\mathbb{R}^2},\quad \forall v\in  H^s(\mathbb R^2),
\end{split}
\end{equation}
where $\psi^{n+\frac12}=(\psi^{n+1}+\psi^{n})/2$.   We can implement the GHF-spectral scheme as with the problem \eqref{specf}, 
but only need to evaluate the matrix  $\bs V$ associated with the potential $|\bs x|^{2\mu}$.  It is a block diagonal matrix
\begin{equation}\label{Mmatrix00}
\begin{split}
&\bs V= \mathrm{diag}\big\{\bs V^0_{1}, \bs V^0_{2},  \cdots,{\bs V}_{a_0^d}^0, {\bs V}_1^1,{\bs V}_2^1,\cdots,{\bs V}_{a_1^d}^1, \cdots, {\bs V}_1^N, {\bs V}_2^N,\cdots,{\bs V}_{a_N^d}^N \big\}, 
\end{split}
\end{equation}
and the entries of each diagonal block can be evaluated explicitly  by using \eqref{HerfunOrthA0},   \eqref{coefrd1} and \eqref{adher}:  
\begin{equation}\label{Mnell00}
\begin{split}
(\bs V^n_{\ell})_{kj} &= \big(|\bs x|^{2\mu}\widecheck{H}_{k, \ell}^{s, n},  \widecheck{H}_{j, \ell}^{s, n}\big)_{\mathbb R^d}=\sum_{p=0}^k(-1)^{k-p}\;{}^s_0\CC^k_p \sum_{q=0}^j (-1)^{j-q}\;{}^s_0\CC^j_q\, \big(|\bs x|^{2\mu}\widehat{H}_{p,\ell}^{0,n},\widehat{H}_{q,\ell}^{0,n}\big)_{\mathbb R^d} 
\\&=\sum_{p=0}^k(-1)^{k-p}\;{}^s_0\CC^k_p \sum_{q=0}^j (-1)^{j-q}\;{}^s_0\CC^j_q\, \sum_{p^\prime=0}^p\;{}^0_\mu\CC^p_{p^\prime} \sum_{q^\prime=0}^q\;{}^0_\mu\CC^q_{q^\prime}\, \big(|\bs x|^{2\mu}\widehat{H}_{p^\prime,\ell}^{\mu,n},\widehat{H}_{q^\prime,\ell}^{\mu,n}\big)_{\mathbb R^d} 
\\ &= \sum_{p=0}^k(-1)^{k-p}\;{}^s_0\CC^k_p \sum_{q=0}^j (-1)^{j-q}\;{}^s_0\CC^j_q\, \sum_{p^\prime=0}^p\;{}^0_\mu\CC^p_{p^\prime} \sum_{q^\prime=0}^q\;{}^0_\mu\CC^q_{q^\prime}.
\end{split}
 \end{equation}

To test the accuracy of  the proposed method, we add an external source term $f(\bs x,t)$ so that  the exact solution is $\psi(\bs x,t)={\rm e}^{-|\bs x|^2-t}$.
 In Figure \ref{fignls} (a), we plot the maximum errors versus  $\Delta t$ at $t=1$, and the second-order convergence is observed.  Here we take $\gamma=1$, $N=10,$ $K=50$   and different $s,\mu$. 
We choose the time stepping size  to be  small so that the error is dominated by the spatial error. In Figure \ref{fignls} (b), we plot maximum errors in the semi-log scale versus various $K$, for which we take $N=10$, $\gamma=1$ and different $s,\mu$.  We observe that the spatial errors decay exponentially as $K$ increases.
 
  \begin{figure}[!th] 
\subfigure[Temporal errors]{
\begin{minipage}[t]{0.42\textwidth}
\centering
\rotatebox[origin=cc]{-0}{\includegraphics[width=1.0\textwidth,height=0.75\textwidth]{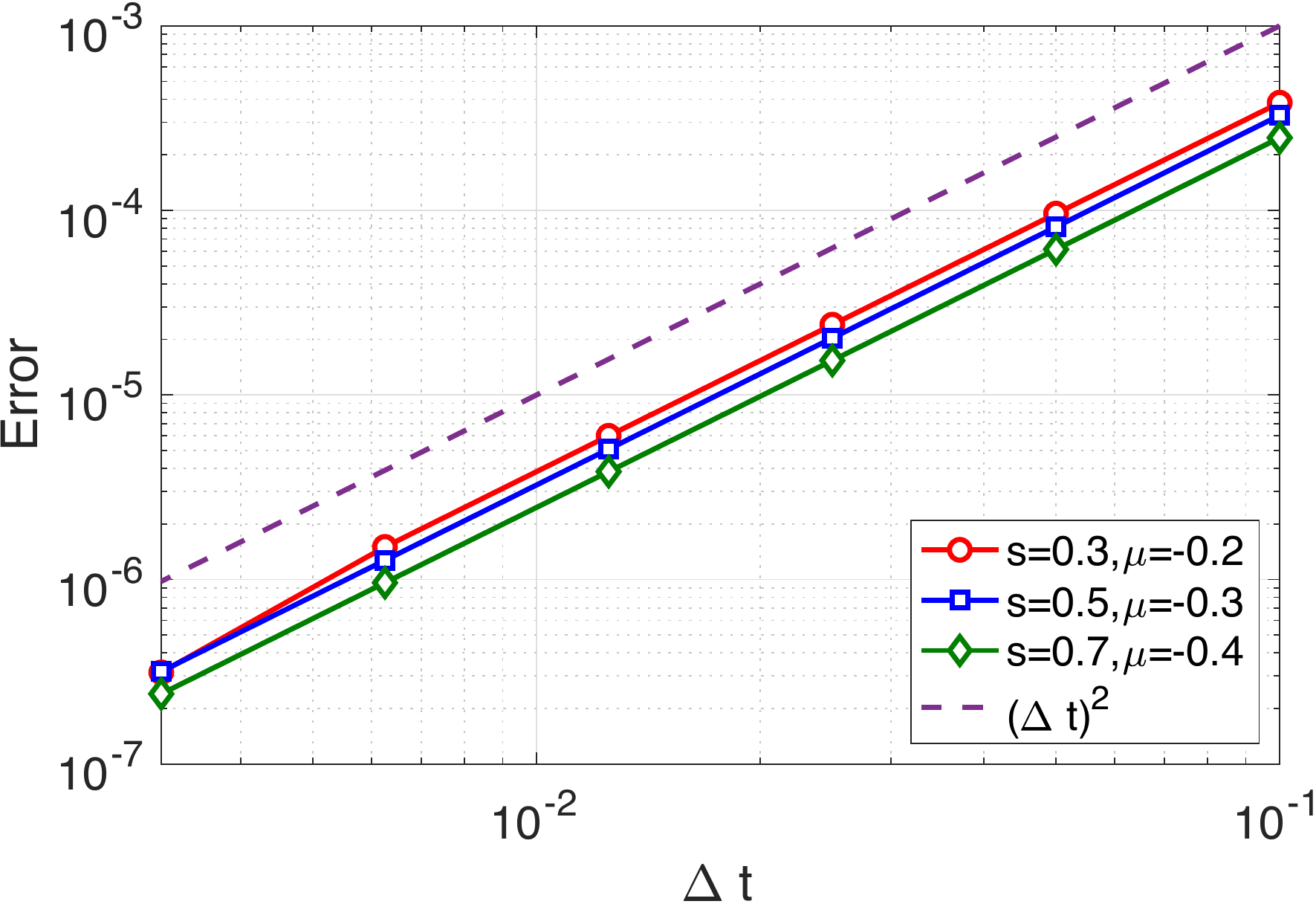}}
\end{minipage}}
\subfigure[Spatial errors]{
\begin{minipage}[t]{0.42\textwidth}
\centering
\rotatebox[origin=cc]{-0}{\includegraphics[width=1.0\textwidth,height=0.75\textwidth]{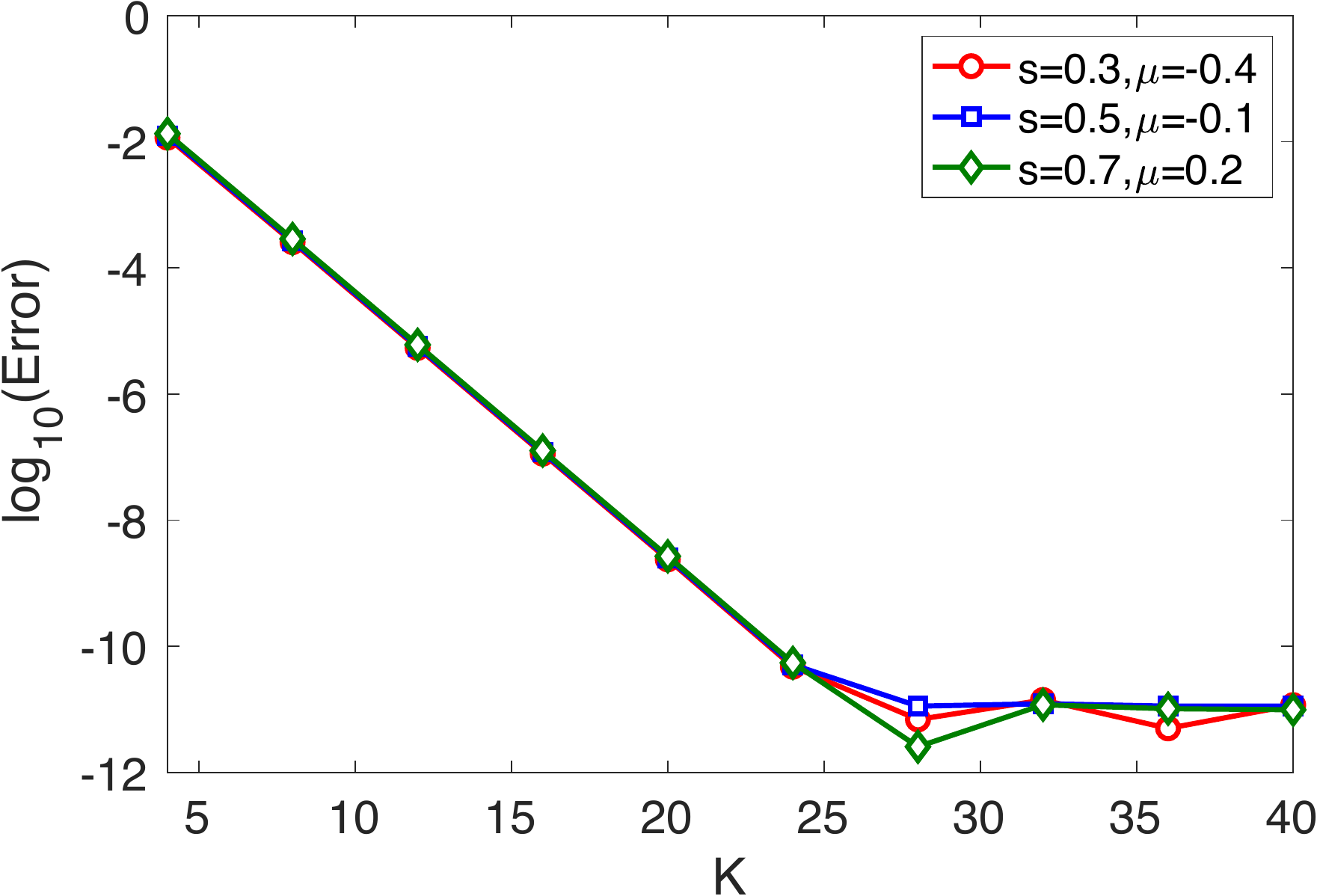}}
\end{minipage}}\vskip -5pt
 \caption
{\small  (a): The temporal errors for Crank-Nicolson scheme at $t=1$ with different $s,\mu$; (b): The spatial errors of GHF-spectral-Galerkin method at $t=1$ with different $s,\mu$.}\label{fignls}
\end{figure}
 
 
 Next, we investigate the dynamics of beam propagations as in \cite{zhang2015} (where the case $\mu=1$ was considered). 
 We take the following incident Gaussian beam as the initial condition: 
 \begin{equation}
 \psi(\bs x, 0)=\psi_0(\bs x)={\rm e}^{-\sigma|\bs x|^2-{\rm i}C|\bs x|},
 \end{equation}
 where the constants $\sigma$ and $C$ are  the beam width and the linear chirp coefficient, respectively.  
In the test, we take $\sigma=C=1$.
 In Figure \ref{Example3}, we depict the profiles of the real part of the numerical solutions for various  $s,\mu$ at $t=2$. 
 Figure \ref{Example3} (a)  shows the solution profile of the usual case with a harmonic potential: $-\Delta +|\bs x|^2$ for comparison.  We observe from the other profiles  that the solutions have different peak intensities and singular behaviours, from which we find the smaller the value of $\mu,$ and the stronger the singularity.  In fact, some similar observations was made 
 in \cite{zhang2015} for the case with $\mu=1.$
    
%
%
 \begin{figure}[!th]\hspace{-15pt}
\subfigure[$s=1$ and $\mu=1$]{
\begin{minipage}[t]{0.28\textwidth}
\centering
\rotatebox[origin=cc]{-0}{\includegraphics[width=0.9\textwidth]{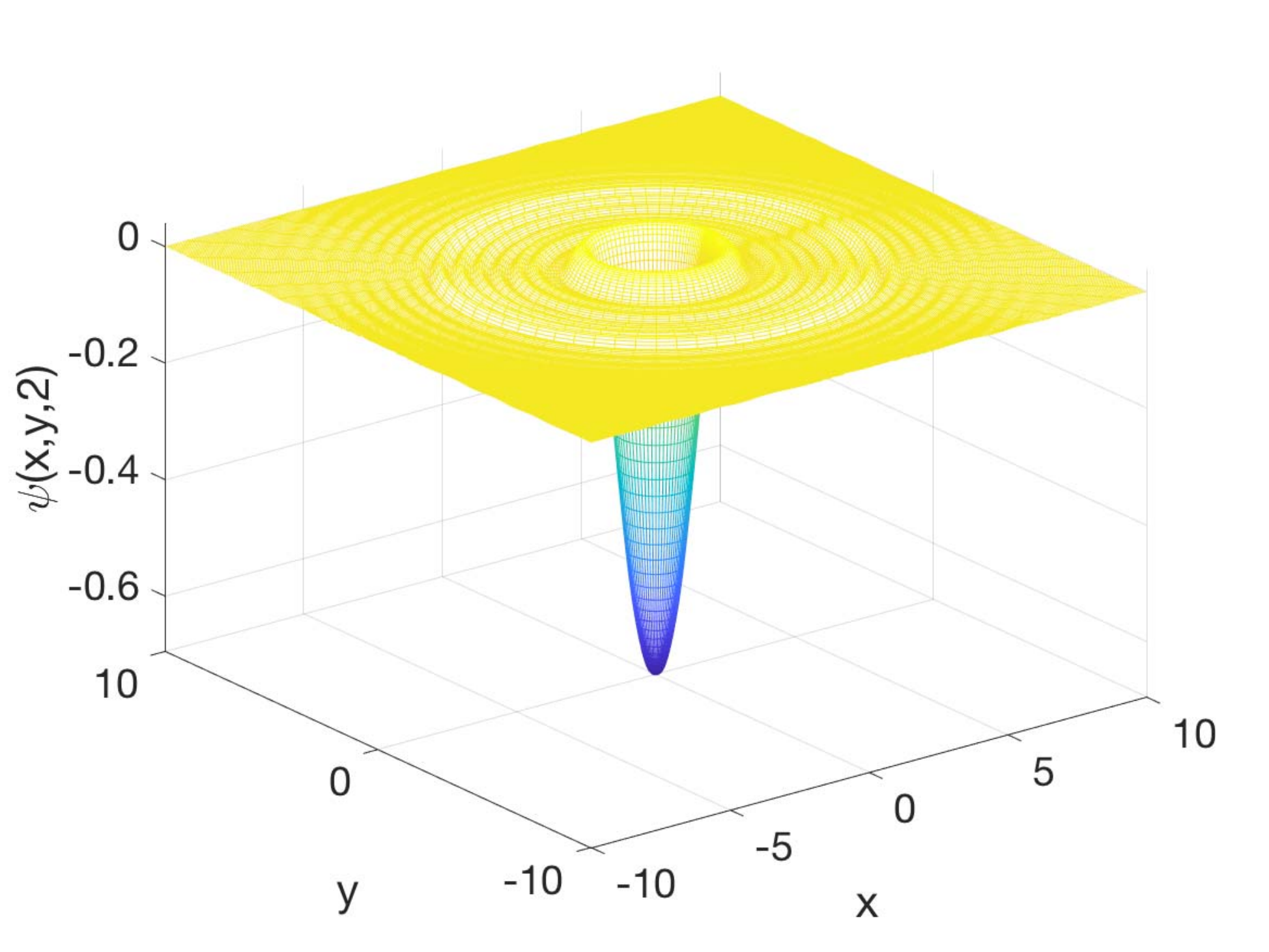}}
\end{minipage}}\hspace{-24pt}
\subfigure[$s=1$ and $\mu=0.7$]{
\begin{minipage}[t]{0.28\textwidth}
\centering
\rotatebox[origin=cc]{-0}{\includegraphics[width=0.9\textwidth]{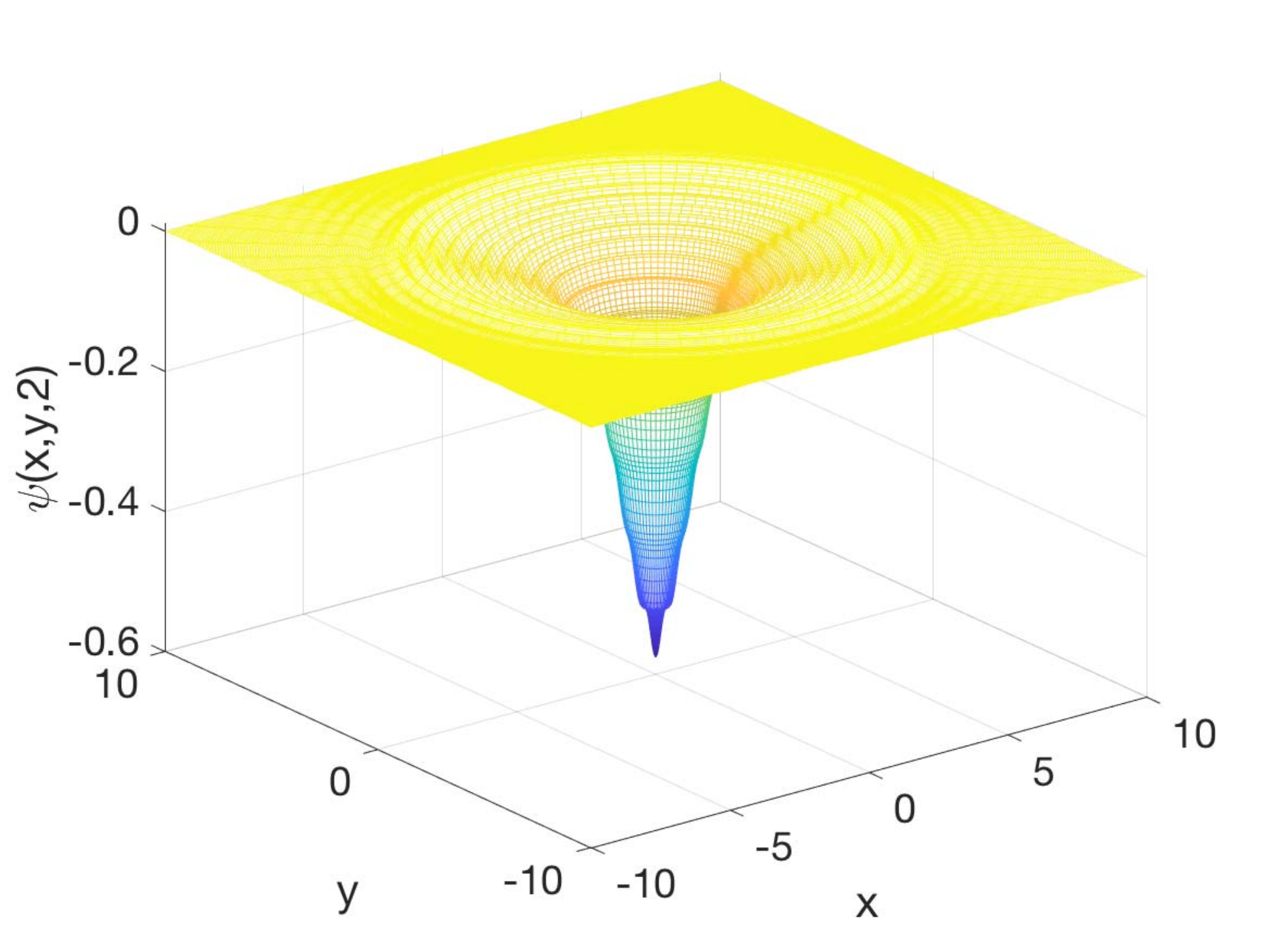}}
\end{minipage}}\hspace{-24pt}
\subfigure[$s=1$ and $\mu=0.3$]{
\begin{minipage}[t]{0.28\textwidth}
\centering
\rotatebox[origin=cc]{-0}{\includegraphics[width=0.9\textwidth]{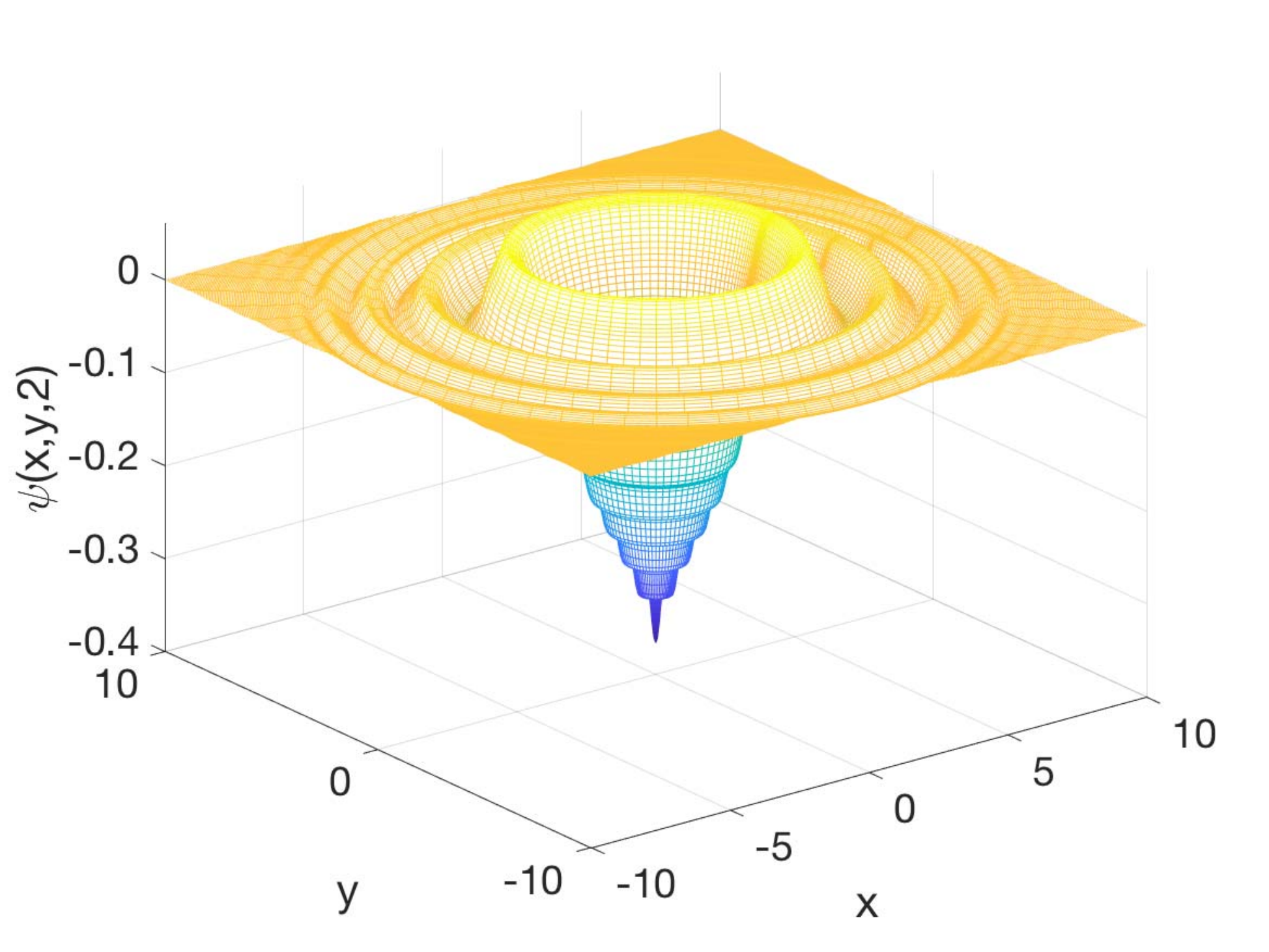}}
\end{minipage}}\hspace{-24pt}
\subfigure[$s=1$ and $\mu=-0.3$]{
\begin{minipage}[t]{0.28\textwidth}
\centering
\rotatebox[origin=cc]{-0}{\includegraphics[width=0.9\textwidth]{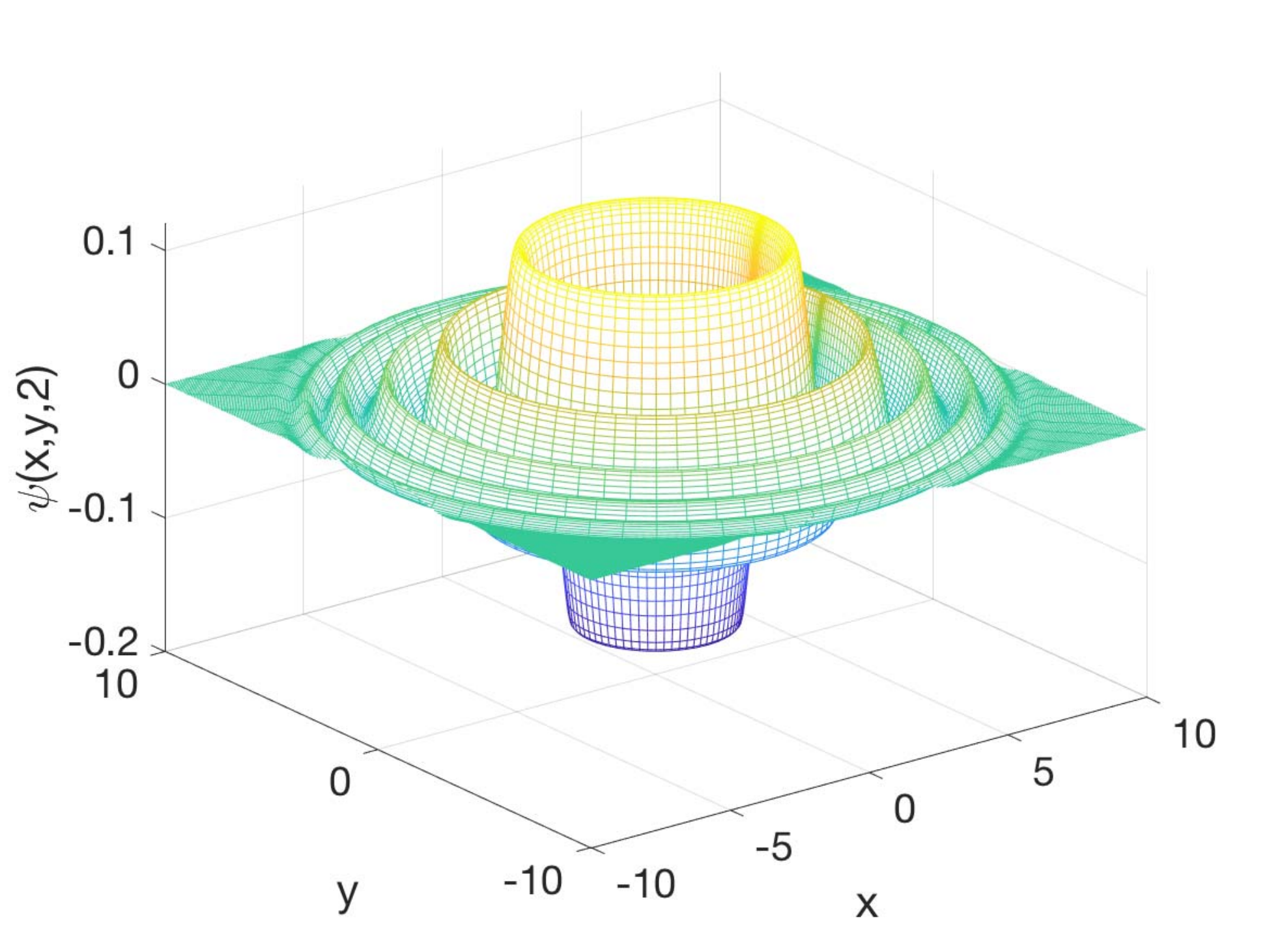}}
\end{minipage}}\vskip -5pt

\hspace{-15pt}
\subfigure[$s=0.7$ and $\mu=1$]{
\begin{minipage}[t]{0.28\textwidth}
\centering
\rotatebox[origin=cc]{-0}{\includegraphics[width=0.9\textwidth]{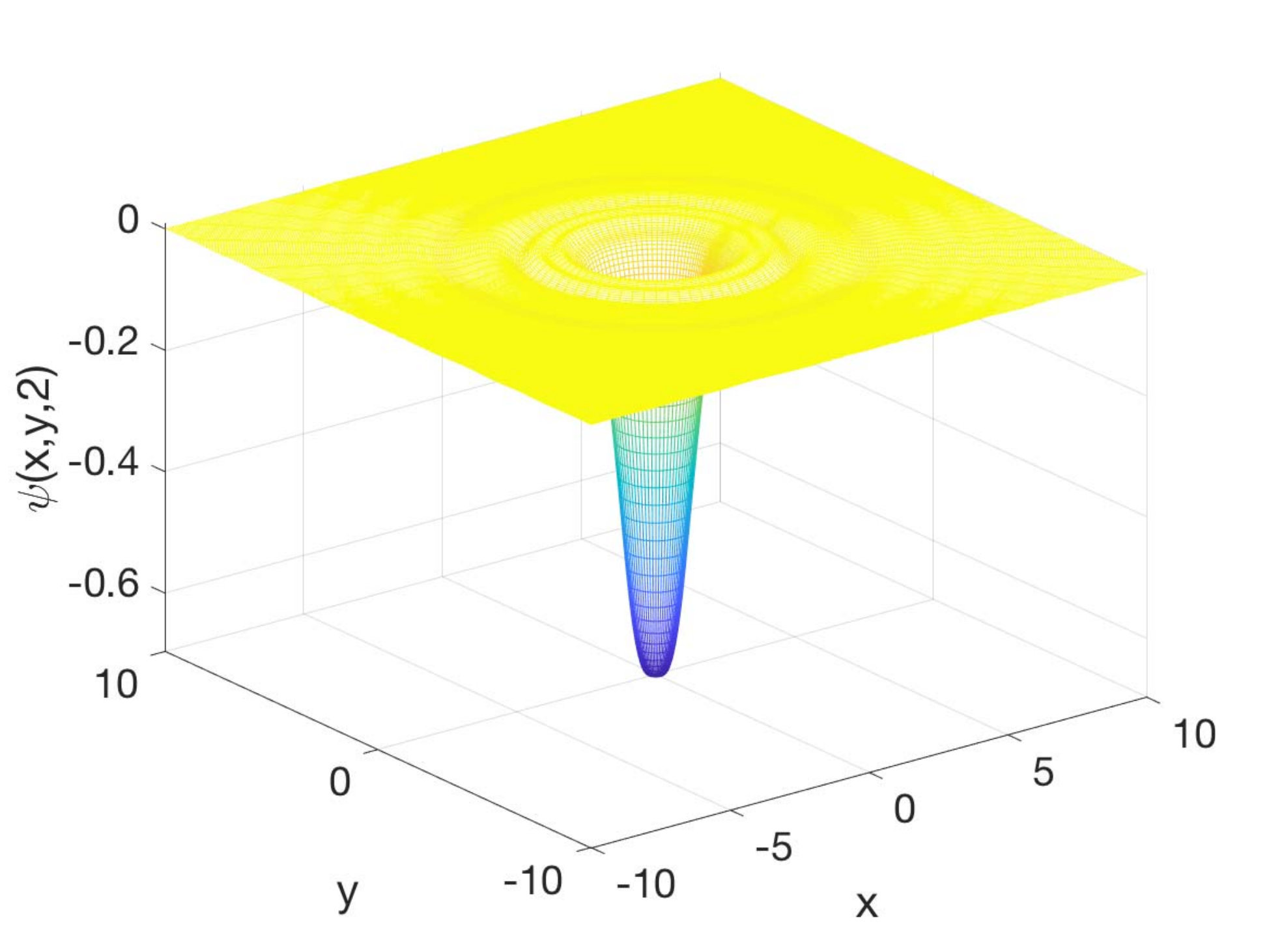}}
\end{minipage}}\hspace{-24pt}
\subfigure[$s=0.7$ and $\mu=0.7$]{
\begin{minipage}[t]{0.28\textwidth}
\centering
\rotatebox[origin=cc]{-0}{\includegraphics[width=0.9\textwidth]{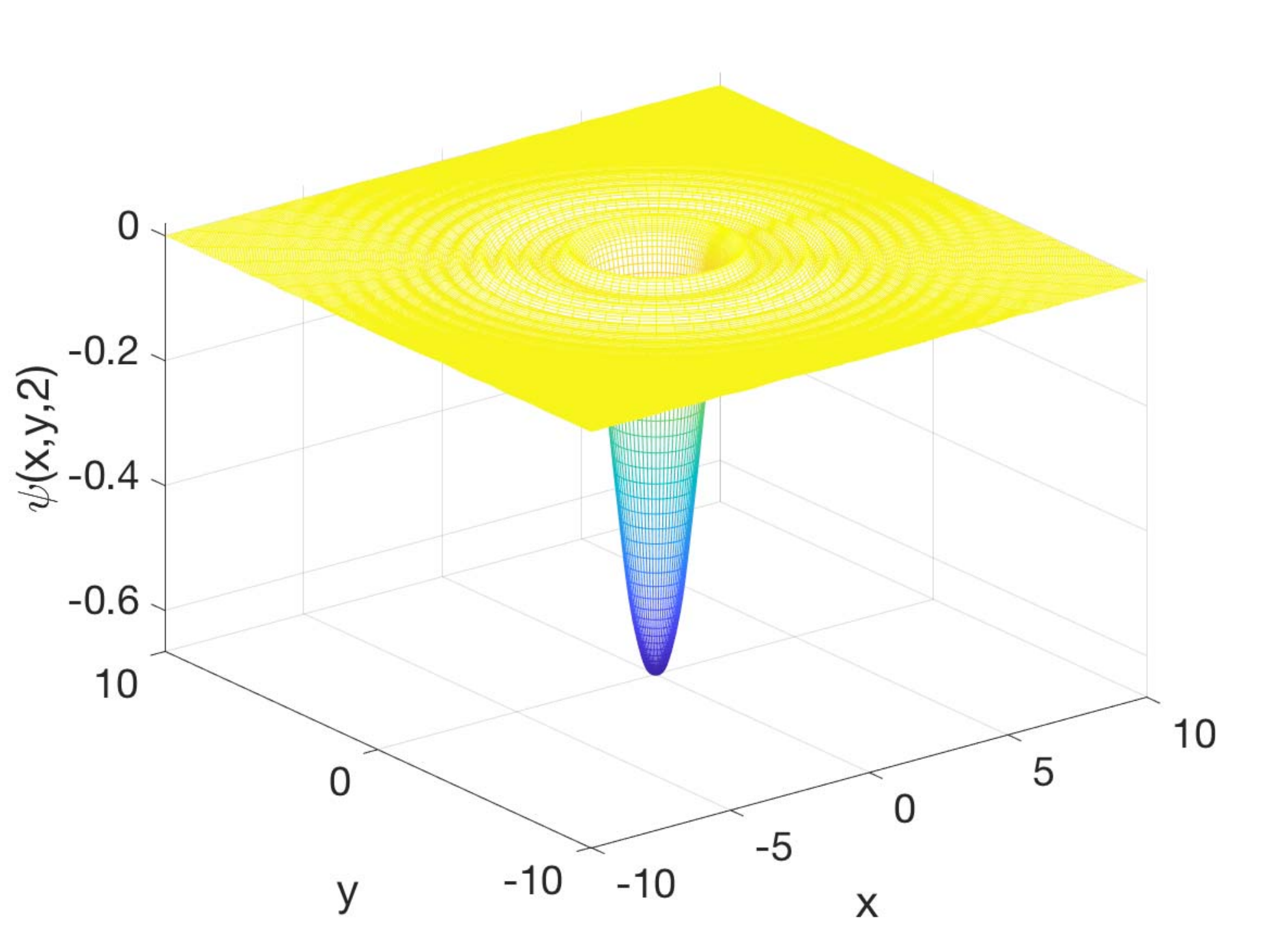}}
\end{minipage}} \hspace{-24pt}
\subfigure[$s=0.7$ and $\mu=0.3$]{
\begin{minipage}[t]{0.28\textwidth}
\centering
\rotatebox[origin=cc]{-0}{\includegraphics[width=0.9\textwidth]{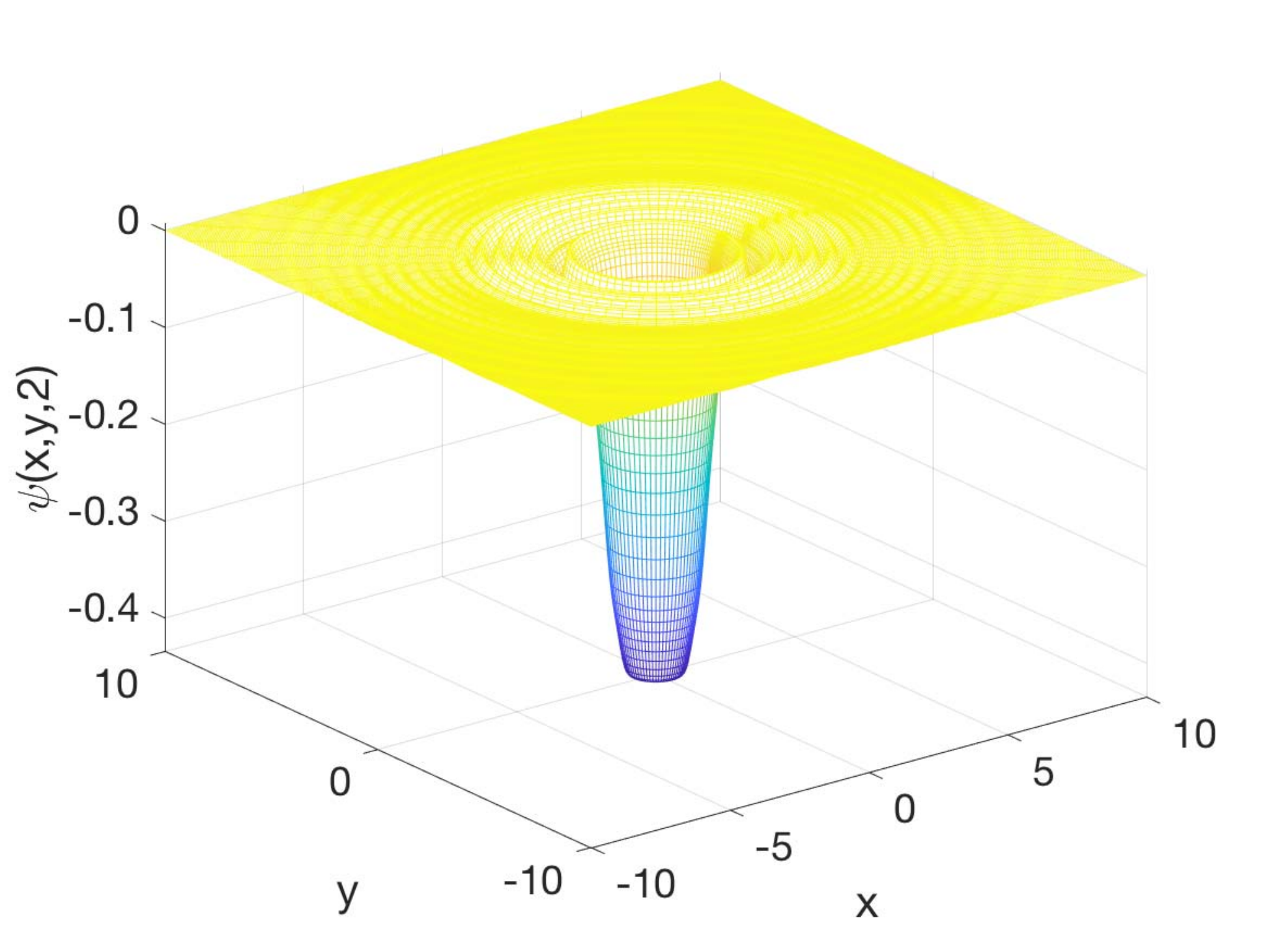}}
\end{minipage}} \hspace{-24pt}
\subfigure[$s=0.7$ and $\mu=-0.3$]{
\begin{minipage}[t]{0.28\textwidth}
\centering
\rotatebox[origin=cc]{-0}{\includegraphics[width=0.9\textwidth]{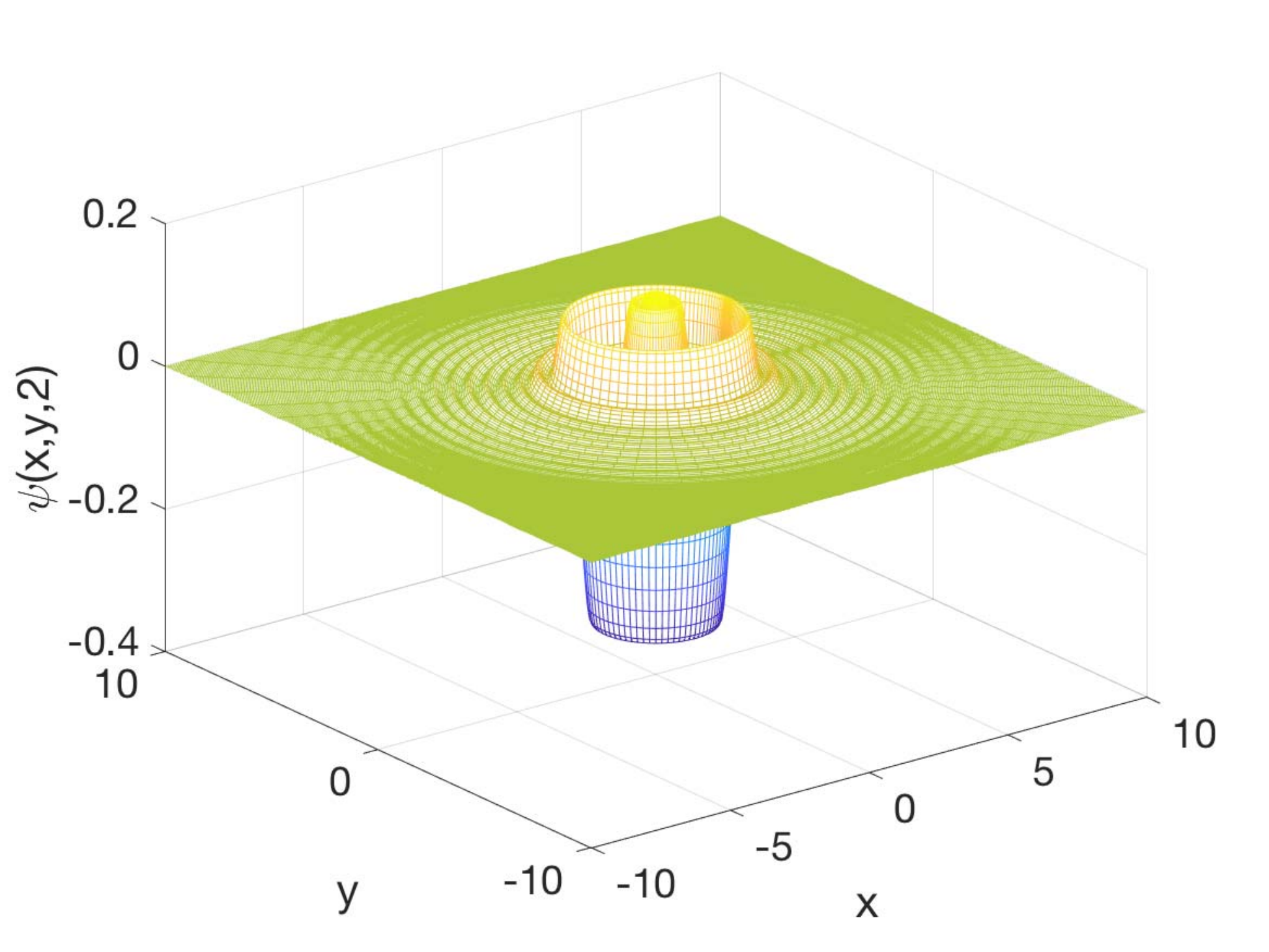}}
\end{minipage}} 

\hspace{-15pt}
\subfigure[$s=0.3$ and $\mu=1$]{
\begin{minipage}[t]{0.28\textwidth}
\centering
\rotatebox[origin=cc]{-0}{\includegraphics[width=0.9\textwidth]{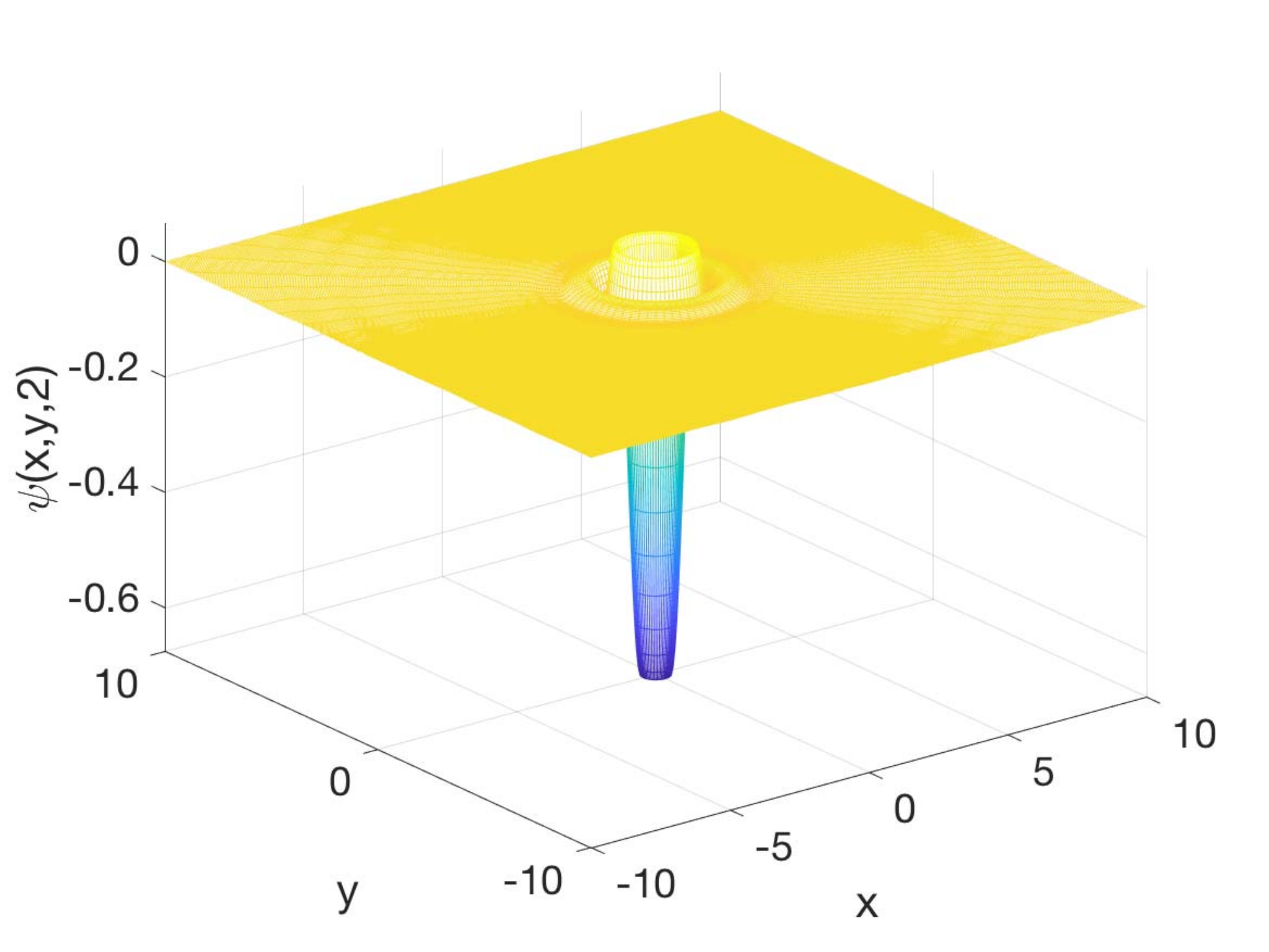}}
\end{minipage}}\hspace{-24pt}
\subfigure[$s=0.3$ and $\mu=0.7$]{
\begin{minipage}[t]{0.28\textwidth}
\centering
\rotatebox[origin=cc]{-0}{\includegraphics[width=0.9\textwidth]{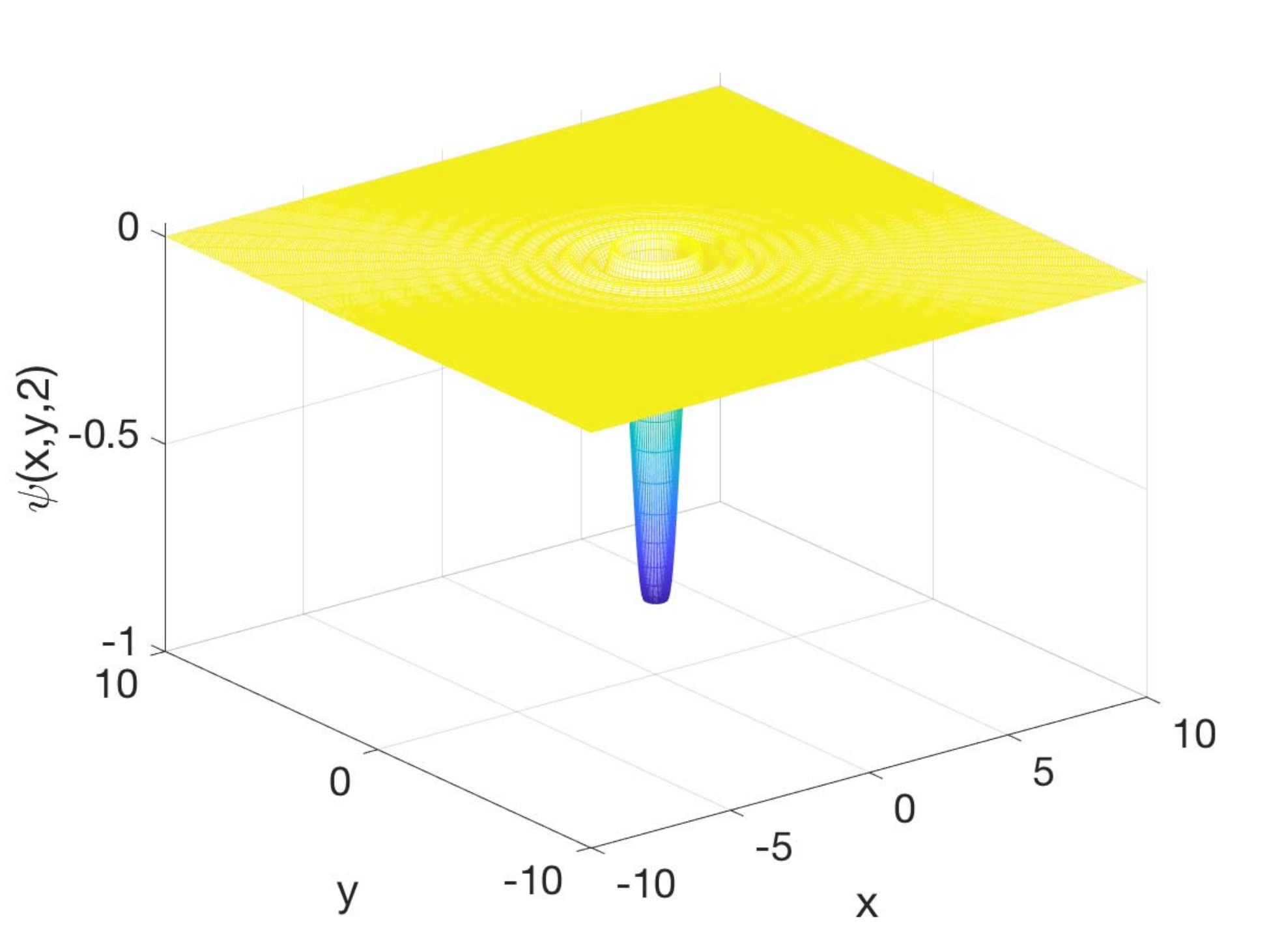}}
\end{minipage}} \hspace{-24pt}
\subfigure[$s=0.3$ and $\mu=0.3$]{
\begin{minipage}[t]{0.28\textwidth}
\centering
\rotatebox[origin=cc]{-0}{\includegraphics[width=0.9\textwidth]{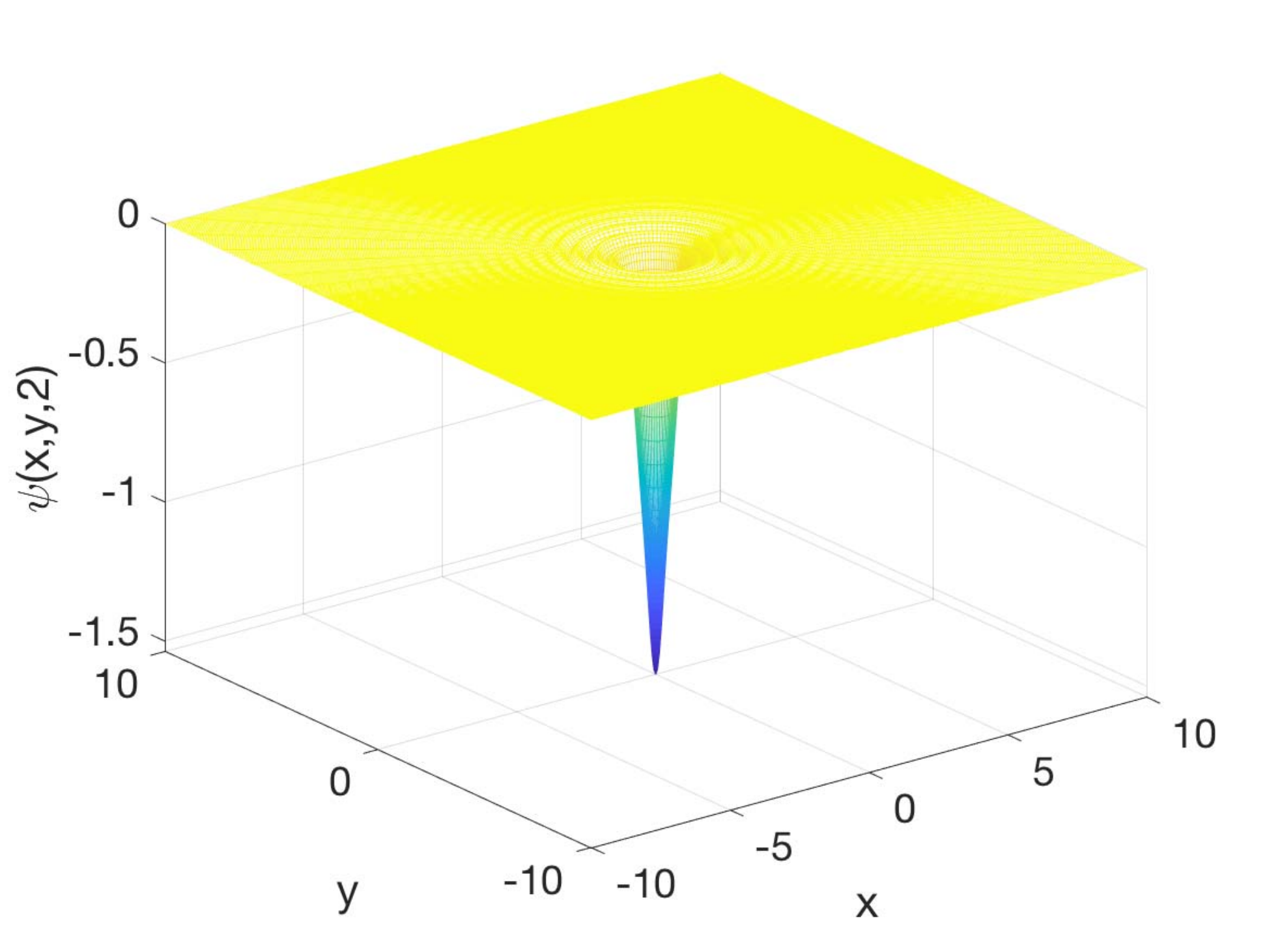}}
\end{minipage}} \hspace{-24pt}
\subfigure[$s=0.3$ and $\mu=-0.3$]{
\begin{minipage}[t]{0.28\textwidth}
\centering
\rotatebox[origin=cc]{-0}{\includegraphics[width=0.9\textwidth]{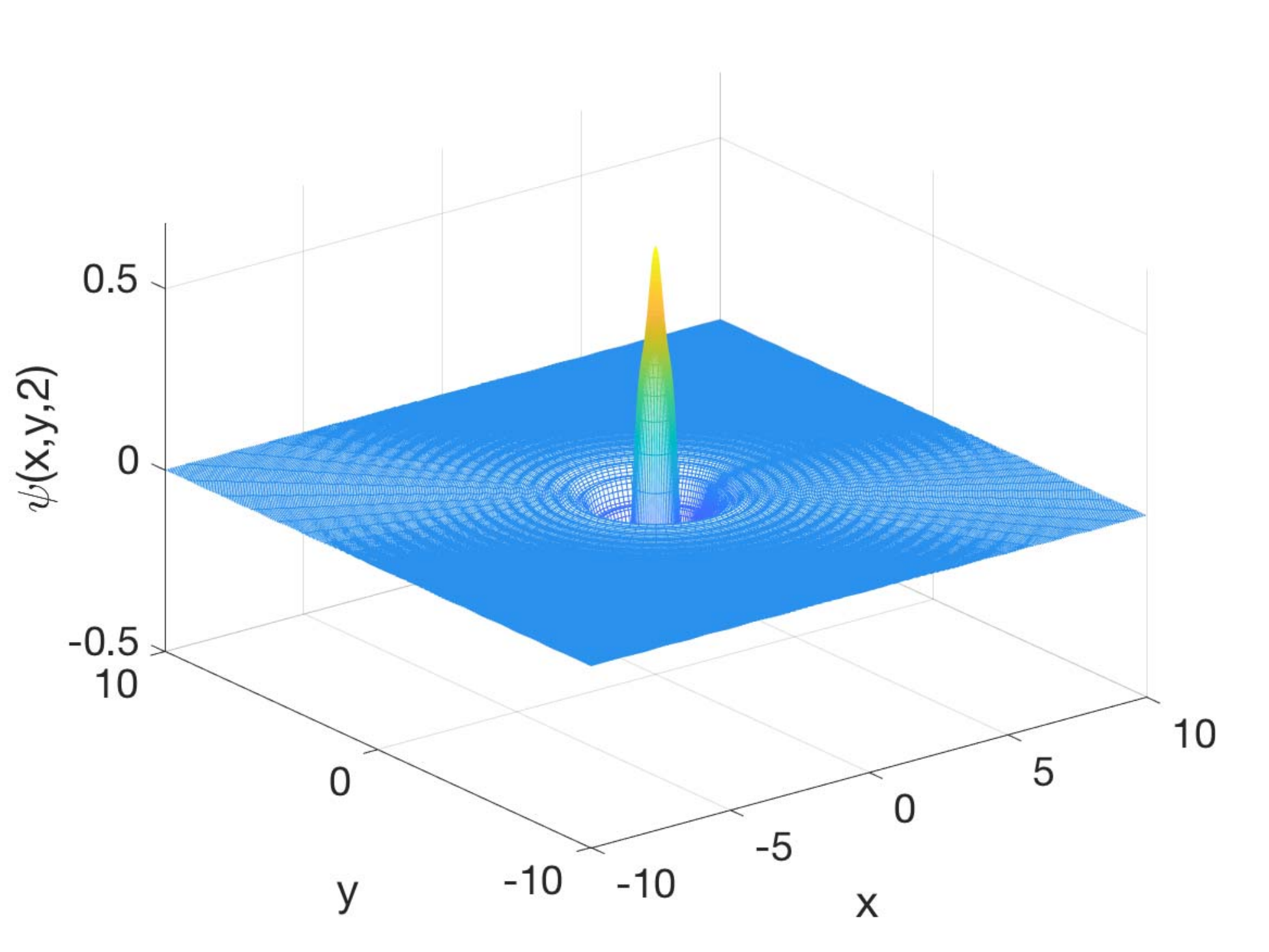}}
\end{minipage}} 
\caption
{\small  The profiles of numerical solution at $t=2$ with different $s,\mu$. }\label{Example3}
\end{figure}

\section{M\"{u}ntz-type GHFs with applications to  Schr\"{o}dinger eigenvalue problems 
}\label{sect6Schro}
\setcounter{equation}{0} \setcounter{lem}{0} \setcounter{thm}{0} \setcounter{cor}{0}\setcounter{remark}{0} 
 In this section, we introduce the second family of generalised Hermite functions for efficient and spectrally accurate solutions of 
 the Schr\"{o}dinger eigenvalue problem: 
 \begin{equation}
\label{fNL}
\begin{cases}
\big[\! -\frac12\Delta+ V(\bx)\big]u(\bx) =\lambda u(\bx) \quad &\text{in}\;\; \RR^d,\\[6pt]
u(\bx)\to 0\quad &\text{as}\;\; |\bx|\rightarrow \infty,
\end{cases}
\end{equation} 
where the potential function $V(\bx)=Z|\bx|^{2\alpha}$ with  $\alpha, Z$ being given constants.  It is known that (i) if $\alpha>-1$, 
 all eigenvalues of \eqref{fNL} are distinct;  (ii) if  $\alpha=-1$ or $Z=0$, the spectrum of the Schr\"odinger  operator $-\frac12 \Delta+\frac{Z}{|\bx|^2}$  
  is a continuous one (cf. \cite{zuazo2001fourier}).

The variational form of \eqref{fNL} is to find $\lambda\in \RR$ and $u\in H^1(\RR^d)\setminus \{0\}$ such that
\begin{equation}\label{weakg} 
 \mathcal{B}(u,v):=\frac12(\nabla u, \nabla v)_{\RR^d}+Z(|\bx|^{2\alpha}u, v)_{\RR^d} =\lambda(u, v)_{\RR^d}, \quad\forall\, v \in H^{1}(\RR^d).
\end{equation}
As shown  in Theorem \ref{rddiffrecur2}, the Hermite functions $\{\widehat{H}_{k, \ell}^{0,n}(\bx)\}$ are the eigenfunctions of the  
Schr\"{o}dinger operator: $-\Delta +|\bs x|^2.$ Here, we intend to explore similar properties for the more general operator
by introducing the  
   M\"untz-type  Hermite functions, and construct efficient and spectrally accurate spectral approximation to \eqref{weakg}.

\subsection{M\"untz-type generalised Hermite functions}  To solve \eqref{weakg} accurately and efficiently,   we introduce  the following  M-GHFs that are orthogonal in the sense of \eqref{orthnew1} below. 
\begin{definition}\label{theorem7.1}
{\em For $\theta> 0, (\ell,n)\in \Upsilon_\infty^d$ and $k\in \mathbb N_0,$ the M\"untz-type GHFs are defined by
\begin{equation}\label{ghfrdx}
\widehat {\mathcal H}^{\theta,n}_{k,\ell}(\bx) =c^{\theta,d}_{k,n}\, L^{(\beta_{n})}_{k} (|\bx|^{2\theta} ) \e^{-\frac{|\bx|^{2\theta}}2} Y^n_{\ell}(\bx), \quad \bx\in \mathbb R^d,
\end{equation}
where 
$$c^{\theta,d}_{k,n}=\sqrt{\frac{2\,k!}{ \Gamma(k+\beta_n+1)}},
\quad \beta_n=\beta^{\theta,d}_{n}=\frac{n+d/2-1}{\theta}.$$}
\end{definition}
 It is seen from \eqref{LagFun01} and \eqref{ghfrdx}  that if $\theta=1,$ it reduces the GHFs 
 $\widehat H^{0,n}_{k,\ell}(\bx)$,  i.e., $\widehat {\mathcal H}^{1,n}_{k,\ell}(\bx)= \widehat H^{0,n}_{k,\ell}(\bx)$.
The so-defined  M\"untz-type GHFs enjoy the following remarkable properties, which are key to the success of  the spectral algorithm for \eqref{weakg}.
\begin{theorem}\label{thmorth} For $\theta>\max(1-d/2,0)$, $(\ell,n),(\iota,m)\in \Upsilon^d_\infty$ and $k,j\in \mathbb{N}_0$, we have 
\begin{equation}\label{orthsch}
\big[\!-\Delta+\theta^2 |\bx|^{4\theta-2}\big]\widehat {\mathcal H}^{\theta,n}_{k,\ell}(\bx)=  2\theta^2 (\beta_{n}+2k+1)\, |\bx|^{2\theta-2}\widehat {\mathcal H}^{\theta,n}_{k,\ell}(\bx),
\end{equation}
and the orthogonality 
\begin{align} \label{orthnew1}
\begin{split}
\big(  \nabla\,   \widehat {\mathcal H}^{\theta,n}_{k,\ell},  \nabla\,  \widehat {\mathcal H}^{\theta,m}_{j,\iota} \big)_{\RR^d}
+\theta^2  \big(|\bx|^{4\theta-2}\widehat {\mathcal H}^{\theta,n}_{k,\ell}, \widehat {\mathcal H}^{\theta,m}_{j,\iota} \big)_{ \RR^d}
 =2\,\theta\,\big(\beta_n+2k+1\big)\delta_{j k} \delta_{mn} \delta_{\ell\iota} .
 \end{split}
 \end{align}
 \end{theorem}
 \begin{proof}

 We can derive from \eqref{eq:Delta}, \eqref{eq:LaplaceBeltrami}, \eqref{ghrdiff2}, \eqref{ghfrdx} and the change of variable $\rho=r^{\theta}$ that 
\begin{align*}
&  \big[-\Delta + \theta^2 r^{4\theta-2} \big] 
   \widehat{\mathcal H}^{\theta,n}_{k,\ell}(\bx) 
    \\&=c^{\theta,d}_{k,n}
   \Big( -\frac{1}{r^{d-1}} \partial_r r^{d-1} \partial_r - \frac{1}{r^2} \Delta_{\mathbb{S}^{d-1}}+  \theta^2 r^{4\theta-2} \Big) \, [r^n\,L^{(\beta_n)}_{k} (r^{2\theta} ) \e^{-\frac{r^{2\theta}}2} Y^n_{\ell}(\hat\bx)]
   \\&=c^{\theta,d}_{k,n}
   \Big( -\frac{1}{r^{d-1}} \partial_r r^{d-1} \partial_r + \frac{n(n+d-2)}{r^2} +  \theta^2 r^{4\theta-2} \Big) \, [r^n\,L^{(\beta_n)}_{k} (r^{2\theta} ) \e^{-\frac{r^{2\theta}}2} Y^n_{\ell}(\hat\bx)]
  \\&=   \theta^2  \rho^{2-\frac{2}{\theta}}   \Big( -    \frac{1}{ \rho^{\frac{d+2\theta-2}{\theta}-1 } } \partial_{\rho}  \rho^{\frac{d+2\theta-2}{\theta}-1 } \partial_{\rho}
  +\frac{\frac{n}{\theta} (\frac{n}{\theta} +\frac{d+2\theta-2}{\theta}-2) }{ \rho^2} +  \rho^{2} \Big) \, [\rho^{\frac{n}\theta}\,L^{(\beta_n)}_{k} (\rho^2 ) \e^{-\frac{\rho^2}2} Y^n_{\ell}(\hat\bx)]
  \\&=  \theta^2  \rho^{2-\frac{2}{\theta}}  \Big(4k + \frac{2n}{\theta} +\frac{d+2\theta-2}{\theta} \Big) \, [\rho^{\frac{n}\theta}\,L^{(\beta_n)}_{k} (\rho^2 ) \e^{-\frac{\rho^2}2} Y^n_{\ell}(\hat\bx)]
  \\&= 2\theta^2 \big(\beta_n+2k+1\big)r^{2\theta-2}[r^n\,L^{(\beta_n)}_{k} (r^{2\theta} ) \e^{-\frac{r^{2\theta}}2} Y^n_{\ell}(\hat\bx)]
  \\&= 2\theta^2 \big(\beta_n +2k+1\big) |\bx|^{2\theta-2} \widehat{\mathcal H}^{\theta,n}_{k,\ell}(\bx),
\end{align*}
where we used the identity  derived from \cite[Lemma 2.1]{2018Ma} with $\alpha=\frac{n+d/2-1}{\theta}$ and $\beta=\alpha+\frac{1-d/2}\theta$:
\begin{equation*}
\begin{split}
 \Big[\partial_{\rho}^2+\frac{\frac{d+2\theta-2}{\theta}-1}{\rho} \partial_{\rho}-\frac{\frac{n}{\theta} (\frac{n}{\theta} +\frac{d+2\theta-2}{\theta}-2) }{ \rho^2}-\rho^2+4k + \frac{2n}{\theta} +\frac{d+2\theta-2}{\theta} \Big] \big[\rho^{\frac{n}\theta} L_{k}^{(\beta_n)}(\rho^2)\,\e^{-\frac{\rho^2}{2}} \big]=0.
\end{split}
\end{equation*}
Next, we prove the orthogonality \refe{orthnew1}. By virtue of \eqref{orthsch}, we have from \eqref{ghfrdx} and the change of variable $\rho=r^{2\theta}$ that
\begin{align*}
\big(  \nabla\,\widehat {\mathcal H}^{\theta,n}_{k,\ell}, & \nabla \,\widehat {\mathcal H}^{\theta,m}_{j,\iota} \big)_{\RR^d}
 +\theta^2 \big( |\bx|^{4\theta-2}\, \widehat {\mathcal H}^{\theta,n}_{k,\ell}, \widehat {\mathcal H}^{\theta,m}_{j,\iota} \big)_{\RR^d}
\\
&= 2\theta^2 (\beta_n+2k+1) (c^{\theta,d}_{k,n})^2\delta_{mn}\delta_{\ell\iota}\! \int_0^{\infty} \!\!
r^{2\theta+2n+d-3} L^{(\beta_n)}_{k}(r^{2\theta}) L^{(\beta_n)}_{j}(r^{2\theta})  \e^{ -r^{2\theta}} \d r
\\
& =  \theta  (\beta_n+2k+1) (c^{\theta,d}_{k,n})^2\delta_{mn}\delta_{\ell\iota} \int_0^{\infty} 
\rho^{\beta_n} L^{(\beta_n)}_{k}(\rho) L^{(\beta_n)}_{j}(\rho)  \e^{ -\rho}  \d \rho
\\
&=2\, \theta (\beta_n+2k+1) \, \delta_{mn}\delta_{jk}\delta_{\ell\iota} .
\end{align*} 
This completes the proof.
 \end{proof}
 
%
%

As a  special case of \eqref{orthsch} (i.e.,  $\theta=\frac12$), we can find the explicit representation of   the eigen-pairs of the Schr\"odinger operator with Coulomb potential: 
$-\frac12\Delta- \frac{|Z|}{|\bx|}$ in $d$ dimension, where $Z$ is a nonzero constant.  
 \begin{corollary} 
 \label{corScheig}
 For any $  k\in \NN_0,\,  (\ell, n)\in \Upsilon_{\infty}^d$ and $Z\neq 0$, we have 
 \begin{equation}\label{orthsch2}
\Big[-\frac12\Delta- \frac{|Z|}{|\bx|}\Big] \widehat {\mathcal H}^{\frac12,n}_{k,\ell}\Big(\frac{4|Z|\, \bx}{2n+2k+d-1}\Big)= -\frac{2 Z ^2}{(2n+2k+d-1)^2} \widehat {\mathcal H}^{\frac12,n}_{k,\ell}\Big(\frac{4|Z| \, \bx}{2n+2k+d-1}\Big).
 \end{equation}
 \end{corollary}
 \begin{proof}   Taking $\theta=\frac12$ in \eqref{orthsch} and rearranging the terms,  leads to 
 \begin{equation*}
\Big[\!-\Delta-\frac{\beta_{n}+2k+1} {2 |\bx|} \Big]\widehat {\mathcal H}^{\frac12,n}_{k,\ell}(\bx)=  -\frac 1 4  \widehat {\mathcal H}^{\frac12,n}_{k,\ell}(\bx).
\end{equation*}
With a rescaling in  $r$ direction 
$$\bx \to \frac{4|Z|\, \bx}{\beta_n+2k+1}= \frac{4|Z|\, \bx}{2n+2k+d-1},$$ 
we can obtain  \eqref{orthsch2} immediately. 
 \end{proof}
 
The identity in  Corollary \ref{corScheig} implies  that the spectra of the Schr\"odinger operator with Coulomb potential are given by 
\begin{equation}\label{eigenpairC}
\big\{\lambda_i,  u^{n}_{i,\ell}\big\} := \bigg\{-\frac{2Z^2}{(2i+d-3)^2},\  \widehat{\mathcal{H}}^{\frac12,n}_{i-n-1,\ell} \Big( \frac{4|Z|\, \bx}{2i+d-3} \Big)\bigg\},
 \quad  (\ell, n)\in \Upsilon_{i-1}^d,\  i\in \NN,
\end{equation}
and the multiplicity of each $\lambda_i$ is 
\begin{align*}
m_i^d:=a_0^d + a_1^d +\dots + a_{i-1}^d =\frac{(i-1)_{d-1}+(i)_{d-1}}{(d-1)!}, \quad d\ge 2,
\end{align*}
where we recall that $a_i^d$ (defined in \eqref{cpCp}) is the cardinality of  $\Upsilon_{i}^d\setminus\Upsilon_{i-1}^d$ (defined in \eqref{UpsD}). 

  \begin{remark}\label{Resear} {\em The spectrum of the  Schr\"odinger operator with Coulomb potential is of much interest in quantum mechanics and mathematical physics.  For example, one can find the 
  spectrum expressions in e.g., \cite[P. 132]{Pauling1935} and \cite[Thm.\! 10.10]{Gerald2014} for $d=3$ with a different derivation,  and
  the recent work  \cite{Osherov2017} for the asymptotic study of the eigenfunctions. \qed}
  \end{remark}

Although the orthogonality \eqref{orthnew1} does not imply the orthogonality of each individual term, the stiffness and mass matrices are sparse with finite bandwidth. 
\begin{theorem}\label{stiffnessmass} For $\theta>\max(1-d/2,0)$, $(\ell,n),(\iota,m)\in \Upsilon^d_\infty$ and $k,j\in \mathbb{N}_0$, we have
\label{matrixSch}
 \begin{align}
  \label{orthnew3}
&\big(\nabla   \widehat {\mathcal H}^{\theta,n}_{k,\ell}, \nabla  \widehat {\mathcal H}^{\theta,m}_{j,\iota} \big)_{\RR^d}= \theta\,\delta_{mn} \delta_{\ell\iota}\times \begin{cases}
 \beta_n +2k+1,
  &j=k,\\
\sqrt{(k+1)\big(\beta_n +k+1\big)}, &j=k+1,\\
 \sqrt{(j+1)\big(\beta_n +j+1\big)}, &k=j+1,\\
  0, &\text{otherwise},
 \end{cases}
\end{align}
and for  $n+d/2+\alpha>0$,
 \begin{align}
 \label{orthnew2}
\begin{split}
& \big(|\bx|^{2\alpha}\widehat {\mathcal H}^{\theta,n}_{k,\ell},\widehat {\mathcal H}^{\theta,m}_{j,\iota} \big)_{\RR^d}=\frac{1}{2\theta}\, c^{\theta,d}_{k,n}\,c^{\theta,d}_{j,n}\,\delta_{mn}\,\delta_{\ell\iota}
\\
&\qquad\qquad \times \sum_{p=0}^{\min(k,j)}\frac{\Gamma(k-p+1-\frac{1+\alpha}{\theta})\Gamma(j-p+1-\frac{1+\alpha}{\theta})\Gamma(p+\beta_n+\frac{1+\alpha}{\theta})}{\Gamma^2(1-\frac{1+\alpha}{\theta})\, (k-p)!\,(j-p)!\,p!}.
\end{split}
\end{align}
\end{theorem}
\begin{proof}  In view of  the definition \eqref{ghfrdx},    we derive from  \eqref{Yorth},  \eqref{Lagpoly},  \eqref{huandi} and  the change of variable $\rho=r^{2\theta}$, we derive
 \begin{equation}\label{AAadd}
 \begin{split}
    \big(|\bx|^{2\alpha}& \widehat {\mathcal H}^{\theta,n}_{k,\ell},\widehat {\mathcal H}^{\theta,m}_{j,\iota}   \big)_{\RR^d}
=c^{\theta,d}_{k,n}c^{\theta,d}_{j,n} \delta_{mn}\delta_{\ell\iota} \int_0^{\infty} 
r^{2n+d-1+2\alpha} L^{(\beta_n)}_{k}(r^{2\theta}) L^{(\beta_n)}_{j}(r^{2\theta}) \, \e^{ -r^{2\theta}} \d r
\\
&=   \frac{1}{2\theta}c^{\theta,d}_{k,n}c^{\theta,d}_{j,n}\delta_{mn}\delta_{\ell\iota} \int_0^{\infty} 
\rho^{\frac{n+d/2-1}{\theta}+\frac{\alpha+1-\theta}{\theta}} L^{(\beta_n)}_{k}(\rho) L^{(\beta_n)}_{j}(\rho) \, \e^{ -\rho}  \d \rho
\\
&=  \frac{1}{2\theta} c^{\theta,d}_{k,n}c^{\theta,d}_{j,n} \delta_{mn}\delta_{\ell\iota} \sum_{p=0}^k\sum_{q=0}^j \frac{\Gamma(k-p+\frac{\theta-1-\alpha}{\theta})}{\Gamma(\frac{\theta-1-\alpha}{\theta})(k-p)!}  \frac{\Gamma(j-q+\frac{\theta-1-\alpha}{\theta})}{\Gamma(\frac{\theta-1-\alpha}{\theta})(j-q)!}
\\
&\quad \quad \times \int_0^{\infty} \rho^{\frac{n+d/2+\alpha-\theta}{\theta}} L^{(\frac{n+d/2+\alpha-\theta}{\theta})}_{p}(\rho) L^{(\frac{n+d/2+\alpha-\theta}{\theta})}_{q}(\rho) \, \e^{ -\rho}  \d \rho
\\
&=\frac{c^{\theta,d}_{k,n}\,c^{\theta,d}_{j,n}}{2\theta} \,\delta_{mn}\,\delta_{\ell\iota}\sum_{p=0}^{\min(k,j)}\frac{\Gamma(k-p+\frac{\theta-1-\alpha}{\theta})\Gamma(j-p+\frac{\theta-1-\alpha}{\theta})\Gamma(p+\frac{n+d/2+\alpha}{\theta})}{\Gamma^2(\frac{\theta-1-\alpha}{\theta})\, (k-p)!\,(j-p)!\,p!},
\end{split}
\end{equation}
which gives \eqref{orthnew2}. In particular, if $\alpha=2\theta-1$,  we  derive from \eqref{AAadd}  that 
 \begin{equation}\label{AAadd2}
 \begin{split}
  \big( |\bx|^{4\theta-2}\widehat{\mathcal H}^{\theta,n}_{k,\ell}, \widehat {\mathcal H}^{\theta,m}_{j,\iota}   \big)_{\RR^d}
  &=\frac{1}{2\theta}c^{\theta,d}_{k,n}c^{\theta,d}_{j,n}\delta_{mn}\delta_{\ell\iota} \int_0^{\infty} 
\rho^{\beta_n+1} L^{(\beta_n)}_{k}(\rho) L^{(\beta_n)}_{j}(\rho) \, \e^{ -\rho}  \d \rho
  \\[4pt]
  &\,=  \frac{1}{\theta}\, \delta_{mn}\delta_{\ell\iota}
 \times \begin{cases} 
     \beta_n +2k+1, &j=k,  \\
     -  \sqrt{(k+1)\big(\beta_n +k+1\big)}, &j=k+1,  \\
     - \sqrt{(j+1)\big(\beta_n +j+1\big)}, &k=j+1, \\
     0, & \text{otherwise}.
   \end{cases}
 \end{split}
\end{equation}
 Then   \refe{orthnew3}   is a direct consequence of \refe{orthnew1} and \eqref{AAadd2}.  Note that \eqref{AAadd2} can be also obtained from \eqref{AAadd} with the understanding $\Gamma(z)=0$ if $z$ is 
 negative integer. 
 \end{proof}

 \subsection{Schr\"odinger eigenvalue problem with 
 a Coulomb potential}\label{coulomb} In what follows, we implement the Hermite spectral  method  for the three-dimensional  Schr\"odinger eigenvalue problem  \eqref{fNL} with  a Coulomb potential $V(\bx)=\frac{Z}{|\bs x|}$ with $Z<0$  for the hydrogen atom \cite{Schrodinger26}, that is, 
\begin{align}
\label{Kohn-Sham}
\Big(-\frac 12 \Delta +\frac{Z}{|\bx |}\Big) u(\bx) = \lambda u(\bx), \quad \bx\in \RR^3.
\end{align}
Numerical solution of \eqref{Kohn-Sham}  poses at least two challenges (i)  nonpositive definiteness of the variational form and (ii)  the singularity of the Coulomb potential.   To overcome these,    
 we shall propose an efficient and accurate  spectral method by using the M\"untz-type GHFs with a suitable parameter $\theta= \frac12$, in light of  the Coulomb potential. 

Define the approximation space
 \begin{align*}
\mathcal{W}_{\!N,K} = \text{span}\big\{\widehat{{\mathcal H}}_{k,\ell}^{\frac 12,n}(\kappa\bx): \; 0\le n\le N, \; 1\le \ell\le 2n+1,~
0\le k\le K,\; k, \ell, n\in {\mathbb N}_0\big\},
 \end{align*}
 where a scaling factor $\kappa>0$ is used to enhance the performance of the spectral approximation
 as in usual Hermite spectral methods in one dimension (see, e.g., \cite{Tang1993,ShenTao2011}).
  The spectral approximation scheme for \eqref{weakg} is to find $\lambda_{\!N,K}\in \RR$ and $u_{\!N,K}\in \mathcal{W}_{\!N,K}\setminus \{0\}$ such that
\begin{equation}\label{specscheme0}
 \mathcal{B}(u_{\!N,K}, v_{\!N,K})=\lambda_{\!N,K}(u_{\!N,K}, v_{\!N,K})_{\RR^3}, \quad \forall\, v_{\!N,K} \in \mathcal{W}_{\!N,K}.
\end{equation}
In real  implementation, we write
\begin{align*}
u_{\!N,K}(\bx)= \sum_{n=0}^{N} \sum_{\ell=1}^{2n+1}\sum_{k=0}^{K} \hat u_{k,\ell}^{n}\,\widehat{{\mathcal H}}_{k,\ell}^{\frac 1 2,n}( \kappa \bx),
\end{align*}
and denote
\begin{equation}\begin{split}
 \hat{\bs u}^n_{\ell} =\big(\hat{u}^n_{0,\ell}, \hat{u}^n_{1,\ell},\dots, \hat{u}^n_{K,\ell} \big)^t,\quad \bs{u}=\big(\hat{\bs u}_{1}^0, \hat{\bs u}_1^1,\hat{\bs u}_2^1,\hat{\bs u}_{3}^1,
 \cdots,\hat{\bs u}_1^N,\hat{\bs u}_2^N,\cdots,\hat{\bs u}_{2N+1}^N\big)^{t}.
\end{split}\end{equation}
With this ordering, we denote  the stiffness and the mass matrices  by $\bs S$  and  $\bs M,$ respectively,  with the entries given by  
\begin{align*}
 \mathcal{B}(\widehat{\mathcal H}^{\frac 1 2,n}_{k,\ell}(\kappa \cdot), \widehat {\mathcal H}^{\frac 1 2,m}_{j,\iota} (\kappa \cdot))
 =&\, \frac1{2\kappa} \Big[ (\nabla \,\widehat {\mathcal H}^{\frac 1 2,n}_{k,\ell}, \nabla \widehat {\mathcal H}^{\frac 1 2,m}_{j,\iota} )_{\RR^3}+\frac{1}{4} (\widehat {\mathcal H}^{\frac 1 2,n}_{k,\ell},  \widehat {\mathcal H}^{\frac 1 2,m}_{j,\iota} )_{\RR^3}\Big]
 \\
  &\, + \frac{Z}{\kappa^{2}} \big( |\bx|^{-1} \widehat {\mathcal H}^{\frac 1 2,n}_{k,\ell},  \widehat {\mathcal H}^{\frac 1 2,m}_{j,\iota} \big)_{\RR^3}
- \frac{1}{8\kappa} \big(\widehat {\mathcal H}^{\frac 1 2,n}_{k,\ell},  \widehat {\mathcal H}^{\frac 1 2,m}_{j,\iota} \big)_{\RR^3}
,
 \\
 \big( \widehat {\mathcal H}^{\frac 1 2,n}_{k,\ell}(\kappa \cdot),  \widehat {\mathcal H}^{\frac 1 2,m}_{j,\iota}(\kappa \cdot) \big)_{\RR^3}=&\,\frac{1}{\kappa^{3}} \big(  \widehat{\mathcal H}^{\frac 1 2,n}_{k,\ell}, \widehat {\mathcal H}^{\frac 1 2,m}_{j,\iota} \big)_{\RR^3}.
\end{align*}
Owing to  \refe{orthnew1} and \eqref{AAadd2} with $\theta=\frac12$, 
 both the stiffness matrix $\bs S$ and the mass matrix $\bs M$ are  
tridiagonal.

Consequently,  the scheme  \eqref{specscheme0} has an equivalent form in the following algebraic eigen-system:
\begin{equation}\label{eq:AlgS1}
 \bs S\bs{u} = \lambda_{\!N}\bs{M}\bs{u}.
\end{equation}
Interestingly,  the matrix $\bs S +\frac{\kappa^2}{8} \bs M$  is  diagonal, so we can rewrite \eqref{eq:AlgS1}  as 
\begin{align*}
\Big(\bs S +\frac{\kappa^2}{8} \bs M\Big) \bs u= \Big(\lambda_N+\frac{\kappa^2}{8}\Big) \bs M\bs u,
\end{align*}
which leads a more efficient  implementation.

  \begin{figure}[htb!]
\subfigure[$N=16$ and $\kappa=4$.]{
\begin{minipage}[t]{0.42\textwidth}
\centering
\rotatebox[origin=cc]{-0}{\includegraphics[width=0.85\textwidth]{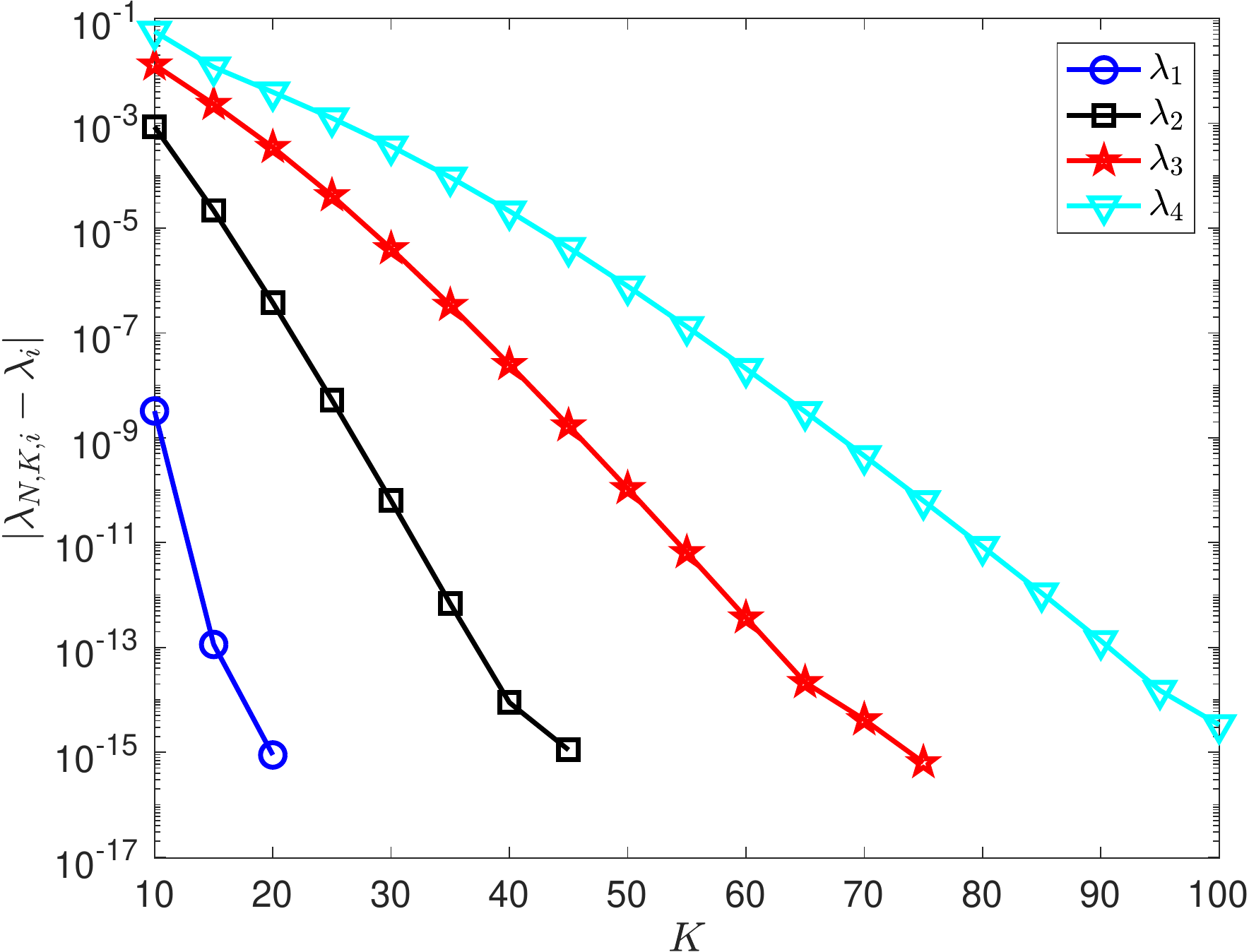}}
\end{minipage}}
\subfigure[$N=16$ and $\kappa=7/4$.]{
\begin{minipage}[t]{0.42\textwidth}
\centering
\rotatebox[origin=cc]{-0}{\includegraphics[width=0.85\textwidth]{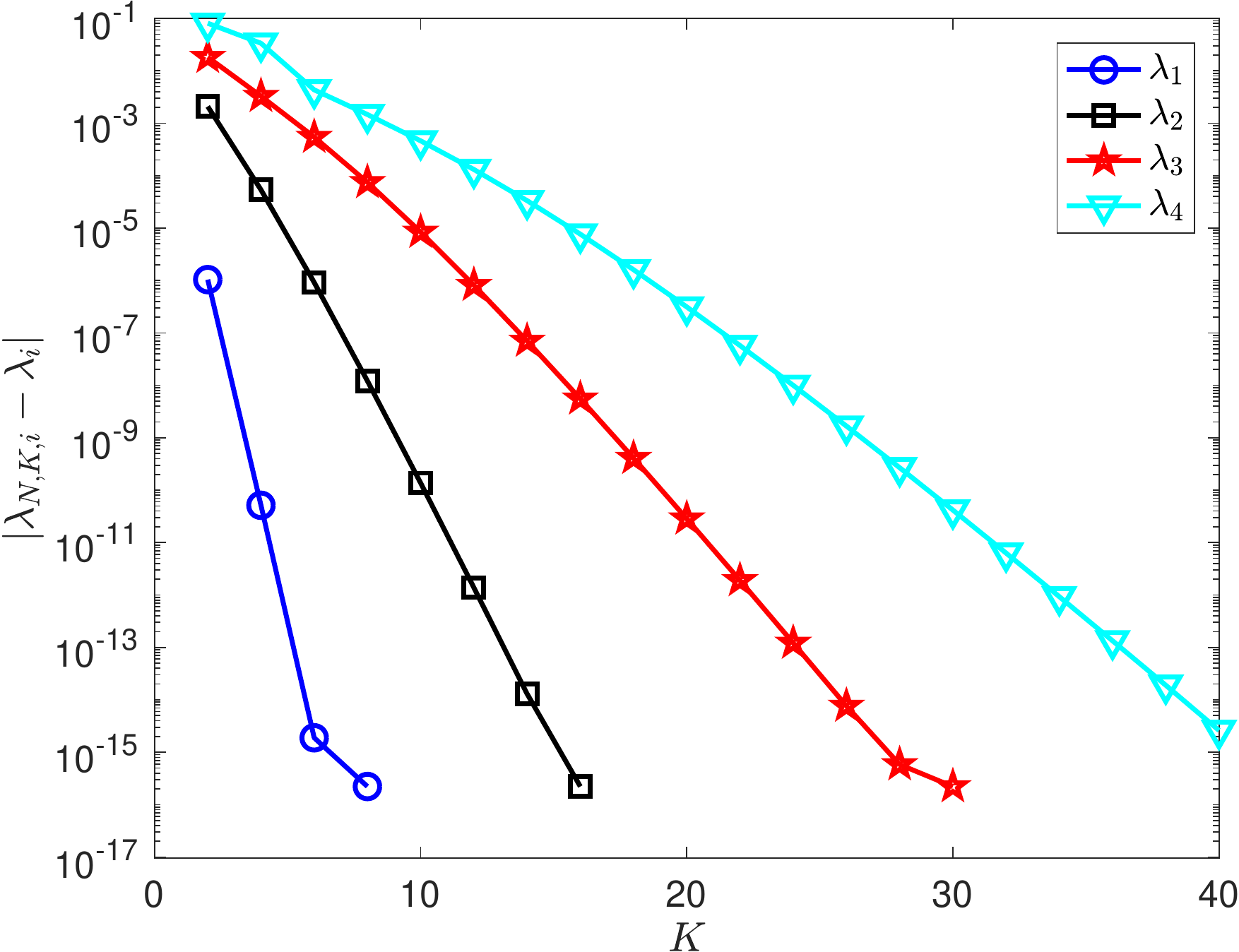}}
\end{minipage}}\vskip -5pt
%
\caption
{\small  The errors of  the smallest $4$ eigenvalues without counting multiplicities versus $K$ for solving \eqref{Kohn-Sham}  with  $Z=-1$.}\label{figKS1}
\end{figure}

In Figure  \ref{figKS1}, we plot the errors between the first $30$ (counted by multiplicity)   smallest numerical eigenvalues and exact eigenvalues in \eqref{eigenpairC} versus $K$ for fixed $N=16$ and two different scaling factors (so that the error of the truncation in angular directions  is negligible).  Observe that the errors decay exponentially  in terms of the cut-off number in the radial direction, along which the eigenfunctions are singular.   We also see that the scaling parameter  affects the convergence rate as the usual Hermite method  (cf.\! \cite{Tang1993}). 

%
%

\subsection{Schr\"odinger eigenvalue problem with 
a fractional power potential} \label{coulomb2} Note that 
for any given rational number $\frac{q}{p}>-2$ with $p\in \NN$ and $q\in \ZZ$, we can always rewrite it as 
\begin{equation}\label{munucond}
\frac{q}{p}=\frac{2\nu-2\mu}{\mu+1}  \quad {\rm with}\;\;\; \mu=2p-1\in \NN,\;\; \nu=2p+q-1\in \NN_0.
\end{equation}
In the sequel, we consider the following Schr\"odinger equation with a fractional power potential as follows
\begin{align}
\label{eq:Sch2}
-\frac 12 \Delta u(\bx) +Z |\bx |^{\frac{2\nu-2\mu}{\mu+1}}\,u(\bx) = \lambda u(\bx), \qquad \bx\in \RR^d,
\end{align}
where $\mu,\nu\in \NN_0.$  Hereafter,   we choose the M\"untz-type GHF approximation with 
$\theta=\frac1{\mu+1},$ to account for both the accuracy and efficiency.  Accordingly, 
 we define the approximation space
 \begin{align*}
\mathcal{W}_{\!N,K}^{d,\frac1{\mu+1}} = \text{span}\big\{\widehat{{\mathcal H}}_{k,\ell}^{\frac1{\mu+1},n}(\kappa\bx): \; 0\le n\le N, \; 1\le \ell\le a_n^d,~
0\le k\le K,\; k, \ell, n\in {\mathbb N}_0\big\}, \quad d\ge 2, 
 \end{align*}
 and for $d=1$, we can always assume $\mu$ is odd and then define the approximation space as
  \begin{align*}
\mathcal{W}_{\!N,K}^{1,\frac1{\mu+1}} = \text{span}\big\{\widehat{{\mathcal H}}_{k,1}^{\frac1{\mu+1},n}(\kappa\bx): \;   \frac{\mu+1}{2}\delta_{n,0}  \le k\le K,  \ n=0,1 \big\},
 \end{align*}
 where $\{\widehat{{\mathcal H}}_{k,1}^{\frac1{\mu+1},n}\}$  are understood as  the M\"untz-type GHFs 
 defined through  generalized  Laguerre polynomials  $L^{(\beta_0)}_k$ with the negative integer $\beta_0=-\frac{\mu+1}{2}$ (cf. \!\cite{LWL17}). This turns out  important to deal with 
 the  strong singularities at the origin  to ensure $u(0)=0$ in one dimension. 

 The generalized Hermite spectral method for \eqref{weakg} is to find $\lambda_{\!N,K}\in \RR$ and $u_{\!N,K}\in \mathcal{W}^{d,\frac1{\mu+1}}_{\!N,K}\setminus \{0\}$ such that
\begin{equation}\label{specscheme}
 \mathcal{B}(u_{\!N,K}, v_{\!N,K})=\lambda_{\!N,K}(u_{\!N,K}, v_{\!N,K})_{\RR^d}, \quad \forall\,  v_{\!N,K} \in \mathcal{W}^{d,\frac1{\mu+1}}_{\!N,K}.
\end{equation}
In the implementation,  we write
\begin{align*}
u_{\!N,K}(\bx)= \sum_{n=0}^{N} \sum_{\ell=1}^{a_n^d}\sum_{k=0}^{K} \hat u_{k,\ell}^{n}\,\widehat{{\mathcal H}}_{k,\ell}^{\frac{1}{\mu+1},n}( \kappa \bx),
\end{align*}
and denote
\begin{equation}\begin{split}
 \hat{\bs u}^n_{\ell} =\big(\hat{u}^n_{0,\ell}, \hat{u}^n_{1,\ell},\dots, \hat{u}^n_{K,\ell} \big)^t,\quad \bs{u}=\big(\hat{\bs u}_1^0,\hat{\bs u}_2^0,\cdots,\hat{\bs u}_{a_0^d}^0,\hat{\bs u}_1^1,\hat{\bs u}_2^1,\cdots,\hat{\bs u}_{a_1^d}^1,
 \cdots,\hat{\bs u}_1^N,\hat{\bs u}_2^N,\cdots,\hat{\bs u}_{a_N^d}^N\big)^{t}.
\end{split}\end{equation}
The corresponding algebraic eigen-system  of   \eqref{specscheme} is 
\begin{equation}\label{eq:AlgS2}
 \bs S\bs{u} = \lambda_{\!N}\bs{M}\bs{u}.
\end{equation}
In view of orthogonality \eqref{orthnew3} and  \eqref{orthnew2}, we find that  for any $q\in \NN_0$,
 \begin{align*}
& \big(|\bs x |^{\frac{2q-2\mu}{\mu+1}} \widehat {\mathcal H}^{\frac1{\mu+1},n}_{k,\ell}(\kappa \cdot),  \widehat {\mathcal H}^{\frac1{\mu+1},m}_{j,\iota}(\kappa \cdot) \big)_{\RR^d}
=\kappa^{-d} \big(  \widehat{\mathcal H}^{\frac1{\mu+1},n}_{k,\ell}, \widehat {\mathcal H}^{\frac1{\mu+1},m}_{j,\iota} \big)_{\RR^d}
\\
& =\frac{\mu+1}2  \kappa^{-d}\, c^{\frac1{\mu+1},d}_{k,n}\,c^{\frac1{\mu+1},d}_{j,n} \,\delta_{mn}\,\delta_{\ell\iota}\sum_{p=\max(j-q,k-q,0)}^{\min(k,j)}\frac{\Gamma(k-p-q)\Gamma(j-p-q)\Gamma(p+\beta_n+ q+1)}{\Gamma^2(-q)\, (k-p)!\,(j-p)!\,p!}
\\
& =\frac{\mu+1}2  \kappa^{-d}\, c^{\frac1{\mu+1},d}_{k,n}\,c^{\frac1{\mu+1},d}_{j,n} \,\delta_{mn}\,\delta_{\ell\iota}\sum_{p=\max(j-q,k-q,0)}^{\min(k,j)}\frac{(-q)_{k-p}(-q)_{j-p} \Gamma(p+\beta_n+ q+1)}{ (k-p)!\,(j-p)!\,p!}.
\end{align*}
Furthermore, one has 
\begin{align*}
 \mathcal{B}(\widehat{\mathcal H}^{\frac1{\mu+1},n}_{k,\ell}(\kappa \cdot), \widehat {\mathcal H}^{\frac1{\mu+1},m}_{j,\iota} (\kappa \cdot))
 & = \frac12\, \kappa^{2-d} (\nabla \,\widehat {\mathcal H}^{\frac1{\mu+1},n}_{k,\ell}, \nabla \widehat {\mathcal H}^{\frac1{\mu+1},m}_{j,\iota} )_{\RR^d}
 \\ &\quad + Z\, \kappa^{-\frac{2\nu-2\mu}{\mu+1}-d} ( |\bx|^{\frac{2\nu-2\mu}{\mu+1}} \widehat {\mathcal H}^{\frac1{\mu+1},n}_{k,\ell},  \widehat {\mathcal H}^{\frac1{\mu+1},m}_{j,\iota} )_{\RR^d} .
\end{align*} 
These indicate that 
the stiffness matrix $\bs S$ is a  sparse banded  matrix  with a   bandwidth $\max(\nu,1)$,
and   the   mass matrix $\bs M$ is also a sparse banded matrix with a  bandwidth $\mu$.

 \begin{figure}[!h]
\subfigure[$d=4$, $Z=1$,  $\mu=3$, $\nu=5$ and $\kappa=500$.]{
\begin{minipage}[t]{0.42\textwidth}
\centering
\rotatebox[origin=cc]{-0}{\includegraphics[width=0.85\textwidth]{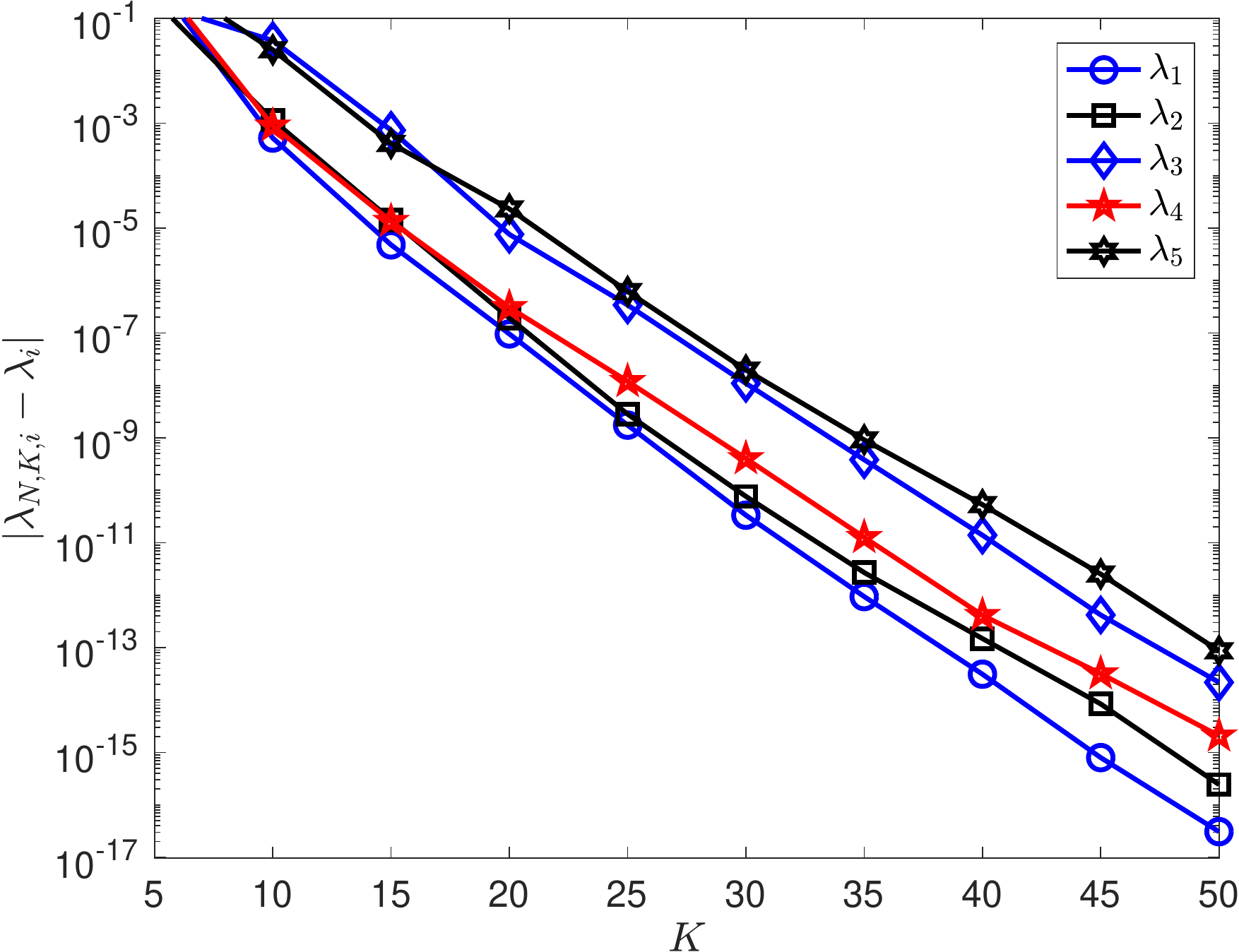}}
\end{minipage}}
\subfigure[$d=3$, $Z=1$, $\mu=1$, $\nu=2$ and $\kappa=2$.]{
\begin{minipage}[t]{0.42\textwidth}
\centering
\rotatebox[origin=cc]{-0}{\includegraphics[width=0.85\textwidth]{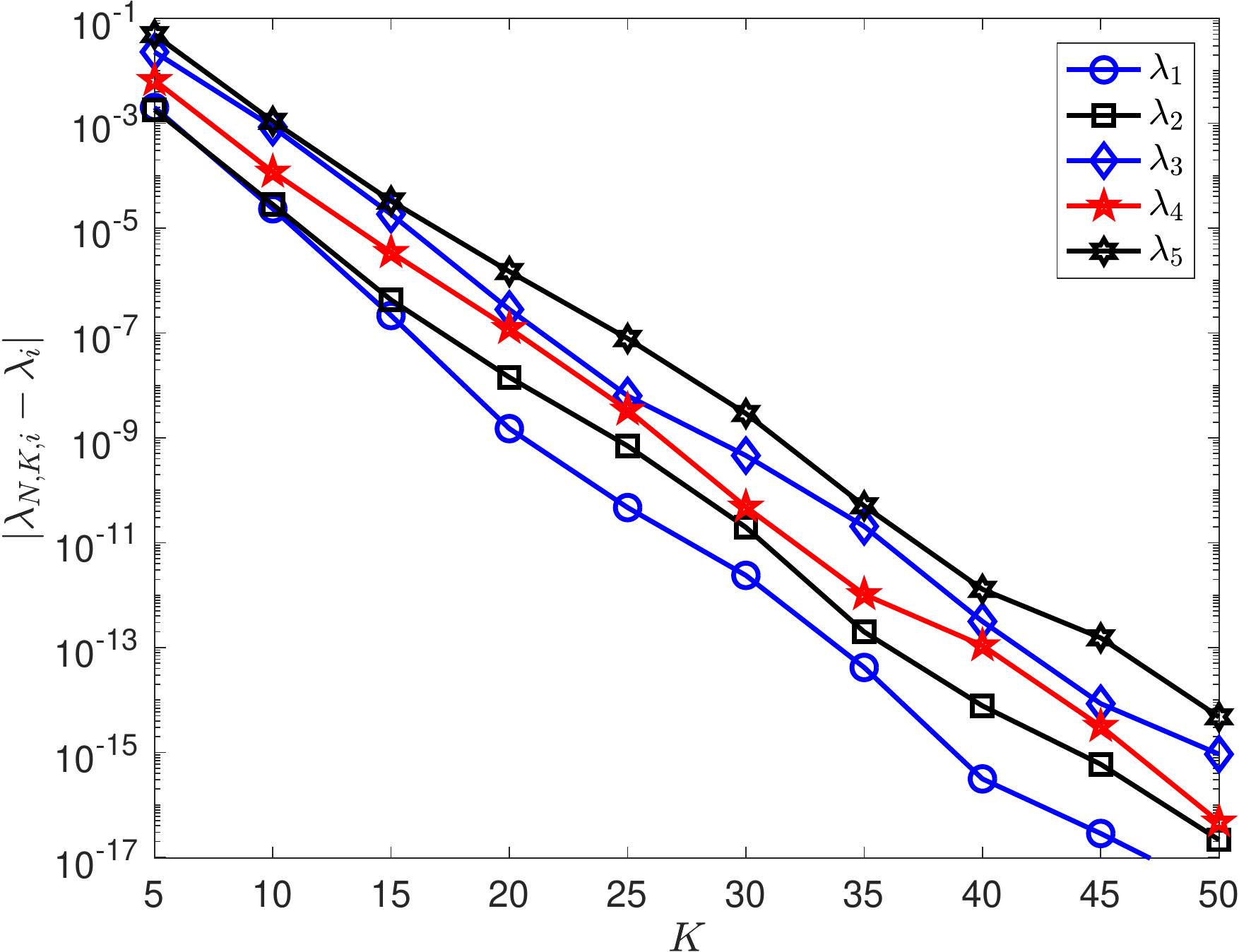}}
\end{minipage}}
\vskip -5pt

\subfigure[$d=2$, $Z=3$,  $\mu=1$, $\nu=4$ and $\kappa=10$.]{
\begin{minipage}[t]{0.42\textwidth}
\centering
\rotatebox[origin=cc]{-0}{\includegraphics[width=0.85\textwidth]{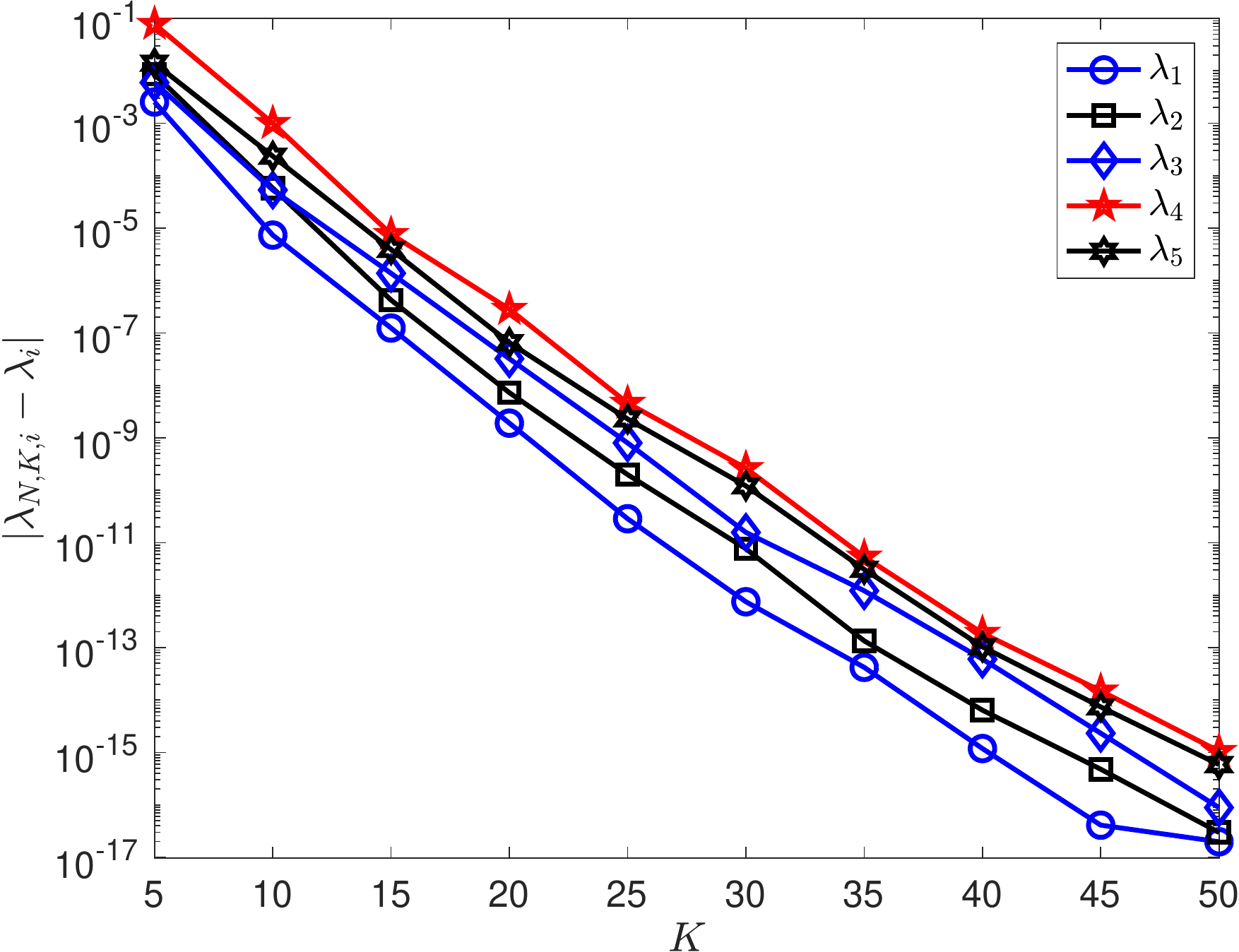}}
\end{minipage}}
\subfigure[$d=1$, $Z=-3$, $\mu=3$, $\nu=2$ and $\kappa=70$.]{
\begin{minipage}[t]{0.42\textwidth}
\centering
\rotatebox[origin=cc]{-0}{\includegraphics[width=0.85\textwidth]{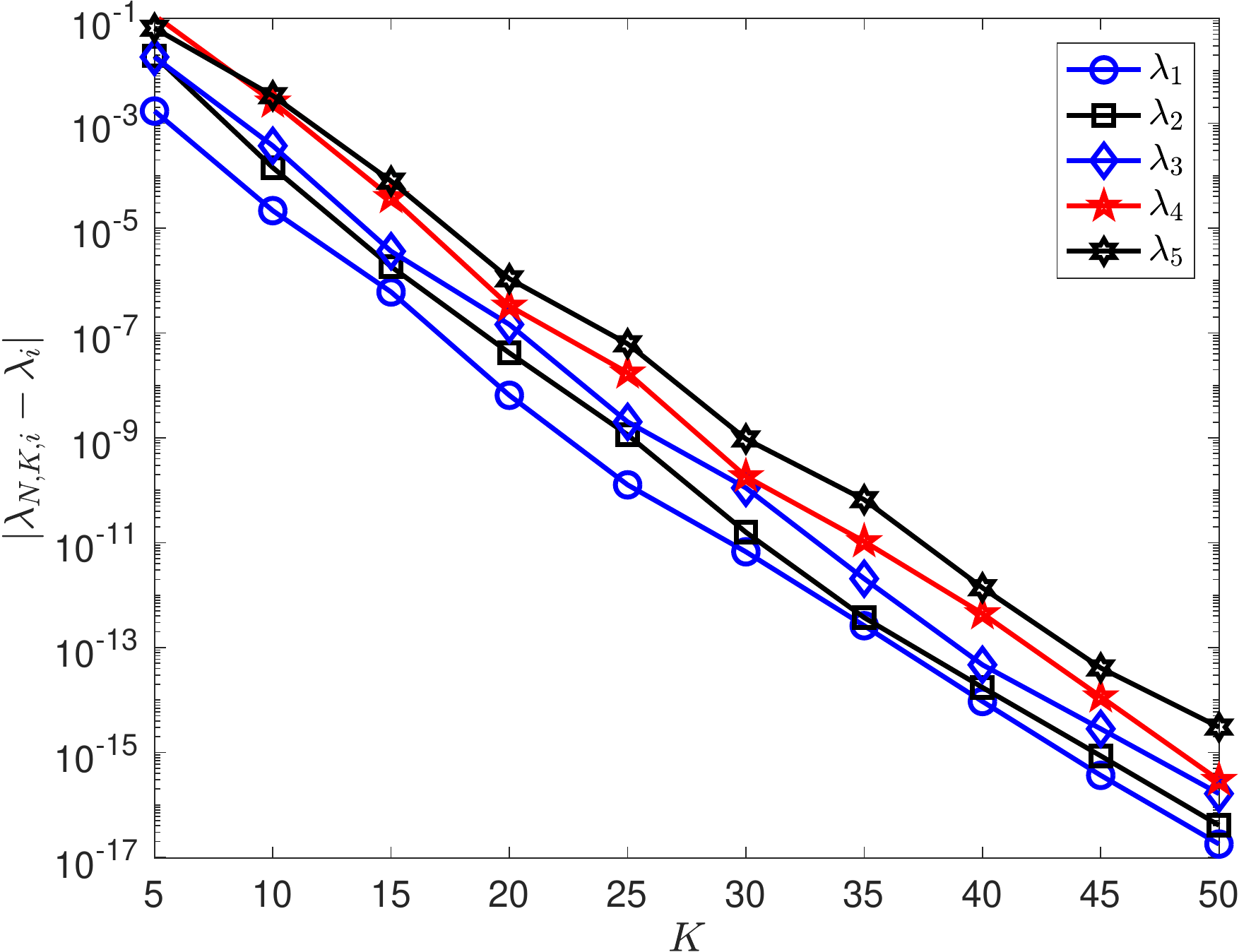}}
\end{minipage}}
\vskip -5pt
\caption
{\small  The errors of  the smallest $5$ eigenvalues without counting multiplicities versus $K$ for solving \eqref{eq:Sch2} with $N=10$.}\label{figKS2}
\end{figure}

In the numerical tests, we fix $N=10$, choose different scaling factor $\kappa$ and test for different $Z$, $\mu, \nu$ and  dimensions. Numerical errors between the smallest eigenvalues  without counting multiplicities and the reference eigenvalues (obtained by the scheme with large $N$ and $K$)
are depicted  in Figure \ref{figKS2}. Exponential orders of convergence are clearly observed in all cases, which demonstrate the effectiveness of the new Hermite spectral method.

\medskip 

\noindent{\bf Acknowledgement}:  The first author would like to thank  Beijing Computational Science Research Center  for hosting his visit devoted to this collaborative work. The fourth author is grateful to Professor Jie Shen at Purdue University for valuable suggestion.

\begin{appendix}
\setcounter{equation}{0}  \setcounter{thm}{0}  \setcounter{pro}{0}  \setcounter{lem}{0} 
\renewcommand{\theequation}{A.\arabic{equation}}
\renewcommand{\thethm}{A.\arabic{thm}}
\renewcommand{\thepro}{A.\arabic{pro}}\renewcommand{\thelem}{A.\arabic{lem}}

 \renewcommand{\theequation}{A.\arabic{equation}}
 \section{The proof of Theorem \ref{lemghp}}\label{appb}


We first recall the orthogonality  (cf.\! \cite[(11.6)]{chihara1955}) 
  \begin{equation}\label{GHermpolyorth}
  \dint_{\RR} H_m^{(\mu)}(x)H_n^{(\mu)}(x) |x|^{2\mu}\,\e^{-x^2}\,\d x=\gamma_n^{(\mu)} \delta_{mn},\quad 
  \gamma_n^{(\mu)}=2^{2n}\,\Big[\frac n 2 \Big]!\, \Gamma\Big(\Big[\frac {n+1} 2 \Big]+\mu+\frac 1 2\Big).
 \end{equation}
According to \cite[P. 42]{chihara1955},  we have 
\begin{equation}\label{deri_GHP}
\partial_x H_{n}^{(\mu)}(x)=2nH_{n-1}^{(\mu)}(x)+2(n-1)\theta_n \, x^{-1}H_{n-2}^{(\mu)}(x),\quad n\ge 1,
\end{equation}
where $\theta_{2k}=0$ and $\theta_{2k+1}=2\mu$  as in \eqref{problem25}.
In particular, for $\mu>-\frac 1 2, $
\begin{equation}\label{deri_GHP2}
\partial_x H^{(\mu)}_{2k}(x)=4kH^{(\mu)}_{2k-1}(x),\quad k\ge 1.
\end{equation}

We first show the modified derivative (cf. \eqref{higherOrder}) formula:   for $k\ge m,$
\begin{equation}\label{evenk}
D_x^m H^{(\mu)}_{2k}(x) =d_k^{(m)} H^{(\mu+m-1)}_{2k-2m+1}(x), 
\quad 
D_x^m H^{(\mu)}_{2k+1}(x) =d_k^{(m)} H^{(\mu+m)}_{2k-2m+1}(x),  \;\;\; d_k^{(m)}=\frac{4^m k!}{(k-m)!}.
\end{equation}
For this purpose, we recall the recurrence relation (cf. \cite[P. 609]{SCF1964}):
\begin{equation}\label{H2ksk}
 2xH^{(\mu+1)}_{2k}(x)=H^{(\mu)}_{2k+1}(x), \quad k\ge 0,
\end{equation}
which,  together with \eqref{deri_GHP2},  implies 
\begin{equation}\label{Henew}
D_x H^{(\mu)}_{2k+1}(x)=\partial_x \Big\{\frac{1}{2x}H^{(\mu)}_{2k+1}(x)\Big\}=\partial_x H^{(\mu+1)}_{2k}(x)=4kH^{(\mu+1)}_{2k-1}(x). 
\end{equation}
Thus,  we obtain from  \eqref{higherOrder} that 
\begin{equation}\label{Henew2}
D_x^2 H^{(\mu)}_{2k+1}(x)=4kD_x H^{(\mu+1)}_{2k-1}(x)=4^2k(k-1) H^{(\mu+2)}_{2k-3}(x). 
\end{equation}
Using this relation repeatedly  yields the second identity in \eqref{evenk}. 
We now consider the first  identity. For $m=1,$ it coincides with  \eqref{deri_GHP2}, so by 
\eqref{Henew}, 
\begin{equation}\label{deri_GHP20}
D_x^2 H^{(\mu)}_{2k}(x)=4k D_x H^{(\mu)}_{2k-1}(x)=4k \partial_x \Big\{\frac{1}{2x}H^{(\mu)}_{2k-1}(x)\Big\}
=4^2k(k-1) H^{(\mu+1)}_{2k-3}(x),
\end{equation}
which leads to the first identity by taking higher modified derivatives and the second identity in  \eqref{evenk}. 
For the orthogonal projection defined in \eqref{ll2orth}, we can write
 \begin{equation}\label{1dproj}
\Pi^{(\mu)}_{N} u(x)=\sum_{n=0}^{N} \tilde u_{n} H^{(\mu)}_{n}(x)=\sum_{k=0}^{[\frac{N}2]} \tilde u_{2k} H^{(\mu)}_{2k}(x)+\sum_{k=0}^{[\frac{N}2]} \tilde u_{2k+1} H^{(\mu)}_{2k+1}(x),
\end{equation}
with  
 \begin{equation*}
 \tilde u_{n}=\frac{1}{\gamma^{(\mu)}_{n}} \int_{\RR} u(x) H^{(\mu)}_{n}(x) \chi^{(\mu)}(x)\, \d x.
\end{equation*} 
We only need to prove show the result with $m\geq 1$, as $m=0$ is obvious. For simplicity, we first assume that $N$ is odd. 
 It is clear that by \refe{1dproj}, 
 \begin{equation} \label{1dproof1}
\|\Pi^{(\mu)}_{N}u-u\|^2_{\chi^{(\mu)}}=\|\Pi^{(\mu)}_{N}u_{\rm e}-u_{\rm e}\|^2_{\chi^{(\mu)}}+
\|\Pi^{(\mu)}_{N}u_{\rm o}-u_{\rm o}\|^2_{\chi^{(\mu)}},
\end{equation}	
where we decompose  $u(x)$ into even and odd parts as $u_{\rm e}(x)$ and $u_{\rm o}(x)$.
We now deal with the first term.  By \eqref{GHermpolyorth} and \eqref{evenk}, we have the orthogonality 
\begin{equation}\label{derim10}
\displaystyle \int_{\RR} D_{x}^{m} H_{2k}^{(\mu)}(x)\, D_{x}^{m} H_{2l}^{(\mu)}(x) \chi^{(\mu+m-1)}(x)\, \d x=h_{2k, m}^{(\mu)}\, \delta_{k, l},
\end{equation}
where for $k\geq m$, 
\begin{equation}\label{hnk1} 
\begin{split} 
h_{2k,m}^{(\mu)}= (d_k^{(m)})^2\gamma_{2k-2m+1}^{(\mu+m-1)}=\frac{2^{4k+2}\, \Gamma(k+\mu+\frac12) (k!)^2}{(k-m)!}. 
 \end{split}
\end{equation}
Thus, by the Parseval's identity, we have 
\begin{equation*}\label{dxmu}
\|D_{x}^{m} u_{\rm e}\|_{\chi^{(\mu+m-1)}}^{2}=\sum_{k=m}^{\infty} h_{2k,m}^{(\mu)}|\tilde u_{2k}|^{2}. 
\end{equation*}
In view of  \refe{derim10}, we obtain from  \refe{hnk1} that  for $m\ge 1,$
\begin{equation}\label{uests}
\begin{split}
\big\|\Pi_{N}^{(\mu)} u_{\rm e}-u_{\rm e}\big\|_{\chi^{(\mu)}}^{2}&=\sum_{k=\frac{N+1}2}^{\infty} \gamma_{2k}^{(\mu)}|\tilde u_{2k}|^{2} \leq \max _{k \geq \frac{N+1}2}\Big\{\frac{\gamma_{2k}^{(\mu)}}{h_{2k,m}^{(\mu)}}\Big\} \sum_{k=\frac{N+1}2}^{\infty} h_{2k,m}^{(\mu)}|\tilde u_{2k}|^{2}
 \\&\leq \frac{\gamma_{N+1}^{(\mu)}}{h_{N+1,m}^{(\mu)}}\big\|D_{x}^mu_{\rm e}\big\|_{\chi^{(\mu+m-1)}}^{2}
\leq \frac{(\frac{N+1}2-m)!}{2^2(\frac{N+1}2)!}\big\|D_{x}^mu_{\rm e}\big\|_{\chi^{(\mu+m-1)}}^{2}.
\end{split}\end{equation}
Similarly, by \eqref{evenk}  and \eqref{GHermpolyorth}, we have the orthogonality 
\begin{equation}\label{derim10o}
\displaystyle \int_{\RR} D_{x}^{m} H_{2k+1}^{(\mu)}(x)\, D_{x}^{m} H_{2l+1}^{(\mu)}(x) \chi^{(\mu+m)}(x)\, \d x=h_{2k+1, m}^{(\mu)}\, \delta_{k, l},
\end{equation}
where for $k\geq m$, 
\begin{equation*}\label{hnk1o} 
\begin{split} 
h_{2k+1,m}^{(\mu)}= (d_k^{(m)})^2\gamma_{2k-2m+1}^{(\mu+m)}=\frac{2^{4k+2}\, \Gamma(k+\mu+\frac32) (k!)^2}{(k-m)!}. 
 \end{split}
\end{equation*}
Then, following the same lines as above, we can show 
\begin{equation}\label{newPiN}
\begin{split}
\big\|\Pi_{N}^{(\mu)} u_{\rm o}-u_{\rm o}\big\|_{\chi^{(\mu)}}^{2}&
 \leq \frac{\gamma_{N+2}^{(\mu)}}{h_{N+2,m}^{(\mu)}}\|D_{x}^mu_{\rm o}\|_{\chi^{(\mu+m)}}^{2}
 \leq \frac{(\frac{N+1}2-m)!}{(\frac{N+1}2)!}\|D_{x}^mu_{\rm o}\|_{\chi^{(\mu+m)}}^{2}.
\end{split}\end{equation}
Thus, a combination of \eqref{1dproof1},  \eqref{uests} and \eqref{newPiN} leads to the estimate \eqref{resultghp} with odd $N.$ 
For even $N,$ we can obtain the same estimate but with $N/2$ in place of $(N+1)/2$ in the upper bound. 

Now, we turn to the proof of  \eqref{errest0}.  If $u e^{\frac {x^2} 2} \in {{\mathcal B}}_{\mu}^{m}(\RR),$  we find from \eqref{proj1d} that  
\begin{equation}\label{newrelation}
\big\|u-\widehat{\Pi}^{(\mu)}_{N} u\big\|_{\omega^{(\mu)}}= \big\| u\e^{\frac{x^2}{2}} -{\Pi}^{(\mu)}_{N} (u \e^{\frac{x^2}{2}})\big\|_{\chi^{(\mu)}}.
\end{equation}
Then the estimate  \eqref{errest0} is a direct consequence of   \eqref{resultghp}.

\end{appendix}

\end{document}